\documentclass[a4paper,10pt,reqno]{amsart}
\usepackage[T1]{fontenc}
\usepackage{amsmath}
\usepackage{amssymb}
\usepackage{amsthm}
\usepackage{mathrsfs}
\usepackage{geometry}
\usepackage{bm}
\usepackage{enumitem}
\usepackage{mathtools}
\usepackage[colorlinks=true, linkcolor=blue, citecolor=blue, urlcolor=blue]{hyperref}
\usepackage{natbib}
\usepackage{color}
\definecolor{darkgreen}{RGB}{0, 100, 0}
\allowdisplaybreaks
\newcommand{\rev}[1]{#1}
\newcommand{\norm}[1]{\left\lVert#1\right\rVert}
\newcommand{\Div}{\operatorname{div}}
\newcommand{\R}{\mathbb{R}}

\newtheorem{definition}{Definition}[section]
\newtheorem{remark}[definition]{Remark}

\newtheorem{lemma}[definition]{Lemma}
\newtheorem{proposition}[definition]{Proposition}
\newtheorem{theorem}[definition]{Theorem}
\newtheorem{assumption}[definition]{Assumption}

\numberwithin{equation}{section} 
\newcommand{\dn}{\partial_n}  
\newcommand{\ve}{\mathbf{v}}    
\newcommand{\dt}{\partial_t}  
\newcommand{\dxj}{\partial_j}        
\newcommand{\Lap}{\Delta}       
\newcommand{\mathcalL}{\mathcal{L}}
\newcommand{\grad}{\nabla} 
 
\newcommand{\BUC}{\text{BUC}}  
\newcommand{\Cspace}[2]{C^{0,#1}([0,T]; #2)}
\newcommand{\bJ}{\tilde{\mathbf{J}}}

\begin{document}
\title{Local Well-posedness of a Diffuse Interface Two-phase Flow Model with Surfactants}
\author{Helmut Abels} 
\address{Fakultät für Mathematik, Universität Regensburg, 93040 Regensburg, Germany}
\email{\href{mailto:helmut.abels@ur.de}{helmut.abels@ur.de}}

\author{Songzhuang Chen} 
\address{School of Mathematics and Statistics, Anhui Normal University, 241002 Wuhu, China}
\email{\href{mailto:songzhuang.chen@ur.de}{songzhuang.chen@ur.de}}

\date{}

\begin{abstract}
  We establish the local well-posedness of strong solutions for a thermodynamically consistent diffuse interface model describing two-phase flows with surfactants, specifically the so-called Model C introduced by Garcke, Lam, and Stinner. Because of the non-standard, mixed-order strongly coupled structure of Model C, the classical Agmon-Douglis-Nirenberg theory and conventional Lopatinskii-Shapiro boundary conditions are not applicable due to divergent scaling behaviors. To overcome this, we develop a tailored resolvent analysis, adapted to the temporal regularity of the source term, in a suitable Hilbert space, which yields an exact cancellation of the highest-order cross-couplings in the estimates. This crucial decoupling enables the rigorous application of the operator-valued Mikhlin multiplier theorem to obtain resolvent estimates on $\mathbb{R}$. 
  Furthermore, we analyze a structural variant, denoted as Model D, which achieves a crucial structural decoupling by transferring the interfacial surfactant energy entirely from the gradient part to the potential part, and we readily establish local well-posedness for strong solutions.\\

\vspace{1ex}
\noindent \textbf{Keywords:} Two-phase flows, surfactants, well-posedness, maximal regularity, mixed-order system,  resolvent estimates \\
\textbf{AMS-Classification:} 35A01, 35A02, 35K41, 74B10, 76S05
\end{abstract}
\maketitle
 \section{Introduction and main result}
 The study of two-phase flows of macroscopically immiscible, viscous, incompressible fluids involving surfactants is highly relevant to numerous biological systems and industrial processes, including emulsification, detergents, and foaming agents. Surface active agents (surfactants) are chemical compounds characterized by their ability to reduce the surface tension of fluid interfaces leading to tangential forces which results in the Marangoni effect. 
 
 Depending on the multi-fluid system, surfactants are broadly classified into two types: insoluble and soluble. Insoluble surfactants exist exclusively on the fluid interfaces, whereas soluble surfactants can reside in both the bulk fluid phases and on the interfaces. In the latter scenario, the migration and exchange of surfactants between the bulk sub-layer and the interface are governed by adsorption and desorption dynamics. 
 
 While the interfacial dynamics of surfactants are classically described by treating the interface as a moving hypersurface, which has been modeled using sharp interface approaches and studied analytically \cite{bothe2010stability, bothe2005well, garcke2006surfactant} and numerically \cite{barrett2015stable, james2004surfactant,khatri2011numerical,lai2008immersed,muradoglu2008front,xu2006level}, they inherently struggle with topological transitions. Phenomena such as the splitting or reconnection of droplets are theoretically challenging to describe and typically require ad-hoc numerical interventions within the sharp interface framework. To circumvent these difficulties, diffuse interface (or phase-field) approaches have been increasingly adopted over the past decade. By replacing the sharp boundary with a transition layer of finite thickness, these models naturally capture complex topological changes in surfactant-laden two-phase flows \cite{abels2019existence, GarckeLamStinner2014, Teigen2009,Teigen2011,VanDerSman2006}.

In Garcke, Lam, and Stinner \cite{GarckeLamStinner2014} three thermodynamically consistent diffuse interface models (denoted as Models A, B, and C) based on fundamental mass balance equations and an energy dissipation inequality were derived, where the Model C describes a scenario where the surfactant is soluble in both bulk phases within the context of sharp interface model. Motivated by its broad physical applicability and robust structure, we adopt this formulation for our current investigation. In this work, we focus on the diffuse interface model (denoted as Model C) with a constant fluid density and normalize it to $\rho = 1$ without loss of generality and mobility $m(\varphi)\equiv1$ that
\begin{subequations}\label{main system}
	\begin{align}
		\partial_t \ve + \nabla \cdot (\ve \otimes \ve + \pi I - 2\eta(\varphi) D(\ve)) &= \nabla \cdot (\tilde{\sigma}(q) (\delta I- \nabla \varphi \otimes \nabla \varphi)), \label{eq:momentum} \\
		\nabla \cdot \ve &= 0, \label{eq:incompressibility} \\
		\partial_t^\bullet \varphi &= \Delta \mu, \label{eq:phase_field} \\
		\mu + \nabla \cdot ( \tilde{\sigma}(q) \nabla \varphi) - \tilde{\sigma}(q) W'(\varphi) &= \sum_{i=1,2} \xi'_i(\varphi) (G(c^{(i)}) - q c^{(i)}), \label{eq:chemical_potential} \\
		\partial_t^\bullet (\xi_1 c^{(1)}(q) + \xi_2 c^{(2)}(q) + \delta c_\Gamma(q)) &= \sum_{i=1,2} \nabla \cdot (M_c^{(i)} (c^{(i)}(q)) \xi_i \nabla q) + \nabla \cdot (M_\Gamma(c_\Gamma(q)) \delta \nabla q) \label{eq:concentration}
	\end{align}
\end{subequations}
in $\Omega\times (0,T)$, where $\delta(\varphi, \nabla\varphi)$ is the Ginzburg-Landau free energy density defined as
$$
\delta=\delta(\varphi, \nabla\varphi) := \frac{1}{2} |\nabla\varphi|^2 +  W(\varphi),
$$ arising from the competition between the bulk potential $W(\varphi)$ and the gradient energy penalty $|\nabla\varphi|^2$. In addition, $\xi_1$ and $\xi_2$ are non-negative smooth functions representing the phase fractions. Specifically, $\xi_2(\varphi)$ satisfies $\xi_2(1) = 1$, $\xi_2(-1) = 0$, and is smooth with respect to $\varphi\in\R$, while $\xi_1(\varphi) = 1 - \xi_2(\varphi)$. Finally, we close the system by prescribing initial values:
\begin{equation}\label{init_0}
	\ve_{\mid t=0}=\ve_0,\quad\mathbf{\varphi}_{\mid t=0} =\mathbf{\varphi}_0, \quad q_{\mid t=0}=q_0, 
\end{equation}
and the boundary values
\begin{equation}\label{bound_0}
	\begin{aligned}
		\ve_{\mid \partial \Omega}=
		\partial_n \varphi_{\mid \partial \Omega}=\partial_n q_{\mid \partial \Omega}=0\quad \text { on } \partial \Omega \times(0, \infty).
	\end{aligned}
\end{equation} 

 Here $\ve$ denotes the fluid velocity,  $D\mathbf{v}:=\frac{1}{2}(\nabla\mathbf{v} + \nabla \mathbf{v}^T)$ is the rate of deformation tensor, $\pi$ is the pressure,  $I$ is the identity tensor,  $\partial_t^{\bullet}:=\partial_t+\mathbf{v} \cdot \nabla$ is the material time derivative, $\varphi :\Omega \rightarrow \mathbb{R}$ denotes the order parameter used to distinguish the bulk fluid phases.  
$\varepsilon >0$ is a parameter related to the thickness of the interfacial layer, $W$ is a potential which can be chosen to be of double-well type with two global minima at $\pm 1$,  $c^{(i)}$ is the bulk density of the surfactants in fluid $i$, $M_c^{(i)}$ is the mobility of the surfactants in fluid $i$, $G_i(c^{(i)})$ is the bulk free energy density associated to the bulk surfactant in fluid $i, (i=1,2) $.
For the viscosity $\eta(\varphi)$, we consider a strictly positive function.
 Furthermore, $c^\Gamma$ is the interfacial surfactant density, $\sigma(c^\Gamma)$ is the density dependent surface tension, $M_\Gamma$ is the mobility of the interfacial surfactants, $\gamma(c^\Gamma)$ is the free energy density associated to the interfacial surfactant. It is a fundamental assumption in our model that the free energy densities $G_i$ and $\gamma$ are strictly convex, i.e., $G_i'' > 0$ and $\gamma'' > 0$. This strict convexity guarantees the local thermodynamic stability of the uniform mixture, preventing spontaneous non-physical phase separation within the bulk fluids or on the interface.

With local thermodynamical equilibrium, the chemical potentials $G'_1(c^{(1)})$, $G'_2(c^{(2)})$, and $\gamma'(c_\Gamma)$ are equal on the interface. We therefore introduce a chemical potential, denoted by $q$, and consider this as an unknown field rather than the densities of the surfactants. Because the free energies $G_i, \gamma$ are strictly convex, their derivatives are strictly monotone, which ensures a well-defined one-to-one correspondence between the $c^{(i)}$ and $q$, i.e.,
\[
c^{(1)}=(G'_1)^{-1}(q), \quad c^{(2)}=(G'_2)^{-1}(q), \quad c_\Gamma=(\gamma')^{-1}(q).
\]
We then also write the surface tension as a function of $q$:
$$
\tilde{\sigma}(q)=\sigma(c_\Gamma(q))=\gamma(c_\Gamma(q))-c_\Gamma(q)q,\quad \text{with } q = \gamma'(c_\Gamma).
$$

However, the interfacial surfactant energy is approximated by coupling the surface energy density $\gamma(c^\Gamma(q))$ directly with the Ginzburg-Landau energy, taking the form $\int_\Omega \gamma(c^\Gamma(q)) \left(|\nabla\varphi|^2 + W(\varphi)\right)dx$. This leads to the high-order nonlinear term $|\nabla\varphi|^2$ appearing under the time derivative in the mass balance equation of \eqref{main system}. Therefore, extracting time compactness for $q$ from this structure demands strong convergence topologies for $\nabla\varphi$ beforehand, which is a highly non-trivial mathematical matter.

So far there are no analytic results on existence of weak or strong solutions, globally or locally in time for the Model C~\eqref{main system}. A main difficulty is the presence of $\partial_t (\delta(\varphi,\nabla\varphi) c_\Gamma(q))$ in \eqref{eq:concentration} and the missing strong compactness of $\nabla \varphi$ given bounds in the natural energy space. Therefore a variant of Model C was proposed in the work by Abels, Garcke, and Weber~\cite{abels2017two}, called Model D in the following, where $\nabla\varphi$ no longer occurs in this term and which has the same sharp interface limit. For this model existence of weak solutions globally in time was shown in the latter work. Recently, for other simpler and less nonlinear alternative models for soluble surfactions existence of global weak solutions in a porous medium was shown by Grasselli et al.~\cite{grasselli2026cahnhilliarddarcyforchheimersurfactantexistencelongtime} in the case of same densities and by Ouyang, Grasselli, and Wu~\cite{Ouyang_2026} with unmatched densities. 
For a model, where an additional $|\Delta\varphi|^2$ is incorporated into the free energy, existence of global weak solutions and existence of strong solutions, which global for $d=2$ and local for $d=3$, was proved by Di Primio, Grasselli, and Wu \cite{diprimio2023wellposedness}. 

The main goal of this contribution is to establish a first existence result for solutions to the Model C~\eqref{main system}, which will be in a strong sense and locally in time. Moreover, we will complement the literature by proving local well-posedness of a thermodynamically consistent variant, called Model D, where existence of weak solutions was only known so far. A detailed exposition of this variant is provided in the following subsection.

\subsection{A Variant of Model C}
As discussed by Abels, Garcke, Lam, and Weber \cite{abels2017two}, in this model the dependence of the surface tension on the surfactant chemical potential (or concentration) in the interfacial energy  is completely transferred from the gradient part to the potential part by using $\int_\Omega \left(\frac{\varepsilon}{2}|\nabla\varphi|^2 + \frac{d(q)}{\varepsilon}\psi(\varphi)\right)dx$ as part of the free energy. To ensure the appropriate sharp interface limit as $\varepsilon\to 0$ is recovered, the correct prefactor for the potential part is strictly defined as $d(q) := h(q) - h'(q)q$, which is the Legendre transform of $h(q) := (\tilde{\sigma}(q))^2$, the square of the surface tension. This yields Model D, which we consider in a bounded domain $\Omega\subset \mathbb{R}^d,~ d = 2,~3$ with a $C^2$-boundary.
\begin{subequations}\label{mod1}
	\begin{align}
		\partial_t(\rho \mathbf{v}) + \Div(\mathbf{v} \otimes(\rho \mathbf{v}+\tilde{\mathbf{J}})) + \nabla p & = \Div(2 \eta D \mathbf{v}) - \Div(\varepsilon \nabla \varphi \otimes \nabla \varphi) && \text{in } \Omega \times(0, T), \label{eq:mod1_momentum} \\
		\nabla\cdot \mathbf{v} & = 0 && \text{in } \Omega \times(0, T, \label{eq:mod1_div} \\
		\partial_t^{\bullet}\left(\frac{1}{\varepsilon} f(q) W(\varphi)+g(q)\right) & = \operatorname{div}(\tilde{m}(\varphi, q) \nabla q) && \text{in } \Omega \times(0, T), \label{eq:mod1_q} \\
		\partial_t^{\bullet} \varphi & = \operatorname{div}(m(\varphi) \nabla \mu) && \text{in } \Omega \times(0, T), \label{eq:mod1_phi} \\
		\mu & = -\varepsilon \Delta \varphi + h(q) \frac{1}{\varepsilon} W^{\prime}(\varphi) && \text{in } \Omega \times(0, T). \label{eq:mod1_mu}
	\end{align}
\end{subequations}
with the initial values
\begin{equation}\label{init}
	\begin{aligned}
		\mathbf{v}_{\mid t=0} & =\mathbf{v}_0, \quad \varphi_{\mid t=0}=\varphi_0&\quad \text{in }\Omega, \\
		\left(\frac{1}{\varepsilon} f(q) W(\varphi)+g(q)\right)_{\mid t=0} & =\frac{1}{\varepsilon} f\left(q_0\right) W\left(\varphi_0\right)+g\left(q_0\right) &\quad \text { in } \Omega,
	\end{aligned}
\end{equation}
and the boundary values
\begin{equation}\label{bound}
	\begin{aligned}
		\mathbf{v}_{\mid \partial \Omega}=\partial_n \varphi_{\mid \partial \Omega}=\partial_n q_{\mid \partial \Omega}=\partial_n \mu_{\mid \partial \Omega}=0 \quad \text { on } \partial \Omega \times(0, T).
	\end{aligned}
\end{equation}
Here, the density $\rho(\varphi)$ is defined as
\begin{equation*}
	\rho(\varphi)= \frac{\rho_2 - \rho_1}{2}\varphi + \frac{\rho_2 + \rho_1}{2},
\end{equation*}
and the relative mass flux $\tilde{\mathbf{J}}$ is given by 
\begin{equation*}
	\tilde{\mathbf{J}}=\frac{\partial \rho(\varphi)}{\partial \varphi}(-m(\varphi) \nabla \mu).
\end{equation*} 
The crucial advantage of this alternate approximation is the structural decoupling. By ensuring that no functions involving $q$ are multiplied by $|\nabla\varphi|^2$, the problematic gradient squared term is entirely removed from the time derivative in the equation for $q$. This structural modification successfully circumvents the need for the control of $\nabla\dt\varphi$, paving the way for standard compactness arguments and rigorous analysis of the coupled system.

This system is consistent with classical systems when surfactants are absent. Specifically, if the surface tension is assumed to be constant (i.e., $h(q) \equiv 1$), the surfactant dynamics decouple, and the governing equations naturally reduce to the thermodynamically consistent diffuse interface model for incompressible two-phase flows with different densities originally derived by Garcke, Gr\"un and the first author \cite{abels2012thermodynamically}. For this reduced, surfactant-free model, the global existence of weak solutions has been rigorously proven for non-degenerate mobility with singular free energy potentials \cite{abels2013existence}, as well as for degenerate mobilities \cite{abels2013incompressible}. Moreover, Abels and Weber \cite{AbelsWeber2021} and Giorgini~\cite{GiorginiCalcVar21,GiorginiIFB22} established the local existence and uniqueness of strong solutions. For a comprehensive overview of these diffuse interface models without surfactants and their related analytical results, we refer to \cite{abels2024global}.

If $h(q) \not\equiv 1$, the coupled system becomes significantly more complex. For this full diffuse interface model, the global existence of weak solutions for general initial data was established by Garcke, Weber and the first author in \cite{abels2019existence}.

\subsection{Main results }
In the following, we present the main results for Model D and Model C through two respective theorems. To establish the local well-posedness for Model C, we first outline the necessary assumptions
\begin{assumption}\label{asumm_1}
\begin{enumerate}[label=(I.A\arabic*), leftmargin=*]
    \item Let $\Omega \subseteq \mathbb{R}^d$ be a bounded domain with $C^3$-boundary and $d=2,3$.
		
		 \item Let $\eta\in C_b^4(\mathbb{R})$ be such that $\eta(s) \geq \eta_0>0$ for every $s \in \mathbb{R}$ and some $\eta_0>0$.
		 \item $W\colon \mathbb{R} \rightarrow [0,\infty)$ is four times continuously differentiable.
		\item Dictated by the underlying physics, the  free energy densities $G_i, \gamma: \mathbb{R}^+ \to \mathbb{R}$ are strictly convex, satisfying $G''_i(s), \gamma''(s) > 0$ for all $s \in \mathbb{R}^+$. For our analysis, we further suppose $G_i,\gamma$ to be of class $C^3$, and $M_c^{(i)}$, $M_\Gamma: \mathbb{R}^+ \to \mathbb{R}$ to be of class $C^2$. 
	\end{enumerate}
\end{assumption}
Based on these conditions, the main result regarding the existence and uniqueness of strong solutions for Model C is stated as follows:
\begin{theorem}[Local well-posedness of Model C] \label{thm:local_well_posedness} 
	Let $ M_c^{(i)}, M_\Gamma, G_i,\gamma, W, c_\Gamma$ be as in Assumption \ref{asumm_1}, and the initial data $\varphi_0 \in H_N^3(\Omega)$, $q_0 \in H_N^2(\Omega)$, and 
   \rev{$\ve_0 \in\left( L_\sigma^2(\Omega), H^2(\Omega)^d \cap H_0^1(\Omega)^d \cap L_\sigma^2(\Omega) \right)_{1-\frac{1}{p}, p}$}
    be given with $|\varphi_0(x)| \le 1$ for all $x \in \Omega$ and $p > 4$. Then, there exists $T > 0$ such that the system \eqref{main system}--\eqref{bound_0} has a unique solution
	\begin{equation} \label{eq:solution_spaces}
		\begin{aligned}
			\mathbf{v} &\in W_p^1(0,T; L_\sigma^2(\Omega)) \cap L^p(0,T; H^2(\Omega)^d \cap H_0^1(\Omega)^d),\\
			\varphi &\in W_p^1(0,T; H^1(\Omega)) \cap L^p(0,T; H_N^3(\Omega)) \cap W_p^2(0,T; H_N^{-3}(\Omega)), \\
			q &\in W_p^1(0,T; L^2(\Omega)) \cap L^p(0,T; H_N^2(\Omega)),
		\end{aligned}
	\end{equation}
    where $H_N^k(\Omega) := \{u \in H^k(\Omega) \mid \partial_n u = 0 \text{ on } \partial\Omega\}$ for $k \ge 2$, and $H_N^{-k}(\Omega) := (H_N^k(\Omega))'$ denotes its topological dual space.
\end{theorem}
\begin{remark}\label{remark-c1}
It should be noted that in \eqref{main system} (Model C), the phase-field equation for $\varphi$ and the surfactant chemical potential equation for $q$ are strongly coupled, making Model C a highly challenging mixed-order strongly coupled system. Specifically, the evolution equation for $q$ contains the mixed spatio-temporal derivative term $\partial_t\nabla\varphi$, which results in that the linearized system is inherently of mixed order. Furthermore, a regularity deficit arises in the fourth-order phase-field equation, which explicitly contains a third-order spatial derivative of the concentration $q$ in the term $\Delta\nabla\cdot(\tilde{\sigma}(q)\nabla\varphi)$. This non-standard mixed-order structure therefore requires a careful new mathematical analysis.
\end{remark}
To establish the main result regarding the maximal regularity of \rev{system \eqref{mod1}--\eqref{bound}}, we impose the following corresponding assumptions:
\begin{assumption}\label{asumm_2}
	\begin{enumerate}[label=(II.A\arabic*), leftmargin=*]
		\item Let $\Omega \subseteq \mathbb{R}^d$ be a bounded domain with $C^4$-boundary and $d=2,3$.
		
		\item  Let $\eta, m \in C_b^5(\mathbb{R})$ and $\tilde{m} \in C_b^5(\mathbb{R}^2)$ be such that $\eta(s) \ge \eta_0 > 0$ and $m(s) > 0$ for every $s \in \mathbb{R}$, and $\tilde{m}(s_1, s_2) > 0$ for every $(s_1, s_2) \in \mathbb{R} \times \mathbb{R}$, for some constant $\eta_0 > 0$.
%
		\item \label{2A3} $W: \mathbb{R} \rightarrow[0,\infty)$ is five times continuously differentiable.
     \rev{ \item \label{2A4} Let $f, h, g \in C_b^5(\mathbb{R})$. Furthermore, let $f$ be non-decreasing, i.e., $f'(s) \ge 0$ for all $s \in \mathbb{R}$, and $g$ be strictly increasing, such that $g'(s) \ge c_g > 0$ for all $s \in \mathbb{R}$, where $c_g$ is a strictly positive constant.    }
	\end{enumerate}
\end{assumption}
The main result of this paper establishes the local existence and uniqueness of strong solutions for the initial-boundary value problem associated with Model D. More precisely, the following holds:

\begin{theorem}[Local well-posedness of Model D]\label{th1}
	Let $\Omega, \eta, \tilde{m}, m, \rho, f, g,$ and $W$ be as in Assumption \ref{asumm_2}. Moreover, let $\mathbf{v}_0 \in H_0^1(\Omega)^d \cap L_\sigma^2(\Omega)$, $\varphi_0 \in \left(L^p(\Omega), W_{p, N}^4(\Omega)\right)_{1-\frac{1}{p}, p}$ and $q_0 \in \left(L^p(\Omega), W_{p, N}^2(\Omega)\right)_{1-\frac{1}{p}, p}$ be given with $|\varphi_0(x)| \leq 1$ for all $x \in \Omega$ and $5<p< 6$. Then, there exists $T>0$ such that the initial-boundary value problem \eqref{mod1}--\eqref{bound} has a unique strong solution
	\[
	\begin{aligned}
		\mathbf{v} &\in W_2^1\left(0, T ; L_\sigma^2(\Omega)\right) \cap L^2\left(0, T ; H^2(\Omega)^d \cap H_0^1(\Omega)^d\right), \\
		\varphi &\in W_p^1\left(0, T ; L^p(\Omega)\right) \cap L^p\left(0, T ; W_{p, N}^4(\Omega)\right), \\
		q &\in W_p^1\left(0, T ; L^p(\Omega)\right) \cap L^p\left(0, T ; W_{p, N}^2(\Omega)\right),
	\end{aligned}
	\]
\rev{   where $W_{p,N}^2(\Omega) = \{u \in W_p^2(\Omega) : \partial_n u|_{\partial\Omega} = 0\}$ and $W_{p,N}^4(\Omega) = \{u \in W_p^4(\Omega) : \partial_n u|_{\partial\Omega} = \partial_n \Delta u|_{\partial\Omega} = 0\}$.} 
\end{theorem}
\begin{remark}\label{rem:temporal_coupling}
    The mathematical analysis of the linearized system for Model D presents \rev{a structural challenge }due to the temporal coupling between the phase-field variable $\varphi$ and the surfactant concentration $q$. Specifically, the time derivative $\partial_t \varphi$ appears explicitly as an off-diagonal term in the third equation of the principal operator. Consequently, the strategy of establishing maximal regularity for each equation independently, an approach utilized by Abels and Weber \cite{AbelsWeber2021} \rev{for a Navier-Stokes-Cahn-Hilliard (NSCH) system} without surfactants cannot be applied directly here. However, by exploiting the tridiagonal structure of the principal matrix operator, this temporal coupling can be easily resolved. As detailed in the subsequent proof, applying elementary algebraic operations allows us to decouple the system and isolate the time derivatives on the diagonal, which paves the way for establishing the maximal $L^p$-regularity.
\end{remark}
\begin{remark}\label{rem:nonlinear_embedding}
	Regarding the proof of local well-posedness, \rev{the coupling structure of the system imposes analytical constraints on the integrability exponent $p$. Although bounding cross-terms in the phase-field equation requires $4<p <6$, }the surfactant equation introduces a stricter threshold of $5 < p \le 6$ due to the coupled convective term $\ve \cdot \grad(f(q)W(\varphi))$. Bridging these specific constraints with critical three-dimensional Sobolev embeddings inherently necessitates the narrow integrability window $5 < p < 6$. Consequently, we adopt this restricted range in Theorem \ref{th1}, as will be elaborated in Section \ref{sec:model_D}.
\end{remark}
\noindent\textbf{Structure of proof of the main results:} To prove the main results, namely the existence and uniqueness of local strong solutions, we employ the Banach fixed-point theorem. To this end, we first linearize \rev{the system around $(\ve_0,\varphi_0,q_0)$. This yields a decompsition of the equations into a linear part and a nonlinear part, i.e., $\mathcal L (\ve, \varphi,q) = \mathcal{F}(\ve, \varphi,q)$. Subsequently, we establish the maximal $L^p$-regularity for the linearized system, i.e., existence and continuity of $\mathcal{L}^{-1}$ between suitable Banach spaces,} and derive the local Lipschitz continuity for the nonlinear part $ \mathcal{F}$.

\rev{However, the mathematical analysis of this mixed-order coupled system in the case of Model C requires a fundamentally different approach. Specifically, as pointed out in Remark \ref{remark-c1}, the components of the principal spatial operator matrix exhibit divergent scaling behaviors in the frequency domain.This structural property does not satisfy the classical parabolicity requirement, which mandates a uniform scaling law for the entire system.} Consequently, a unified scalar parabolic weight cannot be defined. At the same time, due to the inconsistent scaling proportions, the standard resolvent estimates essential for classical operator semigroup theory with the sectorial bound $\|(\lambda I - A)^{-1}\| \le C/|\lambda|$, fail to hold globally for those mixed order systems. The compensations required for the second-order variables strictly contradict those required for the fourth-order variables across different frequency regimes.
 In addition, the classical Agmon-Douglis-Nirenberg (ADN) theory classical framework fails for our strongly coupled mixed-order system. Specifically, since the term $\partial_t^\bullet(\delta c_\Gamma(q))$ \rev{contains mixed spatio-temporal} derivative $\partial_t\nabla\varphi$. Isolating the time derivatives by substituting the fourth-order phase-field equation into the $q$-equation propagates a fifth-order spatial derivative of $\varphi$ into the $q$-equation. 
Hence, if we define an order matrix $\mathbf{O} = (O_{ij})_{2\times 2}$ to represent the highest spatial differentiation degrees of the principal operator acting on $(\varphi, q)^T$, we obtain the following structure:
$$ \mathbf{O} = \begin{pmatrix} 4 & 3 \\ 5 & 2 \end{pmatrix} $$
According to ADN theory, one must find integers $s_1, s_2$ and $t_1, t_2$ satisfying $O_{ij} = s_i + t_j$, which yields the following matrix equation:
$$ \begin{pmatrix} 4 & 3 \\ 5 & 2 \end{pmatrix} = \begin{pmatrix} s_1 + t_1 & s_1 + t_2 \\ s_2 + t_1 & s_2 + t_2 \end{pmatrix}. $$
calculating directly, we obtain $t_1 - t_2 = 1$, and $ t_1 - t_2 = 3,
$ which is algebraically impossible. As a result, the system lacks a classical principal symbol, rendering standard resolvent estimates invalid.

To tackle the fundamental issues of mixed-order systems, Denk and Kaip \cite{denk2013general} systematically developed the Newton polygon approach for $\Omega=\R^d$, i.e., \rev{a domain without boundary}. By assigning distinct, anisotropic parabolic weights to different components, they generalized the Agmon-Douglis-Nirenberg theory and constructed a generalized principal symbol. Relying on Fourier multiplier theorems, they successfully obtained maximal $L_p$-regularity for general mixed-order parabolic systems on the whole space $\mathbb{R}^d$. 

\rev{However, transitioning from the whole space $\mathbb{R}^d$ to a bounded domain $\Omega$ leads new difficulties.} Primarily, the boundary value problem requires the rigorous formulation of appropriate parabolicity/ellipticity conditions. Although the Newton polygon approach effectively resolves the operator analysis in the interior domain, relying on classical Lopatinskii-Shapiro (LS) conditions \rev{is} inadequate for such generalized mixed-order systems. As recently highlighted by Neuttiens and Sauer \cite{neuttiens2025time} in their study of the related mixed-order Cahn-Hilliard-Gurtin model, classical LS conditions merely guarantee the invertibility of the extended boundary matrix. While sufficient for standard parabolic problems, such mere invertibility is insufficient to guarantee well-posedness for strongly coupled mixed-order systems \rev{like Model C.} Instead, a well-posed complementing boundary condition must be strictly stronger: it requires the extended boundary matrix to possess a well-defined mixed-order structure on the boundary $\partial\Omega$. \rev{To address this, we rely on the underlying Hilbert space structure, which constitutes a key aspect of our mathematical approach.} 

To systematically overcome the \rev{structural} challenges posed by the mixed-order strong coupling, our strategy for Model C relies on a carefully tailored resolvent analysis. Building upon rigorous \textit{a priori} estimates, we incorporate $\lambda \hat{f}$ as an augmented source term within the resolvent equations. This strategic formulation legitimately paves the way for the application of the operator-valued Mikhlin multiplier theorem on the whole temporal space $\mathbb{R}$ and at the same time \rev{facilitates an exact cancellation of the problematic strong cross-coupling terms}. Subsequently, the \rev{Paley-Wiener} theorem is employed to restrict these global-in-time results to the finite time interval $(0, T)$. We emphasize that the application of the Mikhlin multiplier theorem in this mixed-order context is \rev{non-trivial}. Due to the heavy involvement of the spectral parameter $\lambda$ in the definition of the multipliers, verifying the requisite first-order derivative bounds for the multiplier matrices presents a \rev{significant analytical challenge}.
\rev{However, establishing these uniform first-order bounds relies on the favorable algebraic structure of the resolvent equations. Specifically, differentiating the system with respect to $\lambda$ leaves the principal operator invariant, yielding new source terms that are directly controlled by the previously established zero-order estimates.} This ultimately permits the rigorous verification of the multiplier conditions and thereby secures the maximal $L^p$-regularity of the resolvent equations on the whole temporal space. A detailed and rigorous proof of this mechanism will be presented in the resolvent estimates of Section \ref{sec:model_c}.

The structure of the article is organized as follows: In Section \ref{sec:Notation}, we fix the notation and summarize the essential preliminaries, including the definitions of the relevant function spaces used throughout the manuscript. Section \ref{sec:model_c} is devoted to establishing the local well-posedness of Model C. We begin by analyzing the linearized system, where we isolate the Stokes equations and derive the crucial resolvent estimates for the strongly coupled, mixed-order Cahn-Hilliard-surfactant subsystem for $(\varphi,q)$. To accommodate non-zero initial data, we systematically extend our maximal regularity results via the construction of a suitable lifting operator. The proof of the main theorem for Model C is then completed by rigorously verifying the local Lipschitz continuity of the nonlinear mapping $\mathcal{F}$ to close the Banach fixed-point argument.

Finally, in Section \ref{sec:model_D}, we address the local well-posedness of Model D. Given the methodological parallels with Section \ref{sec:model_c}, analogous procedures are appropriately condensed to avoid redundancy. The primary focus of this section is directed toward establishing the invertibility of the linearized operator, which is achieved by rigorously verifying the parameter-ellipticity of the principal operator and carefully handling \rev{the boundary conditions.} Furthermore, rather than repeating the exhaustive local Lipschitz estimates for the nonlinear mapping, we exclusively detail the necessary regularity derivations \rev{for the functions} and provide the estimates solely for those specific nonlinear terms that exhibit significant \rev{structural deviations from Model C.} 

\section{Notation and Preliminaries}\label{sec:Notation}
For an open set $\Omega \subseteq \mathbb{R}^d$, the usual Lebesgue space is denoted by $L^p(\Omega)$ for $1 \leq p \leq \infty$. Its vector-valued variant, for a Banach space $X$ is $L^p(\Omega; X)$. Moreover, with 
$m \geq 0$, we define Sobolev space of order $m$ as $W_p^m(\Omega)$ and its vector-valued variant $W_p^m(\Omega; X)$.

In the special case when $p=2$, the Sobolev spaces become Hilbert spaces, and we adopt the standard notation $H^m(\Omega) = W_2^m(\Omega)$. For a general real smoothness index $s \ge 0$, the corresponding space is denoted by $H^s(\Omega)$. We define $H_0^s(\Omega)$ as the closure of $C_0^\infty(\Omega)$ with respect to the $H^s(\Omega)$-norm. 

Given $f \in L^1(\Omega)$, we denote its mean value by $m(f) = \frac{1}{|\Omega|} \int_\Omega f(x) \mathrm{d}x$. For $m =0$, we define the spaces with 0-mean values as $L^q_{(0)}(\Omega) := \{f \in L^q(\Omega) : m(f) = 0\}$ for $1 \leq q \leq \infty$. Accordingly, $H^1_{(0)}(\Omega)$  is the space of $H^1$ functions with zero mean, i.e., $H^1_{(0)}(\Omega) = H^1(\Omega) \cap L^2_{(0)}(\Omega)$, and its dual space is denoted by $H^{-1}_{(0)}(\Omega) = H^1_{(0)}(\Omega)'$.

For time-dependent spaces with $0 < T< \infty$, $BC([0, T]; X)$ is the Banach space of all bounded and continuous functions equipped with the supremum norm, $BUC([0, T]; X)$ is its subspace of uniformly continuous functions.

\rev{\textbf{Besov spaces and embedding properties:}}
The standard Besov space on a domain $\Omega$ is denoted by $B_{pq}^s(\Omega)$, with smoothness index $s \in \mathbb{R}$ and exponents $1 \leq p,q \leq \infty$. Note that $H^s(\Omega) = B^s_{22}(\Omega)$. According to the classical embedding theorem (cf. \cite[section 2.2.1]{runst2011sobolev} ), let $\varepsilon >0$ and suppose $q_0<q_1$, then it holds:
\[
\begin{aligned}
	B^s_{pq_0} (\Omega)\subset B^{s}_{p q_1}(\Omega),  \quad 
B^{s+\varepsilon}_{p \infty} (\Omega)\subset B^{s}_{p q} (\Omega),
\end{aligned}
\]
\rev{for any $s\in \mathbb{R}$ and $0\le p\le \infty$.}

\textbf{Divergence-free vector fields:}
The space of divergence-free $L^2$-vector fields is $L_\sigma^2(\Omega)$, defined as the $L^2$-closure of:
\[
C_{0,\sigma}^\infty(\Omega) = \left\{ \mathbf{u} \in C_0^\infty(\Omega)^d \mid \operatorname{div}\,\mathbf{u} = 0 \right\}.
\]
The orthogonal projection $\mathbb{P}_\sigma : L^2(\Omega)^d \to L_\sigma^2(\Omega)$ (Helmholtz projection) decomposes vector fields into divergence-free and gradient fields.
\begin{lemma}\label{lem:product_estimate_dual}
	(Product estimate in the dual space $(H^1(\Omega))'$). Let $\Omega \subset \mathbb{R}^d$ be a bounded domain with a sufficiently smooth boundary, and let $s > d/2$. Then, \rev{for any multiplier $f \in H^s(\Omega)$ }and any functional $g \in (H^1(\Omega))'$, the following product estimate holds:
	\begin{equation}
		\|f g\|_{(H^1(\Omega))'} \le C \|f\|_{H^s(\Omega)} \|g\|_{(H^1(\Omega))'}, \label{eq:product_H1_dual}
	\end{equation}
	where $C > 0$ is a constant independent of $f$ and $g$.
\end{lemma}
\begin{proof}
	To rigorously establish this estimate, we utilize the definition of the dual norm and the duality pairing identity:
	\begin{equation*}
		\langle fg, v \rangle_{(H^1)', H^1} = \langle g, fv \rangle_{(H^1)', H^1},
	\end{equation*}
	which holds valid for the functional $g \in (H^1(\Omega))'$, the test function $v \in H^1(\Omega)$, and the multiplier $f$. Recognizing that the fractional Sobolev space $H^s(\Omega)$ forms a pointwise multiplier algebra for $H^1(\Omega)$ under the condition $s > d/2$, we obtain the standard multiplier inequality:
	\begin{equation*}
		\|fv\|_{H^1(\Omega)} \le C \|f\|_{H^s(\Omega)}\|v\|_{H^1(\Omega)}.
	\end{equation*}
	By the duality definition, this directly yields:
	\begin{align*}
		\|fg\|_{(H^1(\Omega))'} &= \sup_{\substack{v \in H^1(\Omega) \\ \|v\|_{H^1} \le 1}} \langle fg, v \rangle_{(H^1)', H^1} = \sup_{\substack{v \in H^1(\Omega) \\ \|v\|_{H^1} \le 1}} \langle g, fv \rangle_{(H^1)', H^1} \\
		&\le \sup_{\substack{v \in H^1(\Omega) \\ \|v\|_{H^1} \le 1}} \|g\|_{(H^1(\Omega))'} \|fv\|_{H^1(\Omega)} \le C \|f\|_{H^s(\Omega)} \|g\|_{(H^1(\Omega))'}.
	\end{align*}
	This completes the proof.
\end{proof}

To precisely accommodate the Neumann Laplacian $\Delta_N$ and ensure the correct boundary trace compatibility, we first clarify the weak formulations of the spatial operators for our variational framework.
\begin{definition}[Weak spatial operators and Neumann boundary conditions]\label{def:spatial_operators}
 The weak Neumann Laplacian is a bounded linear operator $\Delta_N : H^1(\Omega) \to H_{(0)}^{-1}(\Omega)$, which is uniquely determined by the duality pairing:
	\begin{equation}\label{eq:weak_laplacian}
    -\langle \Delta_N u, v \rangle_{H_{(0)}^{-1}, H_{(0)}^1} = \int_\Omega \nabla u \cdot \nabla v \, dx, \quad \text{for all } u \in H^1(\Omega) \text{ and } v \in H_{(0)}^1(\Omega).
\end{equation}
While the operator $\Delta_N$ possesses an one-dimensional kernel on $H^1(\Omega)$, its restriction to the zero-mean space $H_{(0)}^1(\Omega)$ establishes a isometric isomorphism onto $H_{(0)}^{-1}(\Omega)$. This structural property allows us to consider the inverse operator $(-\Delta_N)^{-1}: H_{(0)}^{-1}(\Omega) \to H_{(0)}^1(\Omega)$. 
In particular, for any $u, v \in H^{-1}_{(0)}(\Omega)$, we naturally obtain the inner product on the dual space:
\begin{equation}\label{eq:weak_H-1}
    \langle u, (-\Delta_N)^{-1} v \rangle_{H^{-1}_{(0)}, H^1_{(0)}} = \int_\Omega \nabla (-\Delta_N)^{-1} u \cdot \nabla (-\Delta_N)^{-1} v \, dx = (u, v)_{H^{-1}_{(0)}(\Omega)}.
\end{equation}
	
Furthermore, for any vector field $\mathbf{f} \in L^2(\Omega)^d$, define the weak divergence operator $\Div_N : L^2(\Omega)^d \to H_{(0)}^{-1}(\Omega)$ via the natural duality pairing:
	\begin{equation}\label{eq:weak_divergence}
		\langle -\Div_N \mathbf{f}, v \rangle_{H_{(0)}^{-1}, H_{(0)}^1} := \int_\Omega \mathbf{f} \cdot \nabla v \, dx, \quad \text{for all } v \in H_{(0)}^1(\Omega).
	\end{equation}
\end{definition}
For the Hardy space framework, we define the vector-valued Hardy spaces on the lower half-plane $\mathbb{C}^- := \{z = \tau + iy \in \mathbb{C} : y < 0\}$. For a comprehensive treatment of classical half-plane Hardy spaces, we refer to Duren \cite[Chapter 11]{duren1970theory}, and for their rigorous vector-valued generalizations, we refer to Arendt et al. \cite[Section 1.8]{arendt2011vector}.

\begin{definition}[Vector-Valued Hardy Spaces] \label{def:hardy_spaces}
	Let $\mathbb{H}$ be a Banach space, and let $F: \mathbb{C}^- \to \mathbb{H}$ be a strongly analytic function. The Hardy space $H^2(\mathbb{C}^-; \mathbb{H})$ consists of all such functions satisfying the uniform bound:
	\begin{equation}
		\|F\|_{H^2(\mathbb{C}^-; \mathbb{H})} := \sup_{y < 0} \left( \int_{-\infty}^{\infty} \|F(\tau + iy)\|_{\mathbb{H}}^2 \, d\tau \right)^{1/2} < \infty.
	\end{equation}
	Correspondingly, the space $H^\infty(\mathbb{C}^-; \mathbb{H})$ is defined as the class of strongly analytic functions uniformly bounded in the lower half-plane, equipped with the supremum norm:
	\begin{equation}
		\|F\|_{H^\infty(\mathbb{C}^-; \mathbb{H})} := \sup_{z \in \mathbb{C}^-} \|F(z)\|_{\mathbb{H}} < \infty.
	\end{equation}
\end{definition}

\rev{To handle} the strong nonlinearities arising from the thermodynamic potentials, we rely on the standard composition estimates in fractional Sobolev spaces for finitely smooth functions. 
\begin{lemma}[Composition of finitely smooth functions (cf. {\cite[Chapter 5]{runst2011sobolev}})]\label{lem:Runst_Sickel}
	Let $\Omega \subset \mathbb{R}^d$ be a bounded domain with a \rev{Lipschitz} boundary, \rev{and $R > 0$}. Assume that the outer function $G \in C^k(\mathbb{R})$ for some integer $k \ge 2$. Suppose $1 < p < \infty$ and $1 < s < k$. Then, the composition operator $T_G(f) := G(f)$ maps $H_p^s(\Omega) \cap L^\infty(\Omega)$ into itself. Moreover, there exists a constant $C > 0$, depending on $s, p, d, \Omega$, and the $C^k$-norm of $G$ on the interval $[-R, R]$, such that for all $f \in H_p^s(\Omega) \cap L^\infty(\Omega)$ with $\|f\|_{L^\infty(\Omega)} \le R$, it holds:
	\begin{equation} \label{eq:RS_bound}
		\| G(f) \|_{H_p^s(\Omega)} \le C \Big( 1 + \| f \|_{H_p^s(\Omega)} + \| f \|_{H_p^s(\Omega)} \rev{R^{s-1}} \Big).
	\end{equation}
	Furthermore, for \rev{the} given constant $R > 0$, if $\|f_1\|_{L^\infty(\Omega)}, \|f_2\|_{L^\infty(\Omega)} \le R$, there exists a local Lipschitz constant $L = L(R, \|G\|_{C^k([-R, R])}) > 0$ such that
	\begin{equation} \label{eq:RS_lip}
		\| G(f_1) - G(f_2) \|_{H_p^s(\Omega)} \le L \| f_1 - f_2 \|_{H_p^s(\Omega)}.
	\end{equation}
\end{lemma}
\begin{proof}
	\rev{This is a consequence of \cite[Chapter 5]{runst2011sobolev} using extension operators $E: H_p^s(\Omega) \to H_p^s(\mathbb{R}^d)$.}
\end{proof}

\begin{remark}
	We emphasize that the standard composition estimates on the whole space $\mathbb{R}^d$, as presented in \cite{runst2011sobolev}, inherently require the assumption $G(0) = 0$ to ensure the $p$-integrability of the composite function at spatial infinity. However, since our system is formulated on a bounded domain $\Omega$, any constant function trivially belongs to $H_p^s(\Omega)$. By employing the algebraic decomposition $G(f) = [G(f) - G(0)] + G(0)$, the fractional Sobolev norm of the constant $G(0)$ is finite and precisely absorbed by the constant $1$ on the right-hand side of \eqref{eq:RS_bound}. Consequently, the restrictive assumption $G(0) = 0$ can be \rev{omitted for a bounded domain}.
\end{remark}
In particular, we have $uv \in W_p^m(\Omega)$ for all $u, v \in W_p^m(\Omega)$ under the assumptions of the theorem.
Let $X_0, X_1$ be Banach spaces such that $X_1 \hookrightarrow X_0$ densely. It is well known that 
\begin{equation}\label{con-em-th}
	W_p^1(I; X_0) \cap L^p(I; X_1) \hookrightarrow \text{BUC}(I; (X_0, X_1)_{1-1/p, p}), \quad 1 \leq p < \infty,
\end{equation}
continuously for $I = [0, T]$, $0 < T < \infty$, and $I = [0, \infty)$. Here, $(X_0, X_1)_{\theta, p}$ denotes the real interpolation space of $(X_0, X_1)$ with exponent $\theta$ and summation index $p$. Moreover, $\text{BUC}(I; X)$ is the space of all bounded and uniformly continuous $f: I \to X$ equipped with the supremum norm, where $X$ is a Banach space.

Moreover, we will use:

\begin{lemma}\label{XYX}
	Let $X_0 \subseteq Y \subseteq X_1$ be Banach spaces such that 
	$$
	\|x\|_Y \leq C\|x\|_{X_0}^{1-\theta}\|x\|_{X_1}^\theta
	$$
	for every $x \in X_0$ and a constant $C > 0$, where $\theta \in (0,1)$. Then, 
	$$
	\mathcal{C}^{0,\alpha}([0,T]; X_1) \cap L^\infty(0,T; X_0) \hookrightarrow \mathcal{C}^{0,\alpha\theta}([0,T]; Y),
	$$
	continuously.
\end{lemma}
In what follows, we will frequently employ the embedded formula, and require that the embedding constant $C$ is independent of time $T$, which we present via the following Lemma \ref{extension}.
\begin{lemma}\label{extension}
	Let $0 < T_0 < \infty$ be given and $X_0, X_1$ be some Banach spaces such that $X_1 \hookrightarrow X_0$ densely. For every $0 < T \leq \frac{T_0}{2}$, we define
	$$
	X_T := L^p(0, T; X_1) \cap W_{p}^{1}(0, T; X_0),
	$$
	where $1 \leq p < \infty$, equipped with the norm
	$$
	\|u\|_{X_T} := \|u\|_{L^p(0, T; X_1)} + \|u\|_{W_{p}^{1}(0, T; X_0)} + \|u(0)\|_{(X_0, X_1)_{1-1/p, p}}.
	$$
	Then, there exists an extension operator $E : X_T \to X_{T_0}$ and some constant $C > 0$ independent of $T$ such that $\left. E u \right|_{(0, T)} = u$ in $X_T$ and
	$$
	\|E u\|_{X_{T_0}} \leq C \|u\|_{X_T}
	$$
	for every $u \in X_T$ and every $0 < T < \frac{T_0}{2}$. Moreover, there exists a constant $\tilde{C}(T_0) > 0$ independent of $T$ such that
	$$
	\|u\|_{BUC([0, T]; (X_0, X_1)_{1-1/p, p})} \leq \tilde{C}(T_0) \|u\|_{X_T}
	$$
	for every $u \in X_T$ and every $0 < T < \frac{T_0}{2}$.
\end{lemma}
This well-known result and its proof process are elaborated in detail in \cite[Lemma 5.2]{weber2016analysis}.

Since Hilbert spaces inherently belong to the class of $\mathcal{HT}$ spaces, we can state the Operator-Valued Mikhlin Multiplier Theorem adapted to Hilbert spaces as follows:

\begin{lemma}[Operator-Valued Mikhlin Multiplier Theorem, cf. {\cite[Theorem 3.19]{denk2003r}}] \label{thm:mikhlin}
	Let $X$ and $Y$ be Hilbert spaces, and let $1 < p < \infty$. Suppose the operator-valued multiplier $m \in C^1(\mathbb{R} \setminus \{0\}; \mathcal{L}(X, Y))$ satisfies the following uniform boundedness conditions:
	\begin{enumerate}
		\item[(i)] $\displaystyle \sup_{\tau \in \mathbb{R} \setminus \{0\}} \|m(\tau)\|_{\mathcal{L}(X, Y)} =: \kappa_0 < \infty$;
		\item[(ii)] $\displaystyle \sup_{\tau \in \mathbb{R} \setminus \{0\}} \|\tau \partial_\tau m(\tau)\|_{\mathcal{L}(X, Y)} =: \kappa_1 < \infty$.
	\end{enumerate}
	Then the multiplier operator $\mathcal{T}_m$ defined by $\mathcal{T}_m \phi := \mathcal{F}^{-1} [m(\tau) \mathcal{F} \phi(\tau)]$ \rev{for $\phi \in \mathcal{S}(\mathbb{R}; X)$ extends to}  a bounded linear operator from $L^p(\mathbb{R}; X)$ into $L^p(\mathbb{R}; Y)$. Furthermore, its operator norm is bounded by
	\begin{equation*}
		\|\mathcal{T}_m\|_{\mathcal{L}(L^p(\mathbb{R}; X), L^p(\mathbb{R}; Y))} \le C(\kappa_0 + \kappa_1),
	\end{equation*}
	where the constant $C > 0$ depends only on $p, X$, and $Y$.
\end{lemma}
%
To establish the causality of our solution operator, we invoke the Paley-Wiener Theorem. While the classical theorem is formulated for the upper half-plane, we state its equivalent adapted version for our lower half-plane setting, tailored to vector-valued Hilbert spaces.

\begin{theorem}[Paley-Wiener Theorem, cf. {\cite[Theorem 7.2]{rudin1987real}}] \label{thm:paley_wiener}
	Let $X$ be a Hilbert space and $f \in L^2(\mathbb{R}; X)$. The following two conditions are strictly equivalent:
	\begin{enumerate}
		\item[(i)] \rev{The Fourier transform can be analytically extended to a strongly holomorphic function} in the lower half-plane $\mathbb{C}^- := \{z = \tau + iy : y < 0\}$. This extension satisfies:
		\begin{equation*}
			\sup_{y < 0} \int_{\mathbb{R}} \|F(\tau + iy)\|_X^2 \, d\tau < \infty,
		\end{equation*}
		and strongly converges to its boundary value $\hat{f}(\tau)$ in the $L^2$-sense, satisfying:
		\begin{equation} \label{eq:PW_limit}
			\lim_{y \to 0^-} \int_{\mathbb{R}} \|F(\tau + iy) - \hat{f}(\tau)\|_X^2 \, d\tau = 0.
		\end{equation}
		\item[(ii)] $f(t)$ vanishes identically for all negative times, i.e., $f(t) = 0$ for almost every $t < 0$.
	\end{enumerate}
\end{theorem}
\begin{remark}
	The equivalence in our lower half-plane formulation naturally follows from the classical upper half-plane result \cite[Theorem 7.2]{rudin1987real} by exploiting the fundamental reflection symmetry between the Fourier transform and its inverse. Specifically, we utilize the identity:
	\begin{equation*}
		\mathcal{F}(f)(x) = \rev{(2\pi)^n} \mathcal{F}^{-1}(f)(-x)\rev{, \quad \text{for all } x \in \mathbb{R}.}
	\end{equation*}
\end{remark}

As a crucial preparation for establishing the existence and continuity of the inverse operator for the linearized system, we first present the following foundational theory regarding the variable viscosity Stokes operator.
 \begin{lemma}[cf.\ {\cite[Theorem 7]{Abels2009}}] \label{lem:weak_stokes}
		Let $\ve_0 \in L^2_\sigma(\Omega)$, $f \in L^2(0,T; V(\Omega)')$ and let $\varphi: \Omega \times (0,T) \to \mathbb{R}$ be a measurable function, where $0 < T \le \infty$. Then there is a unique weak solution $\ve \in L^2(0,T; V(\Omega))$ of the variable viscosity Stokes system:	
        \begin{equation} \label{eq:stokes_sys}
		\begin{cases}
			\partial_t \ve - \mathrm{div}[2\eta(\varphi)D(\ve)] + \nabla \pi = f & \text{in } \Omega \times (0,T), \\
			\nabla \cdot \ve = 0 & \text{in } \Omega \times (0,T), \\
			\ve = \mathbf{0} & \text{on } \partial\Omega \times (0,T), \\
			\ve|_{t=0} = \ve_0 & \text{in } \Omega,
		\end{cases}
	\end{equation} 
   satisfying
		\begin{equation*}
			\|\partial_t \ve\|_{L^2(0,T; V')} + \|\ve\|_{L^2(0,T; V)} \le C \left( \|\ve_0\|_2 + \|f\|_{L^2(0,T; V')} \right)
		\end{equation*}
		where $C$ depends only on the viscosity bounds and $\Omega$. Here $V=H^2(\Omega) \cap H_0^1(\Omega) \cap L^2_\sigma(\Omega)$, and $V'$ is dual space of $V'$.
	\end{lemma}
	\begin{proposition}[cf.\ {\cite[Proposition 4]{Abels2009}}] \label{prop:max_reg_stokes}
		Let $\varphi \in BUC([0,\infty); W_r^1(\Omega))$ with $r > d \ge 2$, let $0 < T \le \infty$, and let $2 \le p < \infty$. If $f \in L^p(0,T; L^2_\sigma(\Omega)) \cap L^2(0,T; V(\Omega)')$ and $\ve_0 \in (L^2_\sigma(\Omega), V(\Omega))_{1-\frac{1}{p}, p}$, then the weak solution $\ve$ satisfies $\ve \in W_p^1(0,T; L^2_\sigma(\Omega)) \cap L^p(0,T; V_2^2(\Omega))$ and 
		\begin{equation*}
			\|(\partial_t \ve, \nabla^2 \ve)\|_{L^p(0,T; L^2)} \le C_p \left( \|f\|_{L^p(0,T; L^2) \cap L^2(0,T; V_2')} + \|\ve_0\|_{(L^2_\sigma, V^2)_{1-\frac{1}{p}, p}} \right).
		\end{equation*}
	\end{proposition}

\begin{theorem} \label{thm:stokes_unified}
	Let $\Omega \subset \mathbb{R}^d$, $d=2,3$ be a bounded domain with $C^3$-boundary. We consider the generalized Stokes system: 
	\begin{equation} \label{eq:stokes_sys_unified}
		\begin{cases}
			\rho_0 \partial_t \ve - \mathrm{div}[2\eta(\varphi_0)D(\ve)] + \nabla \pi = f & \text{in } \Omega \times (0,T), \\
			\nabla \cdot \ve = 0 & \text{in } \Omega \times (0,T), \\
			\ve = \mathbf{0} & \text{on } \partial\Omega \times (0,T), \\
			\ve|_{t=0} = \ve_0 & \text{in } \Omega,
		\end{cases}
	\end{equation}
	where the viscosity satisfies $\eta(\varphi_0) \in W^{1,q}(\Omega)$ for some $q > d$ and $\eta(\varphi_0) \ge \eta_0 > 0$ on $\overline{\Omega}$. For any given $f \in L^p(0,T; L^2(\Omega)^d)$, $\ve_0 \in\mathbb{I}^p:= \left( L^2_\sigma(\Omega), H^2(\Omega)^d \cap H_0^1(\Omega)^d \cap L^2_\sigma(\Omega) \right)_{1-\frac{1}{p}, p}$ and $0 < T_0 < \infty$, the following hold:
	\begin{itemize}
	    \item[(i)] For constant density $\rho_0 \equiv 1$ and $2\leq p < \infty$: 
	    there exists a unique solution $(\mathbf{v}, \pi)$ to the Stokes system in the class $\mathbf{v} \in \mathbb{E}_T^p$, where we denote
\begin{equation*}
    \mathbb{E}_T^p := W^1_p(0,T; L^2_\sigma(\Omega)) \cap L^p(0,T; H^2(\Omega)^d \cap H^1_0(\Omega)^d \cap L^2_\sigma(\Omega)),
\end{equation*}
for all $0 < T < T_0$. Moreover, there exists a constant $C(T_0) > 0$, independent of $T \in (0, T_0]$, such that the solution satisfies the uniform estimate:
\begin{equation}
    \|\mathbf{v}\|_{\mathbb{E}_T^p} \le C(T_0) \left( \|f\|_{L^p(0,T; L^2(\Omega)^d)} + \|\mathbf{v}_0\|_{\mathbb{I}^p} \right).
\end{equation}
	    \item[(ii)] For variable density $\rho_0 = \rho(\varphi_0)$ and $p = 2$, there exists a unique solution \rev{$\ve \in \mathbb{E}_T^2$, where we denote}
    \begin{equation*}
        \rev{\mathbb{E}_T^2 := W_2^1(0, T; L^2_\sigma(\Omega)) \cap L^2(0, T; H^2(\Omega)^d\cap H^1_0(\Omega)^d),}
    \end{equation*}
    to the projected system $\mathbb{P}_\sigma( \rho_0 \partial_t \ve ) - \mathbb{P}_\sigma( \mathrm{div}[2\eta(\varphi_0)D(\ve)] ) = \mathbb{P}_\sigma f$,
    for all $0<T<T_0$. Moreover, there exists a constant $C(T_0) > 0$, independent of $T \in (0, T_0]$, such that the solution satisfies the uniform estimate:
    \begin{equation}
        \|\ve\|_{\mathbb{E}_T^2} \le C(T_0) \left( \|f\|_{L^2(0,T; L^2(\Omega)^d)} + \|\ve_0\|_{\mathbb{I}^2} \right).
    \end{equation}
	\end{itemize}
\end{theorem}

\begin{proof}
First let $\rho_0=1$ and set
$$
  A_0\ve:=-\mathbb P\Div\bigl(2\eta(\varphi_0)D(\ve)\bigr),
  \qquad \ve\in V.
$$
 For any finite time $0 < T_0 < \infty$, since $2\le p < \infty$,
$$
  f\in L^p(0,T_0;L^2(\Omega)^d)
  \hookrightarrow L^2(0,T_0;L^2(\Omega)^d)
  \hookrightarrow L^2(0,T_0; V').
$$
According to Lemma \ref{lem:weak_stokes}, the projected equation
\begin{equation}
\label{eq:stokes_projected}
  \partial_t\ve+A_0\ve=\mathbb Pf,
\end{equation}
admits a unique weak solution $\ve \in L^2(0,T_0; V(\Omega))$. 
Since $\eta(\varphi_0) \in W^{1,q}(\Omega)$ ($q>d$) is time-independent, applying Proposition~\ref{prop:max_reg_stokes} to \eqref{eq:stokes_projected} directly yields $\mathbf{v} \in \mathbb{E}_{T_0}^p$. Following the same zero-extension argument for the source term $f$ as in {\cite[ Lemma 7]{AbelsWeber2021} }, we restrict the unique existence of the solution back to the finite interval $(0, T)$, which guarantees that the associated solution mapping satisfies a uniform estimate independent of $T$.  

Furthermore, because
$$\nabla \pi=f-\partial_t\ve+\operatorname{div}\bigl(2\eta(\varphi_0)D(\ve)\bigr).
$$
and $f, \partial_t \mathbf{v}, \Div(2\eta(\varphi_0)D(\mathbf{v})) \in L^p(0,T; L^2(\Omega)^d)$, we naturally have $\nabla \pi \in L^p(0,T; L^2(\Omega)^d)$.
    
For the case in (ii), it follow directly from the weighted space framework established in {\cite[Section 5]{AbelsWeber2021} }
\end{proof}

\section{The local well-posedness of Model C}\label{sec:model_c}
Our overarching strategy is to employ the Banach fixed-point theorem to establish the existence and uniqueness of solutions. In this process, we need to apply the Banach inverse operator theorem to formulate the fixed-point equation.
\subsection{Proof of Theorem \ref{thm:local_well_posedness}}
Firstly, we reformulate system \eqref{main system} into its linear and nonlinear components, i.e., the nonlinear system \eqref{main system} is equivalent to  $$\mathcal L (\ve, \varphi,q) = \mathcal{F}(\ve, \varphi,q),$$
where $\mathcal L$  \rev{denotes a linear part of the system and $\mathcal{F}$ represents the nonlinear remainder. }

To this end, we simplify the momentum equation. We use 
\begin{equation*}
	\nabla \cdot \left[ \tilde{\sigma}(q) (\delta(\varphi, \nabla\varphi)I - \nabla\varphi \otimes \nabla\varphi) \right] = \nabla[ \tilde{\sigma}(q)  \delta(\varphi, \nabla\varphi)I] - \tilde{\sigma}(q) \nabla\cdot(\nabla\varphi \otimes \nabla\varphi )- \nabla \tilde{\sigma}(q) \cdot (\nabla\varphi \otimes \nabla\varphi ),
\end{equation*}
and
\begin{align*}
	 - \nabla \cdot (\nabla\varphi \otimes \nabla\varphi) \notag =  - \frac{1}{2}\nabla|\nabla\varphi|^2 + \Delta\varphi \nabla\varphi \notag .
\end{align*}
 Expanding the material derivative on the left-hand side of the Navier-Stokes equation, we obtain a simplified version of the momentum equation that
\begin{equation} \label{eq:ns_simplified}
\partial_t \mathbf{v} + \mathbf{v} \cdot \nabla \mathbf{v} + \nabla \Pi - \nabla \cdot (2\eta(\varphi)D(\mathbf{v})) = - \tilde{\sigma}(q) \Delta\varphi \nabla\varphi,
\end{equation}
where $\Pi=\pi-\tilde{\sigma}(q)  \delta(\varphi, \nabla\varphi)+\frac{1}{2}|\nabla\varphi|^2=\pi-W(\varphi).$
To derive the linearized equations around a given reference state $(\ve_0, q_0, \varphi_0)$, we first consider the Fréchet derivative of the principal \rev{spatial part by aking the variation }with respect to a perturbation parameter $\varepsilon$ in the direction of $(\ve,q, \varphi)$. The highest order terms with respect to $t$ and $x$ are $\partial_t \mathbf{v}$, $\nabla \cdot (2\eta(\varphi)D(\mathbf{v}))$ and $\nabla \Pi$. 
Hence, \rev{the principal part o the linearizations} are respectively given by $\partial_t \mathbf{v}$, $\nabla \cdot (2\eta(\varphi_0)D(\mathbf{v}))$ and $\nabla \Pi$. This we can see in detail in \cite{AbelsWeber2021}.
For the second equation of \eqref{main system}, we linearize the terms $\Delta\left[\nabla \cdot ( \tilde{\sigma}(q) \nabla \varphi) -  \tilde{\sigma}(q) W'(\varphi)\right]$ around a given reference state $(\ve_0, q_0, \varphi_0)$ and obtain:
\begin{equation*}
	\begin{aligned}
		&- \frac{d}{d\varepsilon} \Delta \big[\nabla \cdot (\tilde{\sigma}(q_0 + \varepsilon q) \nabla (\varphi_0 + \varepsilon \varphi))+ \tilde{\sigma}(q_0 + \varepsilon q)W'(\varphi_0 + \varepsilon \varphi) \big] \bigg|_{\varepsilon=0} \\
		&\quad= - \Delta \nabla \cdot (\tilde{\sigma}(q_0) \nabla \varphi) - \Delta \nabla \cdot (\tilde{\sigma}'(q_0) q \nabla \varphi_0) + \Delta (\tilde{\sigma}'(q_0)q W'(\varphi_0 )) + \Delta (\tilde{\sigma}(q_0) W''(\varphi_0 )\varphi).
	\end{aligned}
\end{equation*}
Note that the term $\Delta(\tilde{\sigma}(q_0)W''(\varphi_0)\varphi)$ represents a lower-order spatial contribution compared to the principal fourth-order \rev{operators with respect to $\varphi$.} This term is consequently omitted in our final linearized equation to simplify the resulting linear system.

Furthermore, we linearize the reaction terms within the chemical potential. For the summation involving the Legendre-transformed energy functions, applying the chain rule yields:
\begin{align*}
	&- \frac{d}{d\varepsilon} \Delta \sum_{i=1,2} \xi'_i(\varphi_0 + \varepsilon \varphi) \Big[ G(c^{(i)}(q_0 + \varepsilon q)) - (q_0 + \varepsilon q)c^{(i)}(q_0 + \varepsilon q) \Big] \bigg|_{\varepsilon=0} \\
	&\quad= -\Delta \sum_{i=1,2} \xi''_i(\varphi_0)\varphi \Big[ G(c^{(i)}(q_0)) - q_0 c^{(i)}(q_0) \Big] \\
	&\qquad- \Delta \sum_{i=1,2} \xi'_i(\varphi_0) \Big[ G'(c^{(i)}(q_0))(c^{(i)})'(q_0)q - c^{(i)}(q_0)q - q_0 (c^{(i)})'(q_0)q \Big].
\end{align*}
By the fundamental property of the convex conjugate, we have $G'(c^{(i)}(q_0)) = q_0$. Therefore, the terms involving $(c^{(i)})'(q_0)$ exactly cancel out. Omitting the lower-order term proportional to $\varphi$, the principal linear contribution from the reaction term elegantly reduces to:
\begin{equation*}
	\Delta \sum_{i=1,2} \xi'_i(\varphi_0) c^{(i)}(q_0) q.
\end{equation*}
Following an analogous procedure for the temporal evolution equation, we expand the nonlinearly coupled term. Applying the chain rule and product rule yields the following sequence of identities:
\begin{align*}
	&\frac{d}{d\varepsilon} \partial_t \left[ (q_0 + \varepsilon q) + \sum_{i=1,2} \xi_i(\varphi_0 + \varepsilon \varphi) c^{(i)}(q_0 + \varepsilon q) + \delta(\varphi_0 + \varepsilon \varphi, \nabla \varphi_0 + \varepsilon \nabla \varphi) c_\Gamma(q_0 + \varepsilon q) \right] \bigg|_{\varepsilon=0} \\
	&\quad= \partial_t \Bigg[ q + \sum_{i=1,2} \Big( \xi_i'(\varphi_0) c^{(i)}(q_0) \varphi + \xi_i(\varphi_0)(c^{(i)})'(q_0) q \Big)+ c_\Gamma(q_0)\nabla \varphi_0\cdot\nabla \varphi \\&\qquad\qquad + c_\Gamma(q_0)W'(\varphi_0) \varphi+\delta(\varphi_0, \nabla \varphi_0) c_\Gamma'(q_0) q \Bigg] \\
	&\quad= \alpha(q_0, \varphi_0, \nabla \varphi_0) \partial_t q + \sum_{i=1,2} \xi_i'(\varphi_0) c^{(i)}(q_0) \partial_t \varphi+ c_\Gamma(q_0) W'(\varphi_0) \partial_t \varphi + c_\Gamma(q_0) \nabla \varphi_0 \cdot \nabla \partial_t \varphi,
\end{align*}
where we have identified the effective temporal coefficient as 
\begin{equation*}
	\alpha(q_0, \varphi_0, \nabla \varphi_0) := 1 + \sum_{i=1,2} \xi_i(\varphi_0)(c^{(i)})'(q_0) + \delta(\varphi_0, \nabla \varphi_0) c_\Gamma'(q_0).
\end{equation*}

For the spatial diffusion operator in the $q$-equation, we define:
\begin{equation*}
	\tilde{m}(q, \varphi, \nabla\varphi) := \sum_{i=1,2} M^{(i)}_{c}(c^{(i)}(q)) \xi_i(\varphi) + M_\Gamma(c_\Gamma(q)) \delta(\varphi, \nabla \varphi).
\end{equation*}
Taking the Fréchet derivative of the diffusion term gives:
\begin{align*}
	&\frac{d}{d\varepsilon} \nabla \cdot \Big[ \tilde{m}(q_0 + \varepsilon q, \varphi_0 + \varepsilon \varphi, \nabla(\varphi_0 + \varepsilon \varphi)) \nabla (q_0 + \varepsilon q) \Big] \bigg|_{\varepsilon=0} \\
	&\quad= \nabla \cdot \Big[ \tilde{m}( q_0, \varphi_0,\nabla\varphi_0) \nabla q \Big] \\
	&\qquad+ \nabla \cdot \Bigg[ \Bigg( \sum_{i=1,2} \Big( (M_c^{(i)})'(c^{(i)}(q_0)) (c^{(i)})'(q_0) q \xi_i(\varphi_0) + M_c^{(i)}(c^{(i)}(q_0)) \xi_i'(\varphi_0) \varphi \Big) \\
	&\qquad\qquad\qquad + M_\Gamma'(c_\Gamma(q_0)) c_\Gamma'(q_0) q \delta(\varphi_0, \nabla\varphi_0) + M_\Gamma(c_\Gamma(q_0)) \Big( \nabla\varphi_0 \cdot \nabla\varphi + W'(\varphi_0)\varphi \Big) \Bigg) \nabla q_0 \Bigg].
\end{align*}
Since the terms containing $\nabla q_0$ \rev{are of lower oder} compared to the principal second-order diffusion operator $\nabla \cdot (\tilde{m}(q_0, \varphi_0, \nabla \varphi_0) \nabla q)$, \rev{they are omitted in the leading-order linear system. }

Finally, by summarizing the linearized contributions from the momentum, phase-field, and surfactant equations around the reference state $(\mathbf{v}_0, \varphi_0, q_0)$, we arrive at the complete linearized system. Rearranging \rev{the principal part of the linear} operators to the left-hand side and isolating all nonlinear residues and\rev{ lower-order terms on the right-hand side, the system is reformulated as}
\begin{equation} \label{eq:full_linearized_system}
	\left\{  
	\begin{aligned}
		& \partial_t \mathbf{v} - \nabla \cdot (2\eta(\varphi_0)D(\mathbf{v}))+ \nabla \Pi= \rev{\mathcal{F}_1(\mathbf{v}, \varphi, q)}, \\
		&\partial_t \varphi + \Div(\mathbf{v} \varphi_0) + \Delta \nabla \cdot \Big[\tilde{\sigma}(q_0) \nabla \varphi - c_\Gamma(q_0) q \nabla \varphi_0\Big] \\
        &\qquad\rev{- \Delta \bigg[ \tilde{\sigma}'(q_0)q W'(\varphi_0) + \sum_{i=1,2} \xi'_i(\varphi_0) c^{(i)}(q_0) q \bigg]} = \rev{\mathcal{F}_2(\mathbf{v}, \varphi, q)}, \\
		&\alpha(q_0, \varphi_0, \nabla \varphi_0)\partial_t q + c_\Gamma(q_0) \nabla \varphi_0 \cdot \nabla \partial_t \varphi - \nabla \cdot (\tilde{m}(q_0, \varphi_0, \nabla \varphi_0) \nabla q) = \rev{\mathcal{F}_3(\mathbf{v}, \varphi, q)},
	\end{aligned}
	\right.
\end{equation}
subject to the divergence-free constraint $\nabla \cdot \mathbf{v} = 0$. 
where
\begin{align*}
	\mathcal F_1(\ve, \varphi, q) &=  - \Div\left(2\eta(\varphi_0) D\ve\right) + \Div\left(2\eta(\varphi) D\ve\right)+  \tilde{\sigma}(q)\Delta\varphi \nabla\varphi - \ve\cdot \nabla \ve,\\
	\mathcal{F}_2(\ve, \varphi, q)  &= - \Delta \nabla \cdot (\tilde{\sigma}(q) \nabla \varphi) + \Delta \nabla \cdot \Big[\tilde{\sigma}(q_0) \nabla \varphi - c_\Gamma(q_0) q \nabla \varphi_0\Big] + \Delta (\tilde{\sigma}(q) W'(\varphi)) - \Delta (\tilde{\sigma}'(q_0) q W'(\varphi_0)) \\
	&\quad + \Delta \sum_{i=1,2} \xi'_i(\varphi) \big( G(c^{(i)}(q)) - q c^{(i)}(q) \big) + \Delta \sum_{i=1,2} \xi'_i(\varphi_0) c^{(i)}(q_0) q+\Div(\mathbf{v} \varphi_0)-\Div(\mathbf{v} \varphi), \\
	\mathcal{F}_3(\ve, \varphi, q)  &= - \alpha(q, \varphi, \nabla \varphi) \partial_t q + \alpha(q_0, \varphi_0, \nabla \varphi_0) \partial_t q  + \Div \big[ \tilde{m}(\varphi_0, q_0, \nabla \varphi_0) \nabla q \big] -\Div \big[ \tilde{m}(\varphi, q, \nabla \varphi) \nabla q \big] \\
	&\quad + \big( c_\Gamma(q_0) \nabla \varphi_0 - c_\Gamma(q) \nabla \varphi \big) \cdot \partial_t \nabla \varphi - W'(\varphi) c_\Gamma(q) \partial_t \varphi - \mathbf{v} \cdot \nabla \beta(q, \varphi, \nabla \varphi),
\end{align*}
where $\beta(q, \varphi, \nabla \varphi)=\xi_1(\varphi) c^{(1)}(q) + \xi_2(\varphi) c^{(2)}(q) + \delta(\varphi, \nabla \varphi) c_   \Gamma(q)$.

To establish the existence and uniqueness of strong solutions for our original nonlinear system, we reformulate the problem as an abstract fixed-point equation:
\begin{equation}
	(\ve,\varphi, q) = \mathcal{L}^{-1}(\mathcal{F}(\ve,\varphi, q)).
\end{equation}
The applicability of the Banach Contraction Mapping Theorem hinges on two fundamental properties of the operators involved: the existence and continuity of the inverse linear operator $\mathcal{L}^{-1}$, and the local Lipschitz continuity of the nonlinear mapping $\mathcal{F}$ \rev{between suitable function spaces} within the corresponding functional topologies. To this end, we set the fundamental Banach spaces $Z_T=Z_T^1\times Z_T^2\times Z_T^3$ and the target spaces $Y_T=Y_T^1\times Y_T^2\times Y_T^3$ as:$$\begin{aligned} 	Z_T^1 &:= W_p^1(0,T; L^2(\Omega)) \cap L^p(0,T; H^2(\Omega)^d \cap H_0^1(\Omega)^d),\\ 	Z_T^2 &:= W_p^1(0, T; H^1(\Omega)) \cap L^p(0, T; H_N^3(\Omega)) \cap W_p^2(0, T; H_N^{-3}(\Omega)), \\ 	Z_T^3 &:= W_p^1(0, T; L^2(\Omega))\cap L^p(0, T; H_N^2(\Omega)), \\ 	Y_T^1 &:= L^p(0,T; L^2(\Omega)),\\ 	Y_T^2 &:= W_p^1(0, T; H_N^{-3}(\Omega)) \cap L^p(0, T; H_{(0)}^{-1}(\Omega)), \\ 	Y_T^3 &:= L^p(0, T; L^2(\Omega)) \end{aligned}$$
where, as introduced in Theorem \ref{thm:local_well_posedness}, $H_N^k(\Omega) := \{u \in H^k(\Omega) \mid \partial_n u = 0 \text{ on } \partial\Omega\}$ for $k \ge 2$, and $H_N^{-k}(\Omega) := (H_N^k(\Omega))'$ denotes its topological dual space.
By virtue of the maximal $L^p$-regularity framework and the trace theory of real interpolation spaces, the initial values must belong to the respective trace spaces of $Z_T$ at $t=0$. Specifically, we require:
\begin{equation*}
\begin{aligned} 	\ve_0 &\in \gamma_0(Z_T^1) = \left( L_\sigma^2(\Omega), H^2(\Omega)^d \cap H_0^1(\Omega)^d \cap L_\sigma^2(\Omega) \right)_{1-\frac{1}{p}, p}={}_0 B_{2,p}^{2-\frac{2}{p}}(\Omega)\cap L_\sigma^2(\Omega),\\ 	\varphi_0 &\in \gamma_0(Z_T^2) = (H^1(\Omega), H_N^3(\Omega))_{1-\frac{1}{p}, p} = {}_NB_{2,p}^{3-\frac{2}{p}}(\Omega), \\ 	q_0 &\in \gamma_0(Z_T^3) = (L^2(\Omega), H_N^2(\Omega))_{1-\frac{1}{p}, p} ={}_NB_{2,p}^{2-\frac{2}{p}}(\Omega),
\end{aligned}
\end{equation*}
\rev{where ${}_0B_{2,p}^{2-\frac{2}{p}}(\Omega)$ is the closed} subspace equipped with the homogeneous Dirichlet boundary condition, defined by
\begin{equation*}
    {}_0B_{2,p}^{2-\frac{2}{p}}(\Omega) := \left\{ u \in B_{2,p}^{2-\frac{2}{p}}(\Omega) \;\middle|\; u = 0 \text{ on } \partial\Omega \right\},
\end{equation*}
and the prefix subscript $N$ denotes the incorporation of the homogeneous Neumann boundary condition. For the relevant regularity indices $s$, it is defined as
\begin{equation*}
    {}_NB_{2,p}^{s}(\Omega) := \left\{ u \in B_{2,p}^{s}(\Omega) \;\middle|\; \partial_n u = 0 \text{ on } \partial\Omega \right\}.
\end{equation*}
The existence of the Dirichlet trace requires the index to exceed $1/2$, while the normal derivative \rev{trace requires it to strictly exceed $3/2$.} Since $p > 4$, all initial compatibility conditions $\mathbf{v}_0 = 0$, $\partial_n \varphi_0 = 0$, and $\partial_n q_0 = 0$ on $\partial\Omega$ are well-defined.
We equip $Z_T^1$, $Z_T^2$, and $Z_T^3$ with the norms $\Vert{}\cdot\Vert{}_{Z_T^1}'$, $\Vert{}\cdot\Vert{}_{Z_T^2}'$, and $\Vert{}\cdot\Vert{}_{Z_T^3}'$ defined by:$$\begin{aligned} 	\Vert{}\ve\Vert{}_{Z_T^1}' &:= \Vert{}\ve\Vert{}_{Z_T^1} + \Vert{}\ve(0)\Vert{}_{{}_0 B_{2,p}^{2-\frac{2}{p}}(\Omega)\cap L_\sigma^2(\Omega)}, \\ 	\Vert{}\varphi\Vert{}_{Z_T^2}' &:= \Vert{}\varphi\Vert{}_{Z_T^2} + \Vert{}\varphi(0)\Vert{}_{{}_NB_{2,p}^{3-\frac{2}{p}}(\Omega)}, \\ 	\Vert{}q\Vert{}_{Z_T^3}' &:= \Vert{}q\Vert{}_{Z_T^3} + \Vert{}q(0)\Vert{}_{{}_NB_{2,p}^{2-\frac{2}{p}}(\Omega)}. \end{aligned}$$
We use these norms since they guarantee that for all embeddings we will study later the embedding constant $C$ can be bounded as $T\rightarrow 0$. The last preparation, before we can start with the existence proof, is the definition of the function spaces $X_T := X_T^1 \times X_T^2 \times X_T^3$ by\rev{
$$\begin{aligned} 	X_T^1 = \left\{ \ve \in Z_T^1 \mathrel{\Big\vert{}} \ve\vert{}_{t=0} = \ve_0 \right\}, \quad 	X_T^2 := \left\{ \varphi \in Z_T^2 \mathrel{\Big\vert{}} \varphi\vert{}_{t=0} = \varphi_0 \right\}, \quad X_T^3 := \left\{ q \in Z_T^3 \mathrel{\Big\vert{}} q\vert{}_{t=0} = q_0 \right\}. \end{aligned}$$}We note that $X_T$ is not a vector space due to the initial conditions. It is only an affine linear subspace of $Z_T := Z_T^1 \times Z_T^2 \times Z_T^3$. Having established the rigorous functional framework and the compatibility of the initial data, we are now positioned to verify the two fundamental propositions required for the application of the Banach contraction mapping principle.

\begin{proposition}\label{F}
	Suppose the assumptions of Theorem \ref{thm:local_well_posedness} hold true. Then, \rev{for any radius $R>0$,} there exists a constant $C(T, R) > 0$ such that the nonlinear mapping satisfies the local Lipschitz condition:
	\[
	\|\mathcal{F}(\ve_1,\varphi_1, q_1) - \mathcal{F}(\ve_2,\varphi_2, q_2)\|_{Y_T} \leq C(T, R)\|(\ve_1-\ve_2,\varphi_1 - \varphi_2, q_1 - q_2)\|_{X_T}
	\]
	for all $(\ve_i,\varphi_i, q_i) \in \overline{B_R^{X_T}(0)}$, $i=1,2$. Crucially, the Lipschitz constant exhibits the decay property $C(T, R) \to 0$ as $T \to 0$.
\end{proposition}

\begin{proposition}\label{L}
	Let $T > 0$ and define the affine linear operator $\mathcal{L}$ as before. Then, $\mathcal{L}: X_T \to Y_T$ is a continuously invertible bijection. Furthermore, for any given $T_0 > 0$, the inverse operator $\mathcal{L}^{-1}: Y_T \to X_T$ is uniformly bounded in the sense that there exists a constant $C(T_0) > 0$ satisfying:
  \begin{equation*}
	\rev{\|\mathcal L^{-1}(F_1) - \mathcal L^{-1}(F_2)\|_{X_T} \leq C(T_0)\|F_1-F_2\|_{Y_T} \quad \text{for all} \quad T \in (0, T_0],\quad F_i \in Y_T,\quad i=1,2.}
	 \end{equation*}
\end{proposition}

Provided the validity of Proposition \ref{F} and \ref{L}, we can proceed to construct the unique local strong solution to the fully nonlinear problem.

\begin{proof}[Proof of Theorem \ref{thm:local_well_posedness}]
	We seek to apply the Banach Fixed-Point Theorem to the abstract operator equation $(\ve,\varphi, q) = \mathcal{L}^{-1}(\mathcal{F}(\ve,\varphi, q))$ within the closed ball \rev{$\overline{B_R^{Z_T}}(0) \cap X_T =: \overline{B_R^{X_T}}.$} 

    To this end, we consider some reference state $(\tilde{\mathbf{v}}, \tilde{\varphi}, \tilde{q}) \in X_T$ and define
\begin{equation}\label{eq:lip_bound}
	M := \|\mathcal{L}^{-1} \circ \mathcal{F}(\tilde{\mathbf{v}}, \tilde{\varphi}, \tilde{q})\|_{X_T} < \infty.
\end{equation}

Now, let $R > 0$ be given such that $R > 2M$ and  $(\tilde{\mathbf{v}}, \tilde{\varphi}, \tilde{q}) \in \overline{B_R^{X_T}}$. Then, it follows from Proposition \ref{F} that there exists a constant $C = C(T, R) > 0$ such that
\begin{equation*}
	\|\mathcal{F}(\mathbf{v}_1, \varphi_1, q_1) - \mathcal{F}(\mathbf{v}_2, \varphi_2, q_2)\|_{Y_T} \le C(T, R)\|(\mathbf{v}_1 - \mathbf{v}_2, \varphi_1 - \varphi_2, q_1 - q_2)\|_{X_T}
\end{equation*}
for all $(\mathbf{v}_i, \varphi_i, q_i) \in X_T$ with $\|(\mathbf{v}_i, \varphi_i, q_i)\|_{X_T} \le R$, $i = 1, 2$, where it holds $C(T, R) \to 0$ as $T \to 0$.
	
	
Proposition \ref{L} ensures that $\mathcal{L}^{-1} : Y_T \to X_T$ \rev{is a bounded affine linear mapping with a Lipsichitz constant bounded by $C(T_0) > 0$ independent of $T\in (0, T_0]$. }Exploiting the time-decay of $C(T,R)$, we can select a sufficiently small final time $T > 0$ such that 
    \begin{equation}\label{eq:contraction_cond}
		C(T_0)C(T, R)   \le \frac{1}{4}.
	\end{equation}
Moreover, before verifying the contraction property, we must first confirm that the composite operator maps the closed ball into itself, namely $\mathcal{L}^{-1} \circ \mathcal{F} : \overline{B_R^{X_T}} \to \overline{B_R^{X_T}}$. Let $(\ve,\varphi, q)$ be an arbitrary element in $\overline{B_R^{X_T})}$. 
	Applying the triangle inequality in conjunction with the established bounds \eqref{eq:lip_bound} and \eqref{eq:contraction_cond}, we estimate the norm for any $(\ve,\varphi, q) \in \overline{B_R^{X_T}}$:
	\begin{align*}
		\rev{\| \mathcal{L}^{-1} \circ \mathcal{F}(\ve,\varphi, q) \|_{Z_T}} 
		&\le \| \mathcal{L}^{-1} \circ \mathcal{F}(\ve,\varphi, q) - \mathcal{L}^{-1} \circ \mathcal{F}(\tilde{\ve},\tilde{\varphi}, \tilde{q}) \|_{Z_T} + \| \mathcal{L}^{-1} \circ \mathcal{F}(\tilde{\ve},\tilde{\varphi}, \tilde{q}) \|\rev{_{Z_T}} \\
		&\le C(T_0) \| \mathcal{F}(\ve,\varphi, q) - \mathcal{F}(\tilde{\ve},\tilde{\varphi}, \tilde{q}) \|_{Y_T} + \frac{R}{2} \\
		&\le C(T_0) C(R, T) \| (\ve,\varphi, q) - (\tilde{\ve},\tilde{\varphi}, \tilde{q}) \|_{Z_T} + \frac{R}{2} \\
		&\le \frac{1}{4} \cdot 2R + \frac{R}{2} = R.
	\end{align*}
	This confirms the invariance of the ball $\overline{B_R^{X_T}}$ under the mapping $\mathcal{L}^{-1} \circ \mathcal{F}$.
	
	Having established the self-mapping property, the strict contraction is an immediate consequence of our previous estimates. For any two elements $(\ve_1,\varphi_1, q_1), (\ve_2,\varphi_2, q_2) \in \overline{B_R^{X_T}}$, we have:
	\begin{align*}
		\| \mathcal{L}^{-1} \circ \mathcal{F}(\ve_1,\varphi_1, q_1) - \mathcal{L}^{-1} \circ \mathcal{F}(\ve_2,\varphi_2, q_2) \|_{Z_T} 
		&\le C(T_0)C(R, T) \| (\ve_1,\varphi_1, q_1) - (\ve_2,\varphi_2, q_2) \|_{Z_T} \\
		&\le \frac{1}{2} \| (\ve_1,\varphi_1, q_1) - (\ve_2,\varphi_2, q_2) \|_{Z_T}.
	\end{align*}
Now, the hypotheses of the Banach Fixed-Point Theorem are fully satisfied. Thus, there exists a unique fixed point $(\ve,\varphi, q) \in \overline{B_R^{X_T}(0)}$ solving the abstract operator equation, which corresponds \rev{to a unique strong solution of the nonlinear system.}
	
	Finally, we address the global uniqueness of the solution within the entire maximal regularity space $X_T$. Suppose there exists another solution $(\ve',\varphi', q') \in X_T$. By definition, this solution must be bounded, meaning there exists a potentially larger radius $R' \ge R$ such that $(\ve',\varphi', q') \in \overline{B_{R'}^{X_T}(0)}$. Repeating the contraction argument outlined above with the new radius $R'$, we can find a sufficiently small time interval $(0, T']$ with $T' \le T$ on which the system admits a unique solution in $\overline{B_{R'}^{X_{T'}}(0)}$. This forces the coincidence of the solutions, $(\ve',\varphi', q')|_{[0, T']} = (\ve,\varphi, q)|_{[0, T']}$. A standard continuation procedure then extends this uniqueness to the entire interval $[0, T]$, concluding the proof.
\end{proof}

\subsection{Local Lipschitz continuity of \texorpdfstring{$\mathcal{F}$}{F}}\label{sec:DF}
Before proceeding \rev{with the Lipschitz} continuity, we analyze the Banach spaces where quantities $a(\ve, \varphi,q)$ like $\ve$, $\varphi$, $q$, $\nabla\varphi$, $\tilde{m}(q,\varphi,\nabla\varphi)$, and so on are bounded. Moreover, we emphasize here that when we say "$a(\ve,\varphi,q)$ is bounded in $Z$" for some function space $Z$, we mean that the set $\{a(\ve,\varphi,q) \in X_T: \|(\ve,\varphi,q)\|_{X_T} \leq R\}$ is bounded in $Z$ with an upper bound depending only on $R$ (and not on $T$). That is, there exists a constant $C(R) > 0$ such that $\|a(\ve,\varphi,q)\|_Z \leq C(R)$ for all \rev{$(\ve,\varphi,q) \in X_T$} satisfying $\|(\ve,\varphi,q)\|_{X_T} \leq R$ \rev{and $0<T \le T_0$}.

Prior to proceeding with the detailed Lipschitz continuity estimates, it is \rev{important} to identify the specific Banach spaces within which the variables $\ve, \varphi, q, \nabla\varphi$ and their composite functions $\tilde{\sigma}(q)$, $\delta(\varphi, \nabla\varphi)$, $\alpha(\varphi, \nabla\varphi, q)$, and $\tilde{m}(q,\varphi,\nabla\varphi)$ and so on remain uniformly bounded.

We begin our analysis with the velocity field $\ve \in X_T^1$ and the chemical potential $q \in X_T^3$. Since their regularity classes \rev{are the same}, we first detail the derivation for $q$, after then, the space of $\ve$ will be given similarly. 
Exploiting the continuous embedding of the maximal regularity class into the space of bounded uniformly continuous functions, we immediately obtain the regularity \rev{for} $q$:
\begin{equation*}
	q \in W_p^1(0,T; L^2(\Omega)) \cap L^p(0,T; H_N^2(\Omega)) \hookrightarrow BUC([0,T]; {}_N B_{2,p}^{2-\frac{2}{p}}(\Omega)). 
\end{equation*}
For an arbitrarily small $\varepsilon > 0$ the continuous embeddings between Besov and fractional Sobolev spaces as we show in Section \ref{sec:Notation} yield:
\begin{equation*}
	B_{2,p}^{2-\frac{2}{p}}(\Omega) \hookrightarrow B_{2,1}^{2-\frac{2}{p}-\varepsilon}(\Omega) \hookrightarrow H^{2-\frac{2}{p}-\varepsilon}(\Omega).
\end{equation*}
Let us denote $s_0 := 2 - \frac{2}{p} - \varepsilon$. It follows immediately that $q \in BUC([0, T]; H^{s_0}(\Omega))$. Meanwhile, the time derivative strictly implies:
\begin{equation*}
	q \in W_p^1(0,T; L^2(\Omega)) \hookrightarrow C^{0, 1-\frac{1}{p}}([0,T]; L^2(\Omega)).
\end{equation*}
Utilizing the interpolation inequality for fractional Sobolev spaces, for any intermediate spatial index $s_1 \in (0, s_0)$, there exists an interpolation parameter $\theta_1 \in (0, 1)$ such that:
\begin{equation*}
	\|q(t)\|_{H^{s_1}(\Omega)} \le C \|q(t)\|_{L^2(\Omega)}^{\theta_1} \|q(t)\|_{H^{s_0}(\Omega)}^{1-\theta_1}.
\end{equation*}
where $\theta_1 = \frac{s_0 - s_1}{s_0} = \frac{p(2-\varepsilon-s_1)-2}{p(2-\varepsilon)-2}$. 
Bridging these temporal and spatial bounds by Lemma \ref{XYX}, we achieve the regularity of $q$:
\begin{equation}\label{q-co-space}
	q \in C^{0, 1-\frac{1}{p}}([0, T]; L^2(\Omega)) \cap C([0, T];H^{s_0}(\Omega)) \hookrightarrow C^{0, (1-\frac{1}{p})\theta_1}([0, T]; H^{s_1}(\Omega)).
\end{equation}
To close the nonlinear estimates in classical function spaces, we require the spatial embedding $H^{s_1}(\Omega) \hookrightarrow C^\beta(\bar{\Omega})$ for some $\beta > 0$. This necessitates $s_2 > d/2$. According to the interpolation relation, we calculate:
\begin{equation*}
	s_1 = (1-\theta_1)\left(2-\frac{2}{p}\right).
\end{equation*}
By choosing $\theta_2$ sufficiently small,  the condition $s_2 > d/2$ can be satisfied, which implicitly imposes the dimension-dependent restriction $p>4$ for $d=2,3$.

Similarly, we have 
\begin{equation*}
	\ve \in C^{0, 1-\frac{1}{p}}([0, T]; L^2(\Omega)) \cap C([0, T]; H^{s_0}(\Omega)) \hookrightarrow C^{0, (1-\frac{1}{p})\theta_1}([0, T]; H^{s_1}(\Omega)),
\end{equation*}

Turning to the $\varphi \in  X_T^2$, on the one hand, by the continuous embedding \eqref{con-em-th}, we immediately obtain:
\begin{equation*}
	\varphi \in W_p^1(0,T; H^1(\Omega)) \cap L^p(0,T; H_N^3(\Omega)) \hookrightarrow BUC([0,T]; {}_N B_{2,p}^{3-\frac{2}{p}}(\Omega)),
\end{equation*}
To further optimize the regularity of $\varphi$ and prepare for the classical bounds, we exploit the fine properties of Besov space embeddings and obtain $\varphi \in BUC([0, T]; H^{\tilde{s}_0}(\Omega))$, where $\tilde{s}_0 := 3 - \frac{2}{p} - \varepsilon$

On the other hand, we have $\varphi \in C^{0, 1-\frac{1}{p}}([0, T]; H^1(\Omega))$. Utilizing the interpolation inequality for fractional Sobolev spaces, for any intermediate spatial index $s_2 \in (1, \tilde{s}_0)$, there exists an interpolation parameter $\theta_2 \in (0, 1)$ such that:
\begin{equation*}
	\|\varphi(t)\|_{H^{s_2}(\Omega)} \le C \|\varphi(t)\|_{H^1(\Omega)}^{\theta_2} \|\varphi(t)\|_{H^{\tilde{s}_0}(\Omega)}^{1-\theta_2}.
\end{equation*}
Because of $s_2 = \theta_2 \cdot 1 + (1-\theta_2)\tilde{s}_0$, we have
$
	\theta_2=\frac{p(3 - \varepsilon - s_2) - 2}{p(2 - \varepsilon) - 2}.
$
This yields  
\begin{equation*}
	\varphi \in C^{0, 1-\frac{1}{p}}([0,T]; H^1(\Omega)) \cap C([0,T]; H^{s_0+1}(\Omega)) \hookrightarrow C^{0, \theta_1(1-\frac{1}{p})}([0, T]; H^{s_2}(\Omega)).
\end{equation*}
Taking the spatial gradient shifts the regularity down by one order. Hence
\begin{equation*}
	\nabla\varphi \in C^{0, 1-\frac{1}{p}}([0,T]; L^2(\Omega)) \cap C([0,T]; H^{s_0}(\Omega)) \hookrightarrow C^{0, \theta_1(1-\frac{1}{p})}([0,T]; H^{s_2-1}(\Omega)).
\end{equation*}
We restrict $s_2 > 1+\frac{d}{2}$ in dimensions $d=2,3$ to ensure classical continuity via Sobolev embedding. 
To significantly streamline the subsequent proof, we introduce unified spatial and temporal regularity indices by defining $s := \min\{s_1, s_2 - 1\}$ and $\theta := \min\{\theta_1, \theta_2\}$, respectively. With this uniform bound, we now return to the explicit expressions of the nonlinear residual source terms $\mathcal{F}_1$, $\mathcal{F}_2$, and $\mathcal{F}_3$ in \eqref{eq:full_linearized_system}. Prior to bounding these residuals, it is essential to derive the space of the parametric functions $c^{(i)}(q)$, $c_\Gamma(q)$, $\tilde{\sigma}'(q)$ and $\tilde{\sigma}(q)$. 
Under the physically sound assumption that the surface free energy is strictly convex, i.e., $\gamma'' > 0$, the inverse function theorem guarantees that the mapping $q \mapsto c_\Gamma(q) = (\gamma')^{-1}(q)$ is at least as smooth as $\gamma'$. It follows from the Assumption \ref{asumm_1} that the free energy potential \rev{ $\gamma \in C^{k}(\mathbb{R}^+)$} for an integer $k \ge 3$, and $c_\Gamma \in C^{k-1}(\mathbb{R})$. Since $\tilde{\sigma}'(q)=c_\Gamma(q)$, we obtain $\tilde{\sigma} \in C^{k}(\mathbb{R})$. 
Furthermore, the second derivative obeys
\begin{equation}
	\tilde{\sigma}''(q) = -c'_\Gamma(q) = -\frac{1}{\gamma''(c_\Gamma(q))} < 0, \quad \text{for all }q \in \mathbb{R}.
\end{equation}
For brevity, let $z_q(q)\in\{c_\Gamma(q), \tilde{\sigma}(q), \tilde{\sigma}'(q)\}$. Recall that the chemical potential satisfies $q \in BUC([0,T]; H^{s_0}(\Omega))$ with $s_0=2 - \frac{2}{p} - \varepsilon$. Since $k \ge3 > s_0>\frac{d}{2}$, the $C^k$-regularity of $z_q(q)$ strictly satisfies the finite smoothness requirement of Lemma \ref{lem:Runst_Sickel}. That is, for any $q_1, q_2 \in X_T^2$ uniformly bounded by a constant $R$, it immediately follows that $z_q(q) \in H^{s_0}(\Omega)$ for almost every $t \in (0, T)$, along with the precise Lipschitz estimate:
\begin{equation*}
	\|z_q(q_1) -z_q(q_2) \|_{H^{s_0}(\Omega)} \le L \| q_1 - q_2 \|_{H^{s_0}(\Omega)}.
\end{equation*}
 Consequently, $z_q(q)\in L^\infty(0,T; H^{s_0}(\Omega)).$ Because of $ k-1 \ge2 > s_0$, we can get $z_q'(q)\in L^\infty(0,T; H^{s_0}(\Omega)).$
%
Applying the chain rule and the standard Moser-type calculus inequalities, we can estimate its spatial regularity:
\begin{equation*}
	\| \nabla z_q(q) \|_{H^1(\Omega)} \le C \| z_q'(q) \|_{L^\infty(\Omega)} \| \nabla q \|_{H^1(\Omega)} + C \| z_q'(q) \|_{H^1(\Omega)} \| \nabla q \|_{L^\infty(\Omega)}.
\end{equation*}
Since $H^{s_0}(\Omega) \hookrightarrow L^\infty(\Omega)$, and $q \in L^p(0,T; H^2(\Omega))$, the above estimate implies $z_q(q) \in L^p(0,T; H^2(\Omega))$.

Regarding the temporal regularity, the chain rule yields $\partial_t z_q(q) = z_q'(q)\partial_t q$. By $z_q'(q) \in L^\infty(0,T; L^\infty(\Omega))$ and $\partial_t q \in L^p(0,T; L^p(\Omega))$, we deduce that $z_q(q) \in W_p^1(0,T; L^p(\Omega))$, which is continuously embedded in $C^{0, 1-\frac{1}{p}}([0,T]; L^p(\Omega))$. Meanwhile, we have already $z_q(q) \in W_p^1(0,T; L^p(\Omega))\cap L^p(0,T; H^2(\Omega))$. Following a derivation procedure analogous to \eqref{q-co-space}, we obtain the corresponding regularity for $z_q(q)$:
\begin{equation}\label{z-co-space}
    z(q) \in C([0,T]; H^{s_0}(\Omega)) \cap C^{0, 1-\frac{1}{p}}([0,T]; L^p(\Omega)) \hookrightarrow C^{0, (1-\frac{1}{p})\theta}([0,T]; H^{s}(\Omega)).
\end{equation}

Similarly, recalling from Assumption \ref{asumm_1} that the mobility functions $M_c^{(i)}$ and $M_\Gamma$ are of class $C^k(\mathbb{R}^+)$ with $k \ge 2$, we can apply the exact same composition reasoning. Thus, the composite functions $M_c^{(i)}(q)$ and $M_\Gamma(q)$ inherit the precise identical regularity as $z_q(q)$, yielding:
\begin{equation*}
    M_c^{(i)}(q), M_\Gamma(q) \in C([0,T]; H^{s_0}(\Omega)) \cap C^{0, 1-\frac{1}{p}}([0,T]; L^p(\Omega)) \hookrightarrow C^{0, (1-\frac{1}{p})\theta}([0,T]; H^s(\Omega)).
\end{equation*}

For the $\varphi$ dependent functions $z_\varphi(\varphi)\in\{\eta(\varphi), W(\varphi), W'(\varphi), \xi_i(\varphi)\}$ ($i=1,2$), because of $\eta, W \in C^4(\mathbb{R})$, the smoothness of $\xi_i$ in $\mathbb{R}$, $\varphi \in BUC([0,T]; H^{s_0 +1}(\Omega))$, and using Lemma \ref{lem:Runst_Sickel}, we obtain $z_\varphi(\varphi) \in L^\infty(0,T; H^{s_0+1}(\Omega))$. Relying on an analogous argument as for \eqref{z-co-space}, we further get 
\begin{equation}\label{zz-co-space}
	z_\varphi(\varphi)\in C([0,T]; 	H^{s_0 +1}(\Omega)) \cap C^{0, 1-\frac{1}{p}}([0,T]; H^1(\Omega)) \hookrightarrow C^{0, (1-\frac{1}{p})\theta}([0,T]; H^{s_2}(\Omega)).
\end{equation}

Finally, we analyze the composite coefficient $
\alpha(\varphi, \nabla\varphi, q) = 1 +\sum_{i=1,2} \xi_i(\varphi)(c_{(i)})'(q) + \delta(\varphi, \nabla \varphi) c_\Gamma'(q)$. Recall the definition of the regularized function: 
\begin{equation*}
	\delta(\varphi, \nabla\varphi) = \frac{1}{2}|\nabla\varphi|^2 + W(\varphi).
\end{equation*}
Given the smooth algebraic structure of the elementary functions defining $\delta \in C^\infty(\mathbb{R}^{d+1})$, and utilizing the maximal regularity bounds $\nabla\varphi \in W_p^1(0,T; L^2(\Omega)) \cap L^p(0,T; H^2(\Omega))$, it immediately follows that:
\begin{equation*}
	\delta(\varphi, \nabla\varphi) \in C^{0, 1-\frac{1}{p}}([0,T]; L^2(\Omega)) \cap C([0,T]; 	H^{s_0}(\Omega)) \hookrightarrow C^{0, (1-\frac{1}{p})\theta}([0,T]; H^{s}(\Omega)),
\end{equation*}
which combine with \eqref{z-co-space} and \eqref{zz-co-space} yields 
 \begin{equation*}
 	\alpha (\varphi, \nabla\varphi,q) \in C^{0, 1-\frac{1}{p}}([0,T]; L^2(\Omega)) \cap C([0,T]; 	H^{s_0}(\Omega)) \hookrightarrow C^{0, (1-\frac{1}{p})\theta}([0,T]; H^{s}(\Omega)).
 \end{equation*}
 
 Following a similar procedure for the spatial diffusion operator defined as
\begin{equation*}
	\tilde{m}(q, \varphi, \nabla\varphi) := \sum_{i=1,2} M_c^{(i)}(c^{(i)}(q))\xi_i(\varphi) + M_\Gamma(c_\Gamma(q))\delta(\varphi, \nabla\varphi),
\end{equation*}
we can deduce its corresponding parametric space. As in Assumption \ref{asumm_1}, the composite mobility functions $M_c^{(i)}$ and $M_\Gamma$ possess the necessary regularity. Combining this with the smooth algebraic structure of $\xi_i(\varphi_0)$ and the derived regularity for the regularized function $\delta(\varphi, \nabla\varphi)$, we deduce:
\begin{equation*}
	\tilde{m}(q, \varphi, \nabla\varphi) \in C^{0,1-\frac{1}{p}}([0,T]; L^2(\Omega)) \cap C([0,T]; H^{s_0}(\Omega)) \hookrightarrow C^{0,(1-\frac{1}{p})\theta}([0,T]; H^s(\Omega)).
\end{equation*} 
%
With the Hölder continuity of the coefficients established, we turn our attention to the Lipschitz estimates for the nonlinear source terms. 
\begin{proof} [Proof of Proposition \ref{F}] 
Recall
\begin{align*}
\mathcal F_1(\ve, \varphi, q) &= - \mathbf{v} \cdot \nabla \mathbf{v} - \Div(2\eta(\varphi_0) D\ve) + \Div\left(2\eta(\varphi) D\ve\right)-\tilde{\sigma}(q)  \Delta\varphi \nabla\varphi. 
\end{align*}
We proceed to estimate the difference $\mathcal{F}_1(\ve_1, \varphi_1, q_1) - \mathcal{F}_1(\ve_2, \varphi_2, q_2)$ term by term in $Y^1_T:=L^p(0,T; L^2(\Omega))$.
For the first term, since $\ve, \delta(\varphi, \nabla\varphi)\in C([0,T]; H^{s_0}(\Omega)) \hookrightarrow  C([0,T]; C^0(\overline{\Omega}))$ with $s_0>\frac{d}{2}$ and  $p > 4$.
By extracting the temporal volume factor via $\| \cdot \|_{L^p(0,T)} \le T^{\frac{1}{p}} \| \cdot \|_{L^\infty(0,T)}$, we deduce:
\begin{align*}
&\| \mathbf{v}_1 \cdot \nabla \mathbf{v}_1 - \mathbf{v}_2 \cdot \nabla \mathbf{v}_2 \|_{L^p(0,T; L^2(\Omega))} \\
&\quad\le \| (\mathbf{v}_1 - \mathbf{v}_2) \cdot \nabla \mathbf{v}_1 \|_{L^p(0,T; L^2(\Omega))} + \| \mathbf{v}_2 \cdot (\nabla \mathbf{v}_1 - \nabla \mathbf{v}_2) \|_{L^p(0,T; L^2(\Omega))} \\
&\quad\le  T^{\frac{1}{p}} \| \mathbf{v}_1 - \mathbf{v}_2 \|_{L^\infty(0,T; L^\infty(\Omega))} \| \nabla \mathbf{v}_1 \|_{L^\infty(0,T; L^2(\Omega))} 
	+ T^{\frac{1}{p}} \| \mathbf{v}_2 \|_{L^\infty(0,T; L^\infty(\Omega))} \| \nabla \mathbf{v}_1 - \nabla \mathbf{v}_2 \|_{L^\infty(0,T; L^2(\Omega))} \\
&\quad\le  C R T^{\frac{1}{p}} \| \mathbf{v}_1 - \mathbf{v}_2 \|_{X_T^{1}}.
\end{align*}

By rearranging the difference and noting that $\eta(\varphi_2) - \eta(\varphi_0)$ vanishes at $t=0$, we utilize the temporal Hölder continuity $\varphi \in C^{0, (1-\frac{1}{p})\theta_1}([0,T]; H^s(\Omega))$ and multiplier inequality as in Lemma \ref{lem:product_estimate_dual} to acquire the contraction factor:
\begin{align*}
    &\big\| \Div\big(2(\eta(\varphi_1) - \eta(\varphi_0))D\mathbf{v}_1\big) - \text{div}\big(2(\eta(\varphi_2) - \eta(\varphi_0))D\mathbf{v}_2\big) \big\|_{L^p(0,T; L^2(\Omega))} \\
    &\quad \le C \big\| \Div\big( (\eta(\varphi_1) - \eta(\varphi_2)) D\mathbf{v}_1\big) \big\|_{L^p(0,T; L^2(\Omega))} + C \big\|\Div\big( (\eta(\varphi_2) - \eta(\varphi_0)) (D\mathbf{v}_1 - D\mathbf{v}_2) \big) \big\|_{L^p(0,T; L^2(\Omega))} \\
    &\quad\le C \sup_{t \in [0,T]} \| \eta(\varphi_1) - \eta(\varphi_2) \|_{H^s(\Omega)} \| \mathbf{v}_1 \|_{L^p(0,T; H^2(\Omega))} + C \sup_{t \in [0,T]} \| \eta(\varphi_2) - \eta(\varphi_0) \|_{H^s(\Omega)} \| \mathbf{v}_1 - \mathbf{v}_2 \|_{L^p(0,T; H^2(\Omega))} \\
    &\quad \le C T^{(1-\frac{1}{p})\theta_2} \Big( \| \eta(\varphi_1) - \eta(\varphi_2) \|_{C^{0, (1-\frac{1}{p})\theta_2}([0,T]; H^s(\Omega))} \| \mathbf{v}_1 \|_{X_T^1} + \| \eta(\varphi_2) \|_{C^{0, (1-\frac{1}{p})\theta_2}([0,T]; H^s(\Omega))} \| \mathbf{v}_1 - \mathbf{v}_2 \|_{X_T^1} \Big) \\
    &\quad\le C R T^{(1-\frac{1}{p})\theta_2} \big( \| \varphi_1 - \varphi_2 \|_{X_T^2} +\| \mathbf{v}_1 - \mathbf{v}_2 \|_{X_T^1} \big).
\end{align*}

For the highest-order nonlinear coupling term $\tilde{\sigma}(q)\Delta\varphi\nabla\varphi$, we define the composite function $a_1(\nabla\varphi, q) := \tilde{\sigma}(q)\nabla\varphi$. Since $\Delta\varphi \in L^\infty(0,T; L^2(\Omega))$,we have
\begin{align*}
	&\big\| a_1(\nabla\varphi_1, q_1)\Delta\varphi_1 - a_1(\nabla\varphi_2, q_2)\Delta\varphi_2 \big\|_{L^p(0,T; L^2(\Omega))} \\
	&\quad\le T^{\frac{1}{p}} \| a_1(\nabla\varphi_1, q_1) - a_1(\nabla\varphi_2, q_2) \|_{L^\infty(0,T; C^0(\overline{\Omega}))} \| \Delta\varphi_1 \|_{L^\infty(0,T; L^2(\Omega))} \\
	&\qquad + T^{\frac{1}{p}} \| a_1(\nabla\varphi_2, q_2) \|_{L^\infty(0,T; C^0(\overline{\Omega}))} \| \Delta\varphi_1 - \Delta\varphi_2 \|_{L^\infty(0,T; L^2(\Omega))} \\
	&\quad\le C R T^{\frac{1}{p}} \big( \| \varphi_1 - \varphi_2 \|_{X_T^2 } + \| q_1 - q_2 \|_{X_T^3} \big).
\end{align*}
As a conclusion, combining all the aforementioned bounds, we establish the uniform Lipschitz estimate for the projected source term:
\begin{equation}
	\| \mathcal{F}_1(\mathbf{v}_1, \varphi_1, q_1) - \mathcal{F}_1(\mathbf{v}_2, \varphi_2, q_2) \|_{L^p(0,T; L^2(\Omega))} \le C(R,T) \| (\mathbf{v}_1, \varphi_1, q_1) - (\mathbf{v}_2, \varphi_2, q_2) \|_{X_T},
\end{equation}
where $C(R, T)\rightarrow 0,$ as  $T\rightarrow 0$. For $\mathcal{F}_2$, to rigorously maintain the boundary trace compatibility established by the principal linear operator, we express the outer differential operator in $\mathcal{F}_2$ through the Neumann Laplacian $\Delta_N$ and the weak divergence operator $\text{div}_N$  yielding:
\begin{align*}
&\mathcal{F}_2(\ve, \varphi, q) = - \Delta_N \Div_N (\tilde{\sigma}(q) \nabla \varphi) + \Delta_N \Div_N \Big[\tilde{\sigma}(q_0) \nabla \varphi - c_\Gamma(q_0) q \nabla \varphi_0\Big] + \Delta_N \Big[ \tilde{\sigma}(q) W'(\varphi) - \tilde{\sigma}'(q_0) q W'(\varphi_0) \Big] \\
&\quad+ \Delta_N \sum_{i=1,2} \xi'_i(\varphi) \big( G(c^{(i)}(q)) - q c^{(i)}(q) \big) + \Delta_N \sum_{i=1,2} \xi'_i(\varphi_0) c^{(i)}(q_0) q  + \Div_N(\mathbf{v} \varphi_0) - \Div_N(\mathbf{v} \varphi). 
\end{align*}
To establish the requisite local Lipschitz continuity for the contraction mapping argument, we must rigorously estimate the difference $\|\mathcal{F}_2(\mathbf{v}_1, \varphi_1, q_1) - \mathcal{F}_2(\mathbf{v}_2, \varphi_2, q_2)\|_{Y_T^2}$, where the target space is $Y_T^2 = W_p^1(0,T; H_N^{-3}(\Omega)) \cap L^p(0,T; H_{(0)}^{-1}(\Omega))$. We divide this task into two steps: first, we address the estimation in the $W_p^1(0,T; H_N^{-3}(\Omega))$, followed by the estimation in the spatial Lebesgue space $L^p(0,T; H_{(0)}^{-1}(\Omega))$.

For the highest-order contribution in $W_p^1(0,T; H_N^{-3}(\Omega))$, to  streamline the estimation process, we peel off the outer third-order differential operator $\Delta_N \text{div}_N$. This allows us to isolate the most singular contribution, which originates from the inner flux terms, and equivalently estimate them within the base space $W_p^1(0,T; L^2(\Omega))$. Consequently, we have:
\begin{align*}
I_{11} &\le \| \partial_t \tilde{\sigma}(q_1) \nabla(\varphi_1 - \varphi_2) \|_{L^p(0,T; L^2(\Omega))} + \big\| \big(\tilde{\sigma}(q_0) - \tilde{\sigma}(q_1)\big) \partial_t \nabla(\varphi_1 - \varphi_2) \big\|_{L^p(0,T; L^2(\Omega))} \\
	&\quad \rev{+ \big\| \big(\tilde{\sigma}(q_0) - \tilde{\sigma}(q_1)\big)(\nabla\varphi_1 - \nabla\varphi_2) \big\|_{L^p(0,T; L^2(\Omega))}} \\
	&\le C \sup_{t \in [0,T]} \| \varphi_1 - \varphi_2 \|_{C^0(\bar{\Omega})} \| q_1 \|_{W_p^1(0,T; L^2(\Omega))} + C \sup_{t \in [0,T]} \| q_0 - q_1 \|_{C^0(\bar{\Omega})} \| \varphi_1 - \varphi_2 \|_{W_p^1(0,T; H^1(\Omega))} \\
	&\quad \rev{+ C T^{\frac{1}{p}} \sup_{t \in [0,T]} \|q_0 - q_1\|_{C^0(\bar{\Omega})} \sup_{t \in [0,T]} \|\varphi_1 - \varphi_2\|_{H^1(\Omega)}} \\
	&\le C R T^{(1-\frac{1}{p})\theta} \|\varphi_1 - \varphi_2\|_{C^{(1-\frac{1}{p})\theta}([0,T]; H^s(\Omega))} + C T^{(1-\frac{1}{p})\theta} \|q_1\|_{C^{(1-\frac{1}{p})\theta}([0,T]; H^s(\Omega))} \|\varphi_1 - \varphi_2\|_{X_T^2} \\
	&\quad \rev{+ C R T^{\frac{1}{p}} \|\varphi_1 - \varphi_2\|_{X_T^1}} \\
	&\le C R \rev{\big(} T^{(1-\frac{1}{p})\theta} \rev{+ T^{\frac{1}{p}} \big)} \|\varphi_1 - \varphi_2\|_{X_T^2}.
\end{align*} 
Similarly, for the second term \rev{$I_{12}$}, by applying the Hölder interpolation strategy coupled with the initial data cancellation, we deduce
\begin{align*}
	I_{12}&\le \big\| \big(\partial_t \tilde{\sigma}(q_1) - \partial_t \tilde{\sigma}(q_2)\big) \nabla(\varphi_2 - \varphi_0) \big\|_{L^p(0,T; L^2(\Omega))} + \big\| \big(\tilde{\sigma}(q_1) - \tilde{\sigma}(q_2)\big) \partial_t \nabla \varphi_2 \big\|_{L^p(0,T; L^2(\Omega))} \\
	&\quad \rev{+ \big\| \big(\tilde{\sigma}(q_1) - \tilde{\sigma}(q_2)\big) \nabla(\varphi_2 - \varphi_0) \big\|_{L^p(0,T; L^2(\Omega))}} \\
	&\le C \sup_{t \in [0,T]} \|\varphi_2 - \varphi_0\|_{C^0(\bar{\Omega})} \|q_1 - q_2\|_{W_p^1(0,T; L^2(\Omega))} + C \sup_{t \in [0,T]} \|q_1 - q_2\|_{C^0(\bar{\Omega})} \|\varphi_2\|_{W_p^1(0,T; H^1(\Omega))} \\
	&\quad \rev{+ C T^{\frac{1}{p}} \sup_{t \in [0,T]} \|q_1 - q_2\|_{C^0(\bar{\Omega})} \sup_{t \in [0,T]} \|\varphi_2 - \varphi_0\|_{H^1(\Omega)}}\\
	&\le C T^{(1-\frac{1}{p})\theta} \|\varphi_2\|_{C^{(1-\frac{1}{p})\theta}([0,T]; H^s(\Omega))} \|q_1 - q_2\|_{X_T^3}  + C R T^{(1-\frac{1}{p})\theta} \|q_1 - q_2\|_{C^{(1-\frac{1}{p})\theta}([0,T]; H^s(\Omega))} \\
	&\quad \rev{+ C R T^{\frac{1}{p}} \|q_1 - q_2\|_{X_T^3}} \\
	&\le C R \rev{\big(} T^{(1-\frac{1}{p})\theta} \rev{+ T^{\frac{1}{p}} \big)} \|q_1 - q_2\|_{X_T^3}.
\end{align*}
Finally, to bound the term $I_{13}$, we invoke the algebraic trick of adding and subtracting a cross-term to decouple the product differences. This yields
\begin{align*}
I_{13} &\rev{\le} \big\| \rev{\big(}\tilde{\sigma}'(q_1)\partial_t q_1 - \tilde{\sigma}'(q_2)\partial_t q_2 - \tilde{\sigma}'(q_0)(\partial_t q_1 - \partial_t q_2)\rev{\big)\nabla\varphi_0} \big\|_{L^p(0,T; L^2(\Omega))} \\
	&\quad \rev{+ \big\| \big(\tilde{\sigma}(q_1) - \tilde{\sigma}(q_2) - \tilde{\sigma}'(q_0)(q_1 - q_2)\big)\nabla\varphi_0 \big\|_{L^p(0,T; L^2(\Omega))}} \\
	&\le C \rev{\|\nabla\varphi_0\|_{C^0(\bar{\Omega})}} \big(\big\| (\tilde{\sigma}'(q_1) - \tilde{\sigma}'(q_0))(\partial_t q_1 - \partial_t q_2) \big\|_{L^p(0,T; L^2(\Omega))}+\big\| (\tilde{\sigma}'(q_1) - \tilde{\sigma}'(q_2))\partial_t q_2 \big\|_{L^p(0,T; L^2(\Omega))}\big) \\
	&\quad \rev{+ C \|\nabla\varphi_0\|_{C^0(\bar{\Omega})} \Big( \|\tilde{\sigma}(q_1) - \tilde{\sigma}(q_2)\|_{L^p(0,T; L^2(\Omega))} + \|\tilde{\sigma}'(q_0)\|_{C^0(\bar{\Omega})} \|q_1 - q_2\|_{L^p(0,T; L^2(\Omega))} \Big)} \\
	&\le C \sup_{t \in [0,T]} \| \tilde{\sigma}'(q_1) - \tilde{\sigma}'(q_0) \|_{C^0(\bar{\Omega})} \| q_1 - q_2 \|_{W_p^1(0,T; L^2(\Omega))}  + C \sup_{t \in [0,T]} \| \tilde{\sigma}'(q_1) - \tilde{\sigma}'(q_2) \|_{C^0(\bar{\Omega})} \| q_2 \|_{W_p^1(0,T; L^2(\Omega))} \\
	&\quad \rev{+ C T^{\frac{1}{p}} \|\nabla\varphi_0\|_{C^0(\bar{\Omega})} \Big( \sup_{t \in [0,T]}\|\tilde{\sigma}'\|_{C^0(\bar{\Omega})} + \|\tilde{\sigma}'(q_0)\|_{C^0(\bar{\Omega})} \Big) \sup_{t \in [0,T]}\|q_1 - q_2\|_{L^2(\Omega)}} \\
	&\le C R T^{(1-\frac{1}{p})\theta_2} \| q_1 \|_{C^{(1-\frac{1}{p})\theta_2}([0,T]; H^s(\Omega))} \| q_1 - q_2 \|_{X_T^3} + C R T^{(1-\frac{1}{p})\theta_2} \| q_1 - q_2 \|_{C^{(1-\frac{1}{p})\theta_2}([0,T]; H^s(\Omega))} \\
	&\quad \rev{+ C R T^{\frac{1}{p}} \|q_1 - q_2\|_{X_T^3}} \\
	&\le C R \rev{\big(} T^{(1-\frac{1}{p})\theta_2} \rev{+ T^{\frac{1}{p}} \big)} \| q_1 - q_2 \|_{X_T^3}.
\end{align*}
With the estimates $I_{11}, I_{12}$, and $I_{13}$, we complete the estimate in the space $W_p^1(0,T; H_N^{-3}(\Omega))$, i.e.,
\begin{align*}
	&\big\|\Delta_N \text{div}_N\big[ -\tilde{\sigma}(q_1)\nabla\varphi_1 + \tilde{\sigma}(q_2)\nabla\varphi_2 + \tilde{\sigma}(q_0)(\nabla\varphi_1 - \nabla\varphi_2) + \tilde{\sigma}'(q_0)(q_1 - q_2)\nabla\varphi_0 \big]\big\|_{W_p^1(0,T; L^2(\Omega)) } \\
	&\quad\le C R \big( T^{(1-\frac{1}{p})\theta} + T^{\frac{1}{p}} \big)\big(\|\varphi_1 - \varphi_2\|_{X_T^2} +\|q_1 - q_2\|_{X_T^3}\big).
\end{align*}
Peeling off the outer third-order differential operator $\Delta_N \text{div}_N$, it remains to bound the spatial regularity component in $L^p(0,T; H^2(\Omega))$. Employing the identical algebraic decomposition as before, we distribute the differences and define the corresponding spatial terms $I_{21}$, $I_{22}$, and $I_{23}$ as follow
\begin{align*}
	&\big\| -\tilde{\sigma}(q_1)\nabla\varphi_1 + \tilde{\sigma}(q_2)\nabla\varphi_2 + \tilde{\sigma}(q_0)(\nabla\varphi_1 - \nabla\varphi_2) + \tilde{\sigma}'(q_0)(q_1 - q_2)\nabla\varphi_0 \big\|_{L^p(0,T; H^2(\Omega))} \\
	&\quad\le \big\| \big(\tilde{\sigma}(q_0) - \tilde{\sigma}(q_1)\big)(\nabla\varphi_1 - \nabla\varphi_2) \big\|_{L^p(0,T; H^2(\Omega))} + \big\| \big(\tilde{\sigma}(q_1) - \tilde{\sigma}(q_2)\big)(\nabla\varphi_2 - \nabla\varphi_0) \big\|_{L^p(0,T; H^2(\Omega))} \\
	&\qquad+ \big\| \big(\tilde{\sigma}(q_1) - \tilde{\sigma}(q_2)\big)\nabla\varphi_0 - \tilde{\sigma}'(q_0)(q_1 - q_2)\nabla\varphi_0 \big\|_{L^p(0,T; H^2(\Omega))} \\
	&\quad= 	I_{21} +	I_{22} + 	I_{23}.
\end{align*}
By utilizing the standard fractional Leibniz rule (cf. \cite{kato1988commutator}), and invoking the Sobolev embedding $H^s(\Omega) \hookrightarrow L^\infty(\Omega)$ for $s > d/2$, we obtain the following product estimate in $H^2(\Omega)$:
\begin{align*}
	I_{21} &\le C \left( \int_0^T \|\tilde{\sigma}(q_0) - \tilde{\sigma}(q_1)\|_{H^s(\Omega)}^p \|\nabla\varphi_1 - \nabla\varphi_2\|_{H^2(\Omega)}^p dt \right)^{\frac{1}{p}}\\&\quad+ C \left( \int_0^T \|\tilde{\sigma}(q_0) - \tilde{\sigma}(q_1)\|_{H^2(\Omega)}^p \|\nabla\varphi_1 - \nabla\varphi_2\|_{H^s(\Omega)}^p dt \right)^{\frac{1}{p}} \\
	&\le C \sup_{t \in [0,T]} \|q_0 - q_1\|_{H^s(\Omega)} \|\varphi_1 - \varphi_2\|_{L^p(0,T; H^3(\Omega))} + C \|q_0 - q_1\|_{L^p(0,T; H^2(\Omega))} \sup_{t \in [0,T]} \|\varphi_1 - \varphi_2\|_{H^{s+1}(\Omega)} \\
	&\le C T^{(1-\frac{1}{p})\theta} \|q_1\|_{C^{(1-\frac{1}{p})\theta}([0,T]; H^s(\Omega))} \|\varphi_1 - \varphi_2\|_{X_T^1} + C T^{(1-\frac{1}{p})\theta} \|\varphi_1 - \varphi_2\|_{C^{(1-\frac{1}{p})\theta}([0,T]; H^{s+1}(\Omega))} \|q_0 - q_1\|_{X_T^2} \\
	&\le C R T^{(1-\frac{1}{p})\theta} \|\varphi_1 - \varphi_2\|_{X_T^2}.
\end{align*}
Similarly, we obtain
\begin{align*}
	I_{22} &\le C \sup_{t \in [0,T]} \|q_1 - q_2\|_{H^s(\Omega)} \|\varphi_2 - \varphi_0\|_{L^p(0,T; H^3(\Omega))} + C \sup_{t \in [0,T]} \|\varphi_2 - \varphi_0\|_{H^{s+1}(\Omega)} \|q_1 - q_2\|_{L^p(0,T; H^2(\Omega))} \\
	&\le C R T^{(1-\frac{1}{p})\theta} \|q_1 - q_2\|_{X_T^2}.
\end{align*}
And for the $I_{23}$, by the Mean Value Theorem, 
\begin{equation*}
	\tilde{\sigma}(q_1) - \tilde{\sigma}(q_2) = \int_0^1 \tilde{\sigma}'(y q_1 + (1-y)q_2) dy \, (q_1 - q_2).
\end{equation*}
Let $q_y(t) = y q_1(t) + (1-y) q_2(t)$, and $G(t) = \int_0^1 \tilde{\sigma}'(q_y(t)) - \tilde{\sigma}'(q_0) dy$, then we have $G(0) = 0$ and
\begin{equation*}
	\tilde{\sigma}(q_1) - \tilde{\sigma}(q_2) - \tilde{\sigma}'(q_0)(q_1 - q_2) = \int_0^1 \big(\tilde{\sigma}'(q_y(t)) - \tilde{\sigma}'(q_0)\big) dy \, (q_1 - q_2).
\end{equation*}
due to $\|q_y\|_{X^2_T} \le R$, the continuous Sobolev embedding $X^3_T \hookrightarrow C^{0, (1-\frac{1}{p})\theta_2}([0, T]; H^{s}(\Omega))$ guaranties $$\|q_y\|_{C^{0, (1-\frac{1}{p})\theta}([0, T]; H^{s}(\Omega))} \le C R. $$
Applying the Sobolev composition theorem we have
\begin{equation*}
    \|G(t)\|_{L^p(0,T; H^2(\Omega))} \le C R.
\end{equation*}
Furthermore, by the continuity of $\tilde{\sigma}'$ and the temporal Hölder continuity of $q_y$, there is a constant $L(R)$ such that
\begin{align*}
	&\sup_{t \in [0,T]} \|G(t)\|_{L^\infty(\Omega)} = \sup_{t \in [0,T]} \|G(t) - G(0)\|_{L^\infty(\Omega)} \le \sup_{t \in [0,T]} \int_0^1 \| \tilde{\sigma}'(q_y(t)) - \tilde{\sigma}'(q_y(0)) \|_{L^\infty(\Omega)} dy \\
	&\quad\le L(R) \int_0^1 \sup_{t \in [0,T]}  \| q_y(t) - q_y(0) \|_{L^\infty(\Omega)} dy \le L(R) \int_0^1  \| q_y \|_{C^{0, \theta(1-\frac{1}{p})}([0, T]; H^{s}(\Omega))}T^{(1-\frac{1}{p})\theta} dy \\
	&\quad\le C R L(R)  T^{(1-\frac{1}{p})\theta}.
\end{align*}
Applying the standard Moser-type product estimates in the space $H^2(\Omega)$. Distributing the norms and utilizing the $L^\infty$-bound derived above, we deduce the following
\begin{align*}
	I_{23} &\le C \big\| \tilde{\sigma}(q_1) - \tilde{\sigma}(q_2) - \tilde{\sigma}'(q_0)(q_1 - q_2) \big\|_{L^p(0,T; H^2(\Omega))} \|\nabla\varphi_0\|_{H^2(\Omega)} \\
	&\le C \left( \int_0^T \|G(t)\|_{H^s(\Omega)}^p \|q_1 - q_2\|_{H^2(\Omega)}^p dt \right)^{\frac{1}{p}} + C \left( \int_0^T \|G(t)\|_{H^2(\Omega)}^p \|q_1 - q_2\|_{H^s(\Omega)}^p dt \right)^{\frac{1}{p}} \\
	&\le  C R T^{(1-\frac{1}{p})\theta_2} \|q_1 - q_2\|_{X_T^2} +C R  \|q_1 - q_2\|_{C^{0,(1-\frac{1}{p})\theta_2}([0,T]; H^s(\Omega))}\\
	&\le C R T^{(1-\frac{1}{p})\theta_2} \|q_1 - q_2\|_{X_T^3}.
\end{align*}
In summary,
\begin{align*}
	&\big\| \Delta_N \Div_N\big[-\tilde{\sigma}(q_1)\nabla\varphi_1 + \tilde{\sigma}(q_2)\nabla\varphi_2 + \tilde{\sigma}(q_0)(\nabla\varphi_1 - \nabla\varphi_2) + \tilde{\sigma}'(q_0)(q_1 - q_2)\nabla\varphi_0 \big]\big\|_{W_p^1(0,T; H_N^{-3}(\Omega)) \cap L^p(0,T; H_{(0)} ^{-1}(\Omega))} \\
	&\quad\le C R \big( T^{(1-\frac{1}{p})\theta}  + T^{\frac{1}{p}} \big) \big( \|q_1 - q_2\|_{X_T^3} + \|\varphi_1 - \varphi_2\|_{X_T^2} \big).
\end{align*}

Finally, we estimate the lower-order nonlinear potential term in the space $W_p^1(0,T; H_N^{-3}(\Omega)) \cap L^p(0,T; H_{(0)}^{-1}(\Omega))$. Specifically, by peeling off the Neumann Laplacian operator and bounding the resulting seminorms by the full Sobolev norms, we estimate the norm as follows:
\begin{align*}
	&\big\| -\Delta_N \big[ -\tilde{\sigma}(q_1)W'(\varphi_1) + \tilde{\sigma}(q_2)W'(\varphi_2) + \tilde{\sigma}'(q_0)(q_1 - q_2)W'(\varphi_0) \big] \big\|_{W_p^1(0,T; H_N^{-3}(\Omega)) \cap L^p(0,T; H_{(0)}^{-1}(\Omega))} \\
	&\quad= \big\| \tilde{\sigma}(q_1)W'(\varphi_1) - \tilde{\sigma}(q_2)W'(\varphi_2) + \tilde{\sigma}'(q_0)(q_2 - q_1)W'(\varphi_0) \big\|_{W_p^1(0,T; H_{(0)}^{-1}(\Omega)) \cap L^p(0,T; H^1(\Omega))}.
\end{align*}
In the space $W_p^1(0,T; (H^1(\Omega))')$, setting $a_3(q, \varphi) = \tilde{\sigma}'(q)W'(\varphi)$ and $a_4(q, \varphi) = \tilde{\sigma}(q)W''(\varphi)$, we decompose the time derivative of the difference into four parts.
\begin{align}\nonumber
	&\big\| \partial_t \big[ \tilde{\sigma}(q_1)W'(\varphi_1) - \tilde{\sigma}(q_2)W'(\varphi_2) + \tilde{\sigma}'(q_0)(q_2 - q_1)W'(\varphi_0) \big] \big\|_{L^p(0,T;  H_{(0)}^{-1}(\Omega))} \\\nonumber
	&\quad\le \big\| \tilde{\sigma}'(q_1)\partial_t q_1 W'(\varphi_1) - \tilde{\sigma}'(q_2)\partial_t q_2 W'(\varphi_2)- \tilde{\sigma}'(q_0)(\partial_t q_1 - \partial_t q_2)W'(\varphi_0) \big\|_{L^p(0,T; (H^1(\Omega))')}\\\nonumber
	&\qquad +\big\|  \tilde{\sigma}(q_1)W''(\varphi_1)\partial_t \varphi_1 - \tilde{\sigma}(q_2)W''(\varphi_2)\partial_t \varphi_2  \big\|_{L^p(0,T; (H^1(\Omega))')} \\\nonumber
	&\quad\le \big\| \big(a_3(q_1, \varphi_1) - a_3(q_0, \varphi_0)\big)(\partial_t q_1 - \partial_t q_2) \big\|_{L^p(0,T; (H^1(\Omega))')} + \big\| \big(a_3(q_1, \varphi_1) - a_3(q_2, \varphi_2)\big)\partial_t q_2 \big\|_{L^p(0,T; (H^1(\Omega))')} \\\nonumber
	& \qquad+ \big\| \big(a_4(q_1, \varphi_1) - a_4(q_2, \varphi_2)\big)\partial_t \varphi_1 \big\|_{L^p(0,T; (H^1(\Omega))')} + \big\| a_4(q_2, \varphi_2)(\partial_t \varphi_1 - \partial_t \varphi_2) \big\|_{L^p(0,T; (H^1(\Omega))')} \\\label{I3}
	&\quad = I_{31} + I_{32} + I_{33} +I_{34}.
\end{align}
Recall the embeddings 
\begin{equation*}
	\varphi \in C^{0, 1-\frac{1}{p}}([0,T]; H^1(\Omega)) \cap C([0,T];   H^{s_0+1 }(\Omega)) \hookrightarrow C^{0, \theta_1(1-\frac{1}{p})}([0, T]; H^{s_2}(\Omega)),
\end{equation*}
and
\begin{equation*}
	q \in C^{0, 1-\frac{1}{p}}([0, T]; L^2(\Omega)) \cap C([0, T]; H^{s_0}(\Omega)) \hookrightarrow C^{0, \theta_2(1-\frac{1}{p})}([0, T]; H^{s_1}(\Omega)).
\end{equation*}
Let $s=\min\{s_1,s_2\}$, applying the inequalities \eqref{eq:product_H1_dual} to $K_1$ and exploiting the temporal Hölder continuity of the coefficients relative to the initial data, we obtain:
\begin{align*}
	I_{31} &\le C \sup_{t \in [0,T]} \| a_3(q_1, \varphi_1) - a_3(q_0, \varphi_0) \|_{H^s(\Omega)} \| \partial_t q_1 - \partial_t q_2 \|_{L^p(0,T; (H^1(\Omega))')} \\
	&\le C T^{(1-\frac{1}{p})\theta} \| (q_1, \varphi_1) \|_{C^{(1-\frac{1}{p})\theta}([0,T]; H^s(\Omega))} \| q_1 - q_2 \|_{W_p^1(0,T; (H^1(\Omega))')} \\
	&\le C R T^{(1-\frac{1}{p})\theta} \| q_1 - q_2 \|_{X_T^3}.
\end{align*}
For $I_{32}$, noting the precise cancelation of identical initial data in $t=0$ and similarly applying Lemma \ref{lem:product_estimate_dual}, we derive the bound of $I_{32}$ that
\begin{align*}
	I_{32} &\le C \sup_{t \in [0,T]} \big\| (q_1, \varphi_1) - (q_2, \varphi_2)\|_{H^s(\Omega)} \| \partial_t q_2 \|_{L^p(0,T; (H^1(\Omega))')} \\
	&\le C R T^{(1-\frac{1}{p})\theta} \| (q_1, \varphi_1) - (q_2, \varphi_2) \|_{C^{(1-\frac{1}{p})\theta}([0,T]; H^s(\Omega))} \\
	&\le C R T^{(1-\frac{1}{p})\theta} \big( \| q_1 - q_2 \|_{X_T^3} + \| \varphi_1 - \varphi_2 \|_{X_T^2} \big).
\end{align*}
To rigorously justify the extraction of the temporal supremum for the time derivatives in the remaining terms, we recall the maximal regularity of $\varphi$. Specifically, the time derivative satisfies $\partial_t \varphi \in L^p(0,T; H^1(\Omega)) \cap W_p^1(0,T; H_N^{-3}(\Omega))$. By the Lions-Peetre interpolation theorem, this yields the continuous embedding:
\begin{equation*}
    \partial_t \varphi \in L^p(0,T; H^1(\Omega)) \cap W_p^1(0,T; H_N^{-3}(\Omega)) \hookrightarrow BUC([0,T]; H^{s_3}(\Omega)),
\end{equation*}
where \rev{$s_3 \in (0, 1-\frac{3}{p})$}. This is supported by the interpolation inequality:
\begin{equation*}
    \|\partial_t \varphi\|_{H^{s_3}(\Omega)} \le C \|\partial_t \varphi\|_{H^1(\Omega)}^{1-\theta_3} \|\partial_t \varphi\|_{H_N^{-3}(\Omega)}^{\theta_3}
\end{equation*}
with  $s_3 = 1 - 4\theta_3$. \rev{Since $p>4$}, we ensure $s_3 > 0$, which immediately implies \rev{$H^{s_3}(\Omega) \hookrightarrow L^2(\Omega)$}. Applying Lemma \ref{lem:product_estimate_dual} to obtain
\begin{align*}
I_{33}+ I_{34}&\le C \left( \int_0^T \| a_4(q_1, \varphi_1) - a_4(q_2, \varphi_2) \|_{H^s(\Omega)}^p \rev{\| \partial_t \varphi_1 \|_{L^2(\Omega)}^p} dt \right)^{\frac{1}{p}} \\ &\quad+ C \left( \int_0^T \| a_4(q_2, \varphi_2) \|_{H^s(\Omega)}^p \rev{\| \partial_t \varphi_1 - \partial_t \varphi_2 \|_{L^2(\Omega)}^p} dt \right)^{\frac{1}{p}} \\
&\le C \sup_{t \in [0,T]} \| (q_1, \varphi_1) - (q_2, \varphi_2) \|_{H^s(\Omega)} \sup_{t \in [0,T]} \| \partial_t \varphi_1 \|_{L^{2}(\Omega)} T^{\frac{1}{p}}\\ &\quad + C \sup_{t \in [0,T]} \| a_4(q_2, \varphi_2) \|_{H^s(\Omega)} \sup_{t \in [0,T]} \| \partial_t \varphi_1 - \partial_t \varphi_2 \|_{L^{2}(\Omega)} T^{\frac{1}{p}} \\
	&\le C R T^{\frac{1}{p}} \big( \| q_1 - q_2 \|_{X_T^3} + \| \varphi_1 - \varphi_2 \|_{X_T^2} \big).
\end{align*}
Next, we address the estimate in the base space $L^p(0,T; H^1(\Omega))$ for \eqref{I3}.  Distributing the differences and utilizing the multiplier property of Sobolev spaces $H^s(\Omega) \hookrightarrow L^\infty(\Omega)$ for $s > d/2$, we deduce:
\begin{align*}
    &\big\| \tilde{\sigma}(q_1)W'(\varphi_1) - \tilde{\sigma}(q_2)W'(\varphi_2) + \tilde{\sigma}'(q_0)(q_2 - q_1)W'(\varphi_0) \big\|_{L^p(0,T; H^1(\Omega))} \\
    &\quad \le \big\| (\tilde{\sigma}(q_1) - \tilde{\sigma}(q_2))W'(\varphi_1) \big\|_{L^p(0,T; H^1(\Omega))} + \big\| \tilde{\sigma}(q_2)(W'(\varphi_1) - W'(\varphi_2)) \big\|_{L^p(0,T; H^1(\Omega))} \\ &\qquad+ \big\| \tilde{\sigma}'(q_0)W'(\varphi_0)(q_1 - q_2) \big\|_{L^p(0,T; H^1(\Omega))} \\
    &\quad \le C T^{\frac{1}{p}} \sup_{t \in [0,T]} \Big[ \| \tilde{\sigma}(q_1) - \tilde{\sigma}(q_2) \|_{H^1(\Omega)} \| W'(\varphi_1) \|_{H^s(\Omega)} \Big] \\
    &\qquad + C T^{\frac{1}{p}} \sup_{t \in [0,T]} \Big[ \| W'(\varphi_1) - W'(\varphi_2) \|_{H^1(\Omega)} \| \tilde{\sigma}(q_2) \|_{H^s(\Omega)} \Big] \\
    &\qquad + C T^{\frac{1}{p}} \| \tilde{\sigma}'(q_0)W'(\varphi_0) \|_{H^s(\Omega)} \sup_{t \in [0,T]} \| q_1 - q_2 \|_{H^1(\Omega)} \\
    &\quad \le C R T^{\frac{1}{p} + (1-\frac{1}{p})\theta} \big( \| q_1 - q_2 \|_{X_T^3} + \| \varphi_1 - \varphi_2 \|_{X_T^2} \big).
\end{align*}
Finally, this combined with the temporal bounds $I_{31}-I_{34}$ and the spatial estimate, we achieve the complete Lipschitz bound for the last term of $\mathcal{F}_2$.
\begin{align*}
	&\big\| -\Delta_N \big[ -\tilde{\sigma}(q_1)W'(\varphi_1) + \tilde{\sigma}(q_2)W'(\varphi_2) + \tilde{\sigma}'(q_0)(q_1 - q_2)W'(\varphi_0) \big] \big\|_{W_p^1(0,T; H_N^{-3}(\Omega)) \cap L^p(0,T; H_{(0)}^{-1}(\Omega))} \\
	&\quad \le C \big( T^{(1-\frac{1}{p})\theta} + T^{\frac{1}{p}} \big) \big( \| q_1 - q_2 \|_{X_T^3} + \| \varphi_1 - \varphi_2 \|_{X_T^2} \big).
\end{align*}

Since $L^2(\Omega)\hookrightarrow( H^{-2}(\Omega))'$ we can bound the difference by estimating $\mathbf{v}_1\varphi_1 - \mathbf{v}_2\varphi_2$ in $W_p^1(0,T; L^2(\Omega))$.
Because the velocity space $\mathbf{v} \in W_p^1(0,T; L^2_\sigma(\Omega))$ provides $\partial_t \mathbf{v} \in L^p(0,T; L^2(\Omega))$, we obtain:
Because the velocity space $\mathbf{v} \in W_p^1(0,T; L^2_\sigma(\Omega))$ provides $\partial_t \mathbf{v} \in L^p(0,T; L^2(\Omega))$, we obtain:
\begin{align*}
	&\big\| \Div_N(\mathbf{v}_1\varphi_1) - \Div_N(\mathbf{v}_2\varphi_2)-\Div_N((\ve_2-\ve_1)\varphi_0) \big\|_{W_p^1(0,T; H_N^{-3}(\Omega))} \\
	&\quad \le C \| (\partial_t \mathbf{v}_1 - \partial_t \mathbf{v}_2)(\varphi_1 - \varphi_0) \|_{L^p(0,T; L^2(\Omega))} + C \| (\mathbf{v}_1 - \mathbf{v}_2)\partial_t \varphi_1 \|_{L^p(0,T; L^2(\Omega))} \\
	& \qquad+ C \| \partial_t \mathbf{v}_2(\varphi_1 - \varphi_2) \|_{L^p(0,T;L^2(\Omega))} + C \| \mathbf{v}_2(\partial_t \varphi_1 - \partial_t \varphi_2) \|_{L^p(0,T; L^2(\Omega))} \\
	&\quad \le C \| \partial_t \mathbf{v}_1 - \partial_t \mathbf{v}_2 \|_{L^p(0,T; L^2(\Omega))} \| \varphi_1 - \varphi_0 \|_{L^\infty(0,T; C^0(\overline{\Omega}))}  \rev{+ C \| \mathbf{v}_2 \|_{L^\infty(0,T; C^0(\overline{\Omega}))} \| \partial_t\varphi_1 - \partial_t\varphi_2 \|_{L^\infty(0,T; L^2(\Omega))} T^{\frac{1}{p}}} \\
	&\qquad \rev{+ C \| \mathbf{v}_1 - \mathbf{v}_2 \|_{L^\infty(0,T; C^0(\overline{\Omega}))} \| \partial_t \varphi_1 \|_{L^\infty(0,T; L^2(\Omega))} T^{\frac{1}{p}}} + \| \partial_t\mathbf{v}_2 \|_{L^p(0,T; L^2(\Omega))} \| \varphi_1 - \varphi_2 \|_{L^\infty(0,T; C^0(\overline{\Omega}))}  \\
	&\quad \le C R \big( T^{(1-\frac{1}{p})\theta_1} + T^{\frac{1}{p}} \big) \big( \| \mathbf{v}_1 - \mathbf{v}_2 \|_{X_T^{1}} + \| \varphi_1 - \varphi_2 \|_{X_T^2} \big).
\end{align*}
 For the estimate in the space $L^p(0,T; H_{(0)}^{-1}(\Omega))$, we have
\begin{align*}
	&\big\| \Div_N(\mathbf{v}_1\varphi_1) - \Div_N(\mathbf{v}_2\varphi_2)-\Div_N((\ve_2-\ve_1)\varphi_0) \big\|_{L^p(0,T;H_{(0)}^{-1}(\Omega))} \\
	&\quad \le C \| ( \mathbf{v}_1 -  \mathbf{v}_2)(\varphi_1 - \varphi_0) \|_{L^p(0,T; L^2(\Omega))} +C \|  \mathbf{v}_2(\varphi_1 - \varphi_2) \|_{L^p(0,T; L^2(\Omega))} \\
	&\quad \le C \| \mathbf{v}_1 - \mathbf{v}_2 \|_{L^\infty(0,T; C^0(\overline{\Omega}))} \| \varphi_1 - \varphi_0 \|_{L^\infty(0,T; C^0(\overline{\Omega}))}T^{\frac{1}{p}}+ C \| \mathbf{v}_2 \|\rev{_{L^\infty(0,T; C^0(\overline{\Omega}))}} \| \varphi_1 - \varphi_2 \|_{L^\infty(0,T; C^0(\overline{\Omega}))}T^{\frac{1}{p}} \\
	&\quad \le C R  T^{\frac{1}{p}}  \big( \| \mathbf{v}_1 - \mathbf{v}_2 \|_{X_T^{1}} + \| \varphi_1 - \varphi_2 \|_{X_T^2} \big).
\end{align*}

Next, we estimate the highly nonlinear reaction terms in the phase-field residual $\mathcal{F}_2$. To rigorously handle the boundary conditions and operator domains, we isolate the highest-order singularity by peeling off the outer Neumann Laplacian operator $-\Delta_N$, moving the estimation from $W_p^1(0,T; H_N^{-3}(\Omega))$ to the space $W_p^1(0,T; H_{(0)}^{-1}(\Omega))$.
 Denoting $G_c^{(i)}(q) := G(c^{(i)}(q)) - q c^{(i)}(q)$. Taking the derivative with respect to $q$ yields $\frac{d G_c^{(i)}(q)}{dq} = -c^{(i)}(q)$. By distributing the differences through an exact algebraic rearrangement, we obtain:
\begin{align*}
    &\big\| -\Delta_N \sum_{i=1,2} \big[ \xi_i'(\varphi_1)G_c^{(i)}(q_1) - \xi_i'(\varphi_2)G_c^{(i)}(q_2) + \xi_i'(\varphi_0)c^{(i)}(q_0)(q_1 - q_2) \big] \big\|_{W_p^1(0,T; H_N^{-3}(\Omega))} \\
    &\quad \le C \sum_{i=1,2} \big\| \xi_i'(\varphi_1)G_c^{(i)}(q_1) - \xi_i'(\varphi_2)G_c^{(i)}(q_2) + \xi_i'(\varphi_0)c^{(i)}(q_0)(q_1 - q_2) \big\|_{W_p^1(0,T; H_{(0)}^{-1}(\Omega))} \\
    &\quad \le C \sum_{i=1,2} \big\| (\xi_i'(\varphi_1) - \xi_i'(\varphi_2)) G_c^{(i)}(q_1) \big\|_{W_p^1(0,T; (H^1(\Omega))')}\\
    &\qquad + C \sum_{i=1,2} \big\| (\xi_i'(\varphi_2) - \xi_i'(\varphi_0))(G_c^{(i)}(q_1) - G_c^{(i)}(q_2)) \big\|_{W_p^1(0,T; (H^1(\Omega))')} \\
    &\qquad + C \sum_{i=1,2} \big\| \xi_i'(\varphi_0) \big[ (G_c^{(i)}(q_1) - G_c^{(i)}(q_2)) + c^{(i)}(q_0)(q_1 - q_2) \big] \big\|_{W_p^1(0,T; (H^1(\Omega))')} \\
    &\quad = H_{31} + H_{32} + H_{33}.
\end{align*}

For $H_{31}$, leveraging the strict temporal continuity $\partial_t \varphi \in BUC([0,T];L^2(\Omega))$ established earlier, we can extract the $T^{\frac{1}{p}}$ factor from the $L^p$-norm integration, while extracting the temporal Hölder continuity for the undifferentiated difference vanishing at $t=0$:
\begin{align*}
    H_{31} &\le \sum_{i=1,2} \big\| (\partial_t \xi_i'(\varphi_1) - \partial_t \xi_i'(\varphi_2)) G_c^{(i)}(q_1) \big\|_{L^p(0,T; (H^1(\Omega))')}  + \sum_{i=1,2} \big\| (\xi_i'(\varphi_1) - \xi_i'(\varphi_2)) \partial_t G_c^{(i)}(q_1) \big\|_{L^p(0,T; (H^1(\Omega))')} \\
    &\le C T^{\frac{1}{p}} \sum_{i=1,2} \sup_{t \in [0,T]} \big\| \partial_t \xi_i'(\varphi_1) - \partial_t \xi_i'(\varphi_2) \big\|_{L^2(\Omega)} \sup_{t \in [0,T]} \big\| G_c^{(i)}(q_1) \big\|_{H^s(\Omega)} \\
    &\quad + C T^{(1-\frac{1}{p})\theta} \sum_{i=1,2} \big\| \xi_i'(\varphi_1) - \xi_i'(\varphi_2) \big\|_{C^{(1-\frac{1}{p})\theta}([0,T]; H^s(\Omega))} \big\| \partial_t G_c^{(i)}(q_1) \big\|_{L^p(0,T; L^2(\Omega))} \\
    &\le  C R T^{\frac{1}{p}} \| \partial_t \varphi_1 - \partial_t \varphi_2 \|_{L^\infty(0,T; L^2(\Omega))}  + C R T^{(1-\frac{1}{p})\theta}  \| \varphi_1 - \varphi_2 \|_{C^{0, (1-\frac{1}{p})\theta}([0,T]; H^s(\Omega))} \\
    &\le C R \big( T^{\frac{1}{p}} + T^{(1-\frac{1}{p})\theta} \big) \| \varphi_1 - \varphi_2 \|_{X_T^2}.
\end{align*}

Before estimating $H_{32}$ and $H_{33}$, we compute the common time derivative difference using the identity $\partial_t G_c^{(i)}(q) = -c^{(i)}(q) \partial_t q$ that
\begin{align*}
	&\| \partial_t G_c^{(i)}(q_1) - \partial_t G_c^{(i)}(q_2) \|_{L^p(0,T; H_{(0)}^{-1}(\Omega))}\\  
	&\quad \le\big\| -c^{(i)}(q_1) \partial_t q_1 + c^{(i)}(q_2) \partial_t q_2 \big\|_{L^p(0,T;  L^2(\Omega))} \\
	&\quad \le  \big\| (c^{(i)}(q_1) - c^{(i)}(q_2)) \partial_t q_1 \big\|_{L^p(0,T;  L^2(\Omega))} + \big\| c^{(i)}(q_2) (\partial_t q_1 - \partial_t q_2) \big\|_{L^p(0,T;  L^2(\Omega))} \\
 &\quad\le  C \| q_1 - q_2 \|_{L^\infty(0,T; H^s(\Omega))} \| \partial_t q_1 \|_{L^p(0,T;  L^2(\Omega))} + C \| c^{(i)}(q_2) \|_{L^\infty(0,T; H^s(\Omega))} \| \partial_t q_1 - \partial_t q_2 \|_{L^p(0,T; L^2(\Omega))} \\
	&\quad \le  C R \| q_1 - q_2 \|_{X_T^3}.
\end{align*}
Substituting this back into $H_{32}$, we acquire the contraction bound from the initial condition $\varphi_2|_{t=0} = \varphi_0$:
\begin{align*}
	H_{32} &\le C \sum_{i=1,2} \| \partial_t \xi_i'(\varphi_2) \|_{L^\infty(0,T; (H^1(\Omega))')} \| G_c^{(i)}(q_1) - G_c^{(i)}(q_2) \|_{L^\infty(0,T; H^s(\Omega))} T^{\frac{1}{p}} \\
	&\quad + C \sum_{i=1,2} \| \xi_i'(\varphi_2) - \xi_i'(\varphi_0) \|_{L^\infty(0,T; H^s(\Omega))} \big\| \partial_t G_c^{(i)}(q_1) - \partial_t G_c^{(i)}(q_2) \big\|_{L^p(0,T; (H^1(\Omega))')} \\
	&\le C R T^{\frac{1}{p}} \| q_1 - q_2 \|_{X_T^3} + C T^{(1-\frac{1}{p})\theta_1} C R\| \varphi_2 - \varphi_0 \|_{C^{0, (1-\frac{1}{p})\theta_1}([0,T]; H^s(\Omega))}  \| q_1 - q_2 \|_{X_T^3} \\
	&\le C R \big( T^{\frac{1}{p}} + T^{(1-\frac{1}{p})\theta} \big) \| q_1 - q_2 \|_{X_T^3}.
\end{align*}
Finally, for $H_{33}$, the constant coefficient $\xi_i'(\varphi_0)$ cleanly isolates the perfect cancellation between the nonlinear derivative and the linear principal part at $t=0$:
\begin{align*}
	H_{33} &\leq \sum_{i=1,2} \Big\| \xi_i'(\varphi_0) \big[ -c^{(i)}(q_1)\partial_t q_1 + c^{(i)}(q_2)\partial_t q_2 + c^{(i)}(q_0)\partial_t(q_1 - q_2) \big] \Big\|_{L^p(0,T; H_{(0)}^{-1}(\Omega))} \\
	&\leq \sum_{i=1,2} \Big\| \xi_i'(\varphi_0) \big[ -(c^{(i)}(q_1) - c^{(i)}(q_2))\partial_t q_1 - (c^{(i)}(q_2) - c^{(i)}(q_0))(\partial_t q_1 - \partial_t q_2) \big] \Big\|_{L^p(0,T;  (H^1(\Omega))')} \\
	&\le C \sum_{i=1,2} \| \xi_i'(\varphi_0) \|_{C^0(\overline{\Omega})} \| c^{(i)}(q_1) - c^{(i)}(q_2) \|_{L^\infty(0,T; H^s(\Omega))} \| \partial_t q_1 \|_{L^p(0,T; L^2(\Omega))} \\
	&\quad + C \sum_{i=1,2} \| \xi_i'(\varphi_0) \|_{C^0(\overline{\Omega})} \| c^{(i)}(q_2) - c^{(i)}(q_0) \|_{L^\infty(0,T; H^s(\Omega))} \| \partial_t q_1 - \partial_t q_2 \|_{L^p(0,T; L^2(\Omega))} \\
	&\le C R T^{(1-\frac{1}{p})\theta_2} \| q_1 - q_2 \|_{C^{0, (1-\frac{1}{p})\theta_2}([0,T]; H^s(\Omega))} + C R T^{(1-\frac{1}{p})\theta_2} \| q_2 - q_0 \|_{C^{0, (1-\frac{1}{p})\theta_2}([0,T]; H^s(\Omega))} \| q_1 - q_2 \|_{X_T^3} \\
	&\le C R T^{(1-\frac{1}{p})\theta} \| q_1 - q_2 \|_{X_T^3}.
\end{align*}

Combining the spatial and temporal bounds $H_{31}$, $H_{32}$, and $H_{33}$, we obtain the following
\begin{align*}
 &\big\| -\Delta_N \sum_{i=1,2} \Big[ \xi_i'(\varphi_1) G_c^{(i)}(q_1) - \xi_i'(\varphi_2) G_c^{(i)}(q_2) + \xi_i'(\varphi_0) c^{(i)}(q_0) (q_1 - q_2) \Big] \big\|_{W_p^1(0,T; H_N^{-3}(\Omega))}\\
	&\quad\le  C R \big( T^{\frac{1}{p}} + T^{(1-\frac{1}{p})\theta} \big) \big( \| \varphi_1 - \varphi_2 \|_{X_T^2} + \| q_1 - q_2 \|_{X_T^3} \big),
\end{align*}
where $\theta=\min\{\theta_1,\theta_2\}$.

For the estimate of the last term in $\mathcal{F}_2$ in the space $L^p(0, T; H_{(0)}^{-1}(\Omega))$, since
$ \varphi, q \in BUC([0, T]; H^s(\Omega))$, and $
s > \frac{d}{2}$, with $ p > 4 \text{ and } \varepsilon \text{ small enough,}$ we get the estimate with the help of product estimate in Lemma \ref{lem:product_estimate_dual} that
\begin{align*}
    &\left\| -\Delta_N \sum_{i=1,2} \left[ \xi'_i(\varphi_1) G_c^{(i)}(q_1) - \xi'_i(\varphi_2) G_c^{(i)}(q_2) + \xi'_i(\varphi_0) c^{(i)}(q_0) (q_1 - q_2) \right] \right\|_{L^p(0, T; H_{(0)}^{-1}(\Omega))} \\
    &\quad \leq C \left\| \sum_{i=1,2} \left[ \xi'_i(\varphi_1) G_c^{(i)}(q_1) - \xi'_i(\varphi_2) G_c^{(i)}(q_2) + \xi'_i(\varphi_0) c^{(i)}(q_0) (q_1 - q_2) \right] \right\|_{L^p(0, T; H^{1}(\Omega))} \\
    &\quad \leq C \sum_{i=1,2} \left\| (\xi'_i(\varphi_1) - \xi'_i(\varphi_2)) G_c^{(i)}(q_1) \right\|_{L^p(0, T; H^{1}(\Omega))} + C \sum_{i=1,2} \left\| \xi'_i(\varphi_2) (G_c^{(i)}(q_1) - G_c^{(i)}(q_2)) \right\|_{L^p(0, T; H^{1}(\Omega))} \\
    &\qquad + C \sum_{i=1,2} \left\| \xi'_i(\varphi_0) c^{(i)}(q_0) (q_1 - q_2) \right\|_{L^p(0, T; H^{ 1}(\Omega))} \\
    &\quad \leq C T^{\frac{1}{p}} \| \xi'_i(\varphi_1) - \xi'_i(\varphi_2) \|_{L^\infty(0, T; H^s(\Omega))} \| G_c^{(i)}(q_1) \|_{L^\infty(0, T; H^1(\Omega))} \\
    &\qquad + C T^{\frac{1}{p}} \| \xi'_i(\varphi_2) \|_{L^\infty(0, T; H^s(\Omega))} \| G_c^{(i)}(q_1) - G_c^{(i)}(q_2) \|_{L^\infty(0, T; H^1(\Omega))} \\
    &\qquad + C T^{\frac{1}{p}} \| \xi'_i(\varphi_0) \|_{H^s(\Omega)} \| c^{(i)}(q_0) \|_{H^1(\Omega)} \| q_1 - q_2 \|_{L^\infty(0, T; H^1(\Omega))} \\
    &\quad \leq C R T^{\frac{1}{p}} \left( \| \varphi_1 - \varphi_2 \|_{L^\infty(0, T; H^s(\Omega))} + \| q_1 - q_2 \|_{L^\infty(0, T; H^1(\Omega))} \right) \\
    &\quad \leq C R T^{\frac{1}{p}} \left( \| \varphi_1 - \varphi_2 \|_{X_T^2} + \| q_1 - q_2 \|_{X_T^3} \right).
\end{align*}

Thus, the requisite Lipschitz estimate for the reaction terms is fully established. This completes the estimate of $ \mathcal{F}_2$ that
\begin{equation*}
	\| \mathcal{F}_2(\ve_1,\varphi_1, q_1) - \mathcal{F}_2(\ve_2,\varphi_2, q_2) \|_{Y_T^2} \le C R \big(T^{(1-\frac{1}{p})\theta} + T^{\frac{1}{p}}\big) \big( \| \ve_1 - \ve_2 \|_{X_T^1}+\| \varphi_1 - \varphi_2 \|_{X_T^2} + \| q_1 - q_2 \|_{X_T^3} \big).
\end{equation*}
 
\rev{Next, we turn to the estimation of the effective source term $\mathcal{F}_3$.} The target space for this residual is the base Lebesgue space $Y_T^3 := L^p(0,T; L^2(\Omega))$. Recall that:
\begin{align*}
    \mathcal{F}_3(\mathbf{v}, \varphi, q) &= - \alpha(q, \varphi, \nabla \varphi) \partial_t q + \alpha(q_0, \varphi_0, \nabla \varphi_0) \partial_t q - W'(\varphi) c_\Gamma(q) \partial_t \varphi- \mathbf{v} \cdot \nabla \beta(q, \varphi, \nabla \varphi)\\
    &\quad + \Div \big[ \tilde{m}(\varphi_0, q_0, \nabla \varphi_0) \nabla q \big] - \Div \big[ \tilde{m}(\varphi, q, \nabla \varphi) \nabla q \big] + \big( c_\Gamma(q_0) \nabla \varphi_0 - c_\Gamma(q) \nabla \varphi \big) \cdot \partial_t \nabla \varphi.
\end{align*}

For the first component of $\mathcal{F}_3$, applying the triangle inequality and noting the continuous embedding $H^s(\Omega) \hookrightarrow C^0(\bar{\Omega})$ for $s > d/2$, we exploit the temporal Hölder continuity of the coefficient $\alpha(q, \varphi, \nabla \varphi)\in C^{0, (1-\frac{1}{p})\theta}([0,T]; H^s(\Omega))$ relative to its initial data. This allows us to extract the contraction factor $T^{(1-\frac{1}{p})\theta}$:
\begin{align*}
    &\big\| \alpha(q_0, \varphi_0, \nabla \varphi_0) \partial_t(q_1 - q_2) - \big(\alpha(q_1, \varphi_1, \nabla \varphi_1)\partial_t q_1 - \alpha(q_2, \varphi_2, \nabla \varphi_2)\partial_t q_2\big) \big\|_{L^p(0,T; L^2(\Omega))} \\
    &\quad \le C \big\| \big(\alpha(q_0, \varphi_0, \nabla \varphi_0) - \alpha(q_1, \varphi_1, \nabla \varphi_1)\big) \partial_t(q_1 - q_2) \big\|_{L^p(0,T; L^2(\Omega))} \\
    &\qquad + C \big\| \big(\alpha(q_1, \varphi_1, \nabla \varphi_1) - \alpha(q_2, \varphi_2, \nabla \varphi_2)\big) \partial_t q_2 \big\|_{L^p(0,T; L^2(\Omega))} \\
    &\quad \le C \sup_{t \in [0,T]} \| \alpha(q_0, \varphi_0, \nabla \varphi_0) - \alpha(q_1, \varphi_1, \nabla \varphi_1) \|_{C^0(\bar{\Omega})} \| \partial_t(q_1 - q_2) \|_{L^p(0,T; L^2(\Omega))} \\
    &\qquad + C \sup_{t \in [0,T]} \| \alpha(q_1, \varphi_1, \nabla \varphi_1) - \alpha(q_2, \varphi_2, \nabla \varphi_2) \|_{C^0(\bar{\Omega})} \| \partial_t q_2 \|_{L^p(0,T; L^2(\Omega))} \\
    &\quad \le C R T^{(1-\frac{1}{p})\theta} \big( \| \varphi_1 - \varphi_2 \|_{X_T^2} + \| q_1 - q_2 \|_{X_T^3} \big).
\end{align*}

For the non-principal term that involves the time derivative of the phase field, we define the multiplier $a_3(\varphi, q) = W'(\varphi)c_\Gamma(q)$. Using the continuous embedding $H^s(\Omega) \hookrightarrow C^0(\overline{\Omega})$ for $s > d/2$, and Lemma \ref{lem:Runst_Sickel}, we have:
For the non-principal term that involves the time derivative of the phase field, we define the multiplier $a_3(\varphi, q) = W'(\varphi)c_\Gamma(q)$. Using the continuous embedding $H^s(\Omega) \hookrightarrow C^0(\overline{\Omega})$ for $s > d/2$, and Lemma 2.4, we have:
\begin{align*}
    &\|a_3(\varphi_0, q_0)\partial_t(\varphi_1 - \varphi_2) - (a_3(\varphi_1, q_1)\partial_t\varphi_1 - a_3(\varphi_2, q_2)\partial_t\varphi_2)\|_{L^p(0,T; L^2(\Omega))} \\
    &\le \|(a_3(\varphi_0, q_0) - a_3(\varphi_1, q_1))\rev{\partial_t(\varphi_1 - \varphi_2)}\|_{L^p(0,T; L^2(\Omega))} + \|(a_3(\varphi_1, q_1) - a_3(\varphi_2, q_2))\partial_t\varphi_2\|_{L^p(0,T; L^2(\Omega))} \\
    &\le C \sup_{t\in[0,T]} \|(\varphi_1, q_1) - (\varphi_0, q_0)\|_{C^0(\overline{\Omega})} \|\rev{\partial_t(\varphi_1 - \varphi_2)}\|_{L^p(0,T; L^2(\Omega))} \\
    &\quad + C \sup_{t\in[0,T]} \|(\varphi_1, q_1) - (\varphi_2, q_2)\|_{C^0(\overline{\Omega})} \|\partial_t\varphi_2\|_{L^p(0,T; L^2(\Omega))} \\
    &\le C R T^{(1-\frac{1}{p})\theta} (\|\varphi_1 - \varphi_2\|_{X_T^2} + \|q_1 - q_2\|_{X_T^3}).
\end{align*}

Analogously, for the gradient-coupled term, we set $a_4(\nabla\varphi, q) = c_\Gamma(q)\nabla \varphi$. Exploiting the temporal Hölder regularity $\nabla \varphi \in C^{0, (1-\frac{1}{p})\theta}([0,T]; H^s(\Omega))$, we obtain:
\begin{align*} 
	&\big\| a_4(\nabla \varphi_0, q_0)(\partial_t \nabla \varphi_1 - \partial_t \nabla \varphi_2) - \big(a_4(\nabla \varphi_1, q_1)\partial_t \nabla \varphi_1 - a_4(\nabla \varphi_2, q_2)\partial_t \nabla \varphi_2\big) \big\|_{L^p(0,T; L^2(\Omega))} \\
	&\quad \le C \sup_{t \in [0,T]} \| (\nabla \varphi_0, q_0) - (\nabla \varphi_1, q_1) \|_{H^s(\Omega)} \| \rev{\partial_t \nabla (\varphi_1 - \varphi_2)} \|_{L^p(0,T; L^2(\Omega))} \\
	&\qquad + C \sup_{t \in [0,T]} \| (\nabla \varphi_1, q_1) - (\nabla \varphi_2, q_2) \|_{H^s(\Omega)} \| \partial_t \nabla \varphi_2 \|_{L^p(0,T; L^2(\Omega))} \\
	&\quad \le C R T^{(1-\frac{1}{p})\theta} \big( \| \varphi_1 - \varphi_2 \|_{X_T^2} + \| q_1 - q_2 \|_{X_T^3} \big).
\end{align*}

Finally, for the strongly coupled diffusion term, we bound the spatial divergence by estimating the inner flux within the space $L^p(0,T; H^1(\Omega))$. Relying on the multiplier algebra property of $H^s(\Omega)$, we deduce:
\begin{align*}
	&\big\| \Div \big[ \tilde{m}(\varphi_1, q_1, \nabla \varphi_1) \nabla q_1 - \tilde{m}(\varphi_2, q_2, \nabla \varphi_2) \nabla q_2 - \tilde{m}(\varphi_0, q_0, \nabla \varphi_0) \nabla (q_1 - q_2) \big] \big\|_{L^p(0,T; L^2(\Omega))} \\
	&\quad \le C \big\| \big(\tilde{m}(\varphi_1, q_1, \nabla \varphi_1) - \tilde{m}(\varphi_2, q_2, \nabla \varphi_2)\big) \nabla q_1 \big\|_{L^p(0,T; H^1(\Omega))} \\
	&\qquad + C \big\| \big(\tilde{m}(\varphi_2, q_2, \nabla \varphi_2) - \tilde{m}(\varphi_0, q_0, \nabla \varphi_0)\big) \nabla (q_1 - q_2) \big\|_{L^p(0,T; H^1(\Omega))} \\
	&\quad \le C \big \| \nabla q_1  \big\|_{L^p(0,T; H^1(\Omega))}  \big\| \tilde{m}(\varphi_1, q_1, \nabla \varphi_1) - \tilde{m}(\varphi_2, q_2, \nabla \varphi_2)  \big\|_{C^0([0,T]; H^s(\Omega))} \\
	&\qquad + C  \big\| \nabla q_1 - \nabla q_2  \big\|_{L^p(0,T; H^1(\Omega))}  \big\| \tilde{m}(\varphi_2, q_2, \nabla \varphi_2) - \tilde{m}(\varphi_0, q_0, \nabla \varphi_0) \big \|_{C^0([0,T]; H^s(\Omega))} \\
	&\quad \le C R T^{(1-\frac{1}{p})\theta} \| (\varphi_1, q_1, \nabla \varphi_1) - (\varphi_2, q_2, \nabla \varphi_2) \|_{C^{0, (1-\frac{1}{p})\theta}([0,T]; H^s(\Omega))} \\
	&\qquad + C T^{(1-\frac{1}{p})\theta} \| (\varphi_2, q_2, \nabla \varphi_2) - (\varphi_0, q_0, \nabla \varphi_0) \|_{C^{0, (1-\frac{1}{p})\theta}([0,T]; H^s(\Omega))} \| q_1 - q_2 \|_{L^p(0,T; H^2(\Omega))} \\
	&\quad \le C R T^{(1-\frac{1}{p})\theta} \big( \| \varphi_1 - \varphi_2 \|_{X_T^2} + \| q_1 - q_2 \|_{X_T^3} \big).
\end{align*}

Recall $\beta(q, \varphi, \nabla \varphi)=\xi_1(\varphi) c^{(1)}(q) + \xi_2(\varphi) c^{(2)}(q) + \delta(\varphi, \nabla \varphi) c_   \Gamma(q)$. Since $\xi_i, c^{(i)}, c_\Gamma \in C^3(\mathbb{R})$, $\ve, q, \varphi, \nabla \varphi \in BUC([0,T]; H^{s_0}(\Omega)),$ and $\delta(\varphi, \nabla \varphi) \in C([0,T]; H^{s_0}(\Omega))$ with $s_0 > \frac{d}{2}$, these give $$\beta(\varphi, \nabla \varphi, q) \in BUC([0,T]; H^{s_0}(\Omega)).$$
Then, we have $\nabla \beta(\varphi, \nabla \varphi, q) \in C([0,T]; H^{s_0-1}(\Omega)),$ and further obtain
\begin{align*}
	&\| \mathbf{v}_1 \cdot \nabla \beta(\varphi_1, \nabla \varphi_1, q_1) - \mathbf{v}_2 \cdot \nabla \beta(\varphi_2, \nabla \varphi_2, q_2) \|_{L^p(0,T;L^2(\Omega))}\\  
	&\quad \le C \| \mathbf{v}_1 - \mathbf{v}_2 \|_{L^\infty(0,T; C^0(\overline{\Omega}))} \| \nabla \beta(\varphi_1, \nabla \varphi_1, q_1) \|_{L^\infty(0,T;L^2(\Omega))} T^{\frac{1}{p}} \\
	&\qquad + C \| \mathbf{v}_2 \|_{L^\infty(0,T; C^0(\overline{\Omega}))} \| \nabla \beta(\varphi_1, \nabla \varphi_1, q_1) - \nabla \beta(\varphi_2, \nabla \varphi_2, q_2) \|_{L^\infty(0,T;L^2(\Omega))} T^{\frac{1}{p}} \\
	&\quad \le C R T^{\frac{1}{p}} \big( \| \mathbf{v}_1 - \mathbf{v}_2 \|_{X_T^1} + \| \varphi_1 - \varphi_2 \|_{X_T^2} + \| q_1 - q_2 \|_{X_T^3} \big).
\end{align*}
Hence, we have
\begin{equation*}
	\| \mathcal{F}_3(\ve_1,\varphi_1, q_1) - \mathcal{F}_3(\ve_2,\varphi_2, q_2) \|_{Y_T^3} \le C R ( T^{\frac{1}{p}}+T^{(1-\frac{1}{p})\theta} )\big(  \| \mathbf{v}_1 - \mathbf{v}_2 \|_{X_T^1}+\| \varphi_1 - \varphi_2 \|_{X_T^2} + \| q_1 - q_2 \|_{X_T^3} \big).
\end{equation*}
Combining the estimates of $\mathcal{F}_1$, $\mathcal{F}_2$ and $\mathcal{F}_3$, we complete the proof of Proposition \ref{F}.     
\end{proof}
\subsection{Linearized System}
Considering the full linearized system \eqref{eq:full_linearized_system}, we note a crucial structural property: the principal linear operator of the Stokes equations decouples from the phase-field and surfactant variables, whereas the Cahn-Hilliard equation and the surfactant equation are connected via high-order coupling terms. Motivated by this underlying triangular structure, we adopt a sequential analytic strategy. 
\rev{The remainder of this section is organized as follows. Since the maximal regularity for the decoupled Stokes equations is already established in Theorem \ref{thm:stokes_unified} in Section \ref{sec:Notation}, our analysis focuses on obtaining the maximal regularity for  $(\varphi, q)$ subsystem. To this end,} we first perform the resolvent estimates for this subsystem in Subsection \ref{subsec:resolvent_estimate}. Finally, in Subsection \ref{subsec:extension_lifting}, we extend the results to accommodate non-zero initial data via a lifting argument, thereby completing the proof of maximal regularity for the full linearized system.

 \subsubsection{Resolvent estimate}\label{subsec:resolvent_estimate}
In this section, we aim to show that the \rev{linear system for $(\varphi, q)$ admits a unique strong solution} exactly in our defined maximal regularity space:
\begin{equation*}
	\begin{aligned}
		\varphi &\in W^1_p(0,T; H^1(\Omega)) \cap L^p(0,T; H_N^3(\Omega)) \cap W^2_p(0,T; H_N^{-3}(\Omega)), \\
		q &\in W^1_p(0,T; L^2(\Omega)) \cap L^p(0,T; H_N^2(\Omega)).
	\end{aligned}
\end{equation*}

For any $T\in(0,\infty)$, \rev{we define the shifted principal linear operator $\tilde{\mathcal{L}}_T$ as}
\begin{equation*}
	\begin{aligned}
\tilde{\mathcal{L}}_T \begin{pmatrix} \varphi \\ q \end{pmatrix} = \begin{bmatrix} &(\omega + \partial_t) \varphi + \Delta_N ( \Div_N (\tilde{\sigma}(q_0) \nabla \varphi) ) - \Delta_N(\Div_N (c_{\Gamma}(q_0) q \nabla \varphi_0)) \\
	&+ \Delta_N \Big( \tilde{\sigma}'(q_0)q W'(\varphi_0 ) + \sum_{i=1,2} \xi_i'(\varphi_0) c^{(i)}(q_0) q \Big) \\ &\alpha(\varphi_0, \nabla \varphi_0, q_0) (\omega + \partial_t) q + c_\Gamma(q_0)\nabla(\omega + \partial_t)\varphi \cdot \nabla\varphi_0 - \Div_N \big( \tilde{m}(\varphi_0, q_0, \nabla \varphi_0) \nabla q \big). \end{bmatrix} 
	\end{aligned}
\end{equation*}
\rev{Here, the spectral shift $\omega > 0$ is to ensure the invertibility of this linear operator.} \rev{Later, for the subsequent nonlinear contraction mapping argument, the full equations will be rewritten as
\begin{equation}\label{linear}
	\tilde{\mathcal{L}}_T \begin{pmatrix} \varphi \\ q \end{pmatrix} = \tilde{\mathcal{F}}_T (\varphi , q ),
\end{equation}
where the nonlinear right-hand side is given by
$$ 
\tilde{\mathcal{F}}_T(\varphi, q) = (\mathcal{F}_T^2 + \omega \varphi, \mathcal{F}_T^3 + \alpha(\varphi_0, \nabla \varphi_0, q_0)\omega q +  c_{\Gamma}(q_0)\omega \nabla \varphi \cdot \nabla \varphi_0)^T. 
$$
} The resulting additional terms in $\mathcal{F}_T$ are lower-order linear perturbations which are easily absorbed as generic source terms in the linear resolvent estimates and satisfy the local Lipschitz conditions required for the subsequent nonlinear contraction mapping argument. For the resolvent estimate, we focus on the linear operator $\tilde{\mathcal{L}}_T$ with generic source terms $(f, g) \in H_N^{-3}(\Omega) \times L^2(\Omega)$. Provided that $f$ naturally satisfies the zero-mean compatibility condition $f \in L_{(0)}^2(\Omega)$, the shifted system reads:
\begin{equation}\label{fg-equtions}
	\tilde{\mathcal{L}}_T (\varphi,q)^T=(f,g)^T
\end{equation}
To be consistent with the initial data requirements in our main well-posedness framework, we state the assumptions on \rev{$(\varphi_0, q_0) \in  {}_NB_{2,p}^{3-\frac{2}{p}}(\Omega) \times  {}_NB_{2,p}^{2-\frac{2}{p}}(\Omega), L^2(\Omega))_{1-\frac{1}{p},p}$} with $|\varphi_0(x)| \le 1$. 

Applying the Fourier transform and setting $\lambda = \omega + i\tau$ with $\omega > 0$, where $\tau \in \mathbb{R}$, we obtain the abstract resolvent problem:
\begin{equation}\label{resovent equtions}
	\begin{aligned}
		\lambda \hat{\varphi} + \Delta_N \Big[ \text{div}_N  (\tilde{\sigma}(q_0) \nabla \hat{\varphi}) - \text{div}_N (c_{\Gamma}(q_0) \hat{q} \nabla \varphi_0) + \Big( \tilde{\sigma}'(q_0) W'(\varphi_0) + \sum_{i=1,2} \xi_i'(\varphi_0) c^{(i)}(q_0) \Big) \hat{q} \Big] &= \hat{f}(\tau), \\
		\alpha(\varphi_0, \nabla \varphi_0, q_0) \lambda \hat{q} + c_{\Gamma}(q_0) \nabla \lambda \hat{\varphi} \cdot \nabla \varphi_0 - \text{div}_N \big( \tilde{m}(\varphi_0, q_0, \nabla \varphi_0) \nabla \hat{q} \big) &= \hat{g}(\tau).
	\end{aligned} 
\end{equation}
Let $\overline{\hat{\varphi}}$ and $\overline{\hat{q}}$ denote the complex conjugates of the Fourier transformed variables $\hat{\varphi}$ and $\hat{q}$, respectively. Integrating $\eqref{resovent equtions}_1$ over $\Omega$ yields $\hat{\varphi} \in L_{(0)}^2(\Omega)$ due to the Neumann boundary conditions and the zero-mean property of the source term. Thus, by multiplying equation $\eqref{resovent equtions}_1$ by $\lambda$ and testing the result with $(-\Delta_N)^{-1} \bar{\lambda} \overline{\hat{\varphi}}$, while simultaneously testing equation $\eqref{resovent equtions}_2$ with $\bar{\lambda} \overline{\hat{q}}$, we obtain
\begin{equation} \label{phi-r-e}
	\begin{aligned}
		&\lambda |\lambda|^2 \|\hat{\varphi}\|_{H_{(0)}^{-1}(\Omega)}^2 + |\lambda|^2 \int_\Omega \tilde{\sigma}(q_0) |\nabla \hat{\varphi}|^2 dx - |\lambda|^2 \int_\Omega  c_\Gamma(q_0) \hat{q}  \nabla \varphi_0\cdot \nabla \overline{\hat{\varphi}} dx\\ 
		&-|\lambda|^2 \int_\Omega \Big( \tilde{\sigma}'(q_0) W'(\varphi_0) + \sum_{i=1,2} \xi_i'(\varphi_0) c^{(i)}(q_0) \Big) \hat{q}  \overline{\hat{\varphi}} dx
		=  |\lambda|^2 \int_\Omega \hat{f} (-\Delta_N)^{-1} \overline{\hat{\varphi}} dx,
	\end{aligned}
\end{equation}
\begin{equation}\label{q-r-e}
	|\lambda|^2 \int_\Omega \alpha(q_0, \varphi_0, \nabla \varphi_0) |\hat{q}|^2 dx +  |\lambda|^2 \int_\Omega c_\Gamma(q_0) \nabla \hat{\varphi} \cdot \nabla \varphi_0 \cdot \overline{\hat{q}} dx + \bar{\lambda} \int_\Omega \tilde{m}(\varphi_0, q_0, \nabla \varphi_0) |\nabla \hat{q}|^2 dx = \bar{\lambda} \int_\Omega \hat{g} \cdot \overline{\hat{q}} dx. 
\end{equation}
where we have used \eqref{eq:weak_H-1} for the first term $\lambda |\lambda|^2 \|\hat{\varphi}\|_{H_{(0)}^{-1}(\Omega)}^2$ in \eqref{phi-r-e}. 
By virtue of the \rev{Sobolev embedding $H^{s_0}(\Omega) \hookrightarrow C^0(\overline{\Omega})$ with $s_0 := 2 - \frac{2}{p} - \varepsilon$,} the coefficient $\alpha(q_0, \varphi_0, \nabla \varphi_0) \in H^{s_0}(\Omega)$ is bounded and uniformly continuous in space. Furthermore, Assumption \ref{asumm_1} guarantees that this coefficient is strictly bounded from below by a uniform constant $\alpha_0 > 0$:
\begin{equation*}
	\alpha(q_0, \varphi_0, \nabla\varphi_0) \ge \alpha_0 > 0 \quad \text{on } \overline{\Omega}.
\end{equation*}
Consequently, we obtain the estimate
\begin{equation}\label{al-estimates}
	|\lambda|^2 \int_\Omega \alpha(q_0, \varphi_0, \nabla \varphi_0) |\hat{q}|^2 dx \geq \alpha_0 |\lambda|^2 \|\hat{q}\|_{L^2(\Omega)}^2. 
\end{equation}
Similarly, there exist $\sigma_0 > 0$ such that $\tilde{\sigma}(q_0) \geq \sigma_0$, and 
\begin{equation}
	|\lambda|^2 \int_\Omega \tilde{\sigma}(q_0) |\nabla \hat{\varphi}|^2 dx \geq \sigma_0 |\lambda|^2 \|\nabla \hat{\varphi}\|_{L^2(\Omega)}^2.
\end{equation}
\rev{Since $\tilde{\sigma}, W, \xi_i, c^{(i)} \in C^2(\mathbb{R})$,} and the reference states possess the regularities $\varphi_0 \in H_N^3(\Omega)$ and $q_0 \in H_N^2(\Omega)$, the standard composition rules in Sobolev spaces guarantee that $$\tilde{\sigma}'(q_0)W'(\varphi_0) + \sum_{i=1,2} \xi_i'(\varphi_0) c^{(i)}(q_0) \in H^2(\Omega).$$
\rev{Exploiting} the Riesz isomorphism $\mathcal{R} = -\Delta_N: H^1_{(0)}(\Omega) \to H^{-1}_{(0)}(\Omega)$ and \eqref{eq:weak_H-1}, yields the interpolation inequality:
\begin{equation*}
	\|\hat{\varphi}\|_{L^2(\Omega)}^2 = \int_\Omega \nabla (-\Delta_N)^{-1} \hat{\varphi} \cdot \nabla \overline{\hat{\varphi}} \, dx \le \|\hat{\varphi}\|_{H^{-1}_{(0)}(\Omega)} \|\hat{\varphi}\|_{H_{(0)}^1(\Omega)}, 
\end{equation*}
for any $\hat{\varphi}\in H^{1}_{(0)}(\Omega)$. Substituting this interpolation bound into our estimate, and applying Young's inequality to decouple the product, we deduce:
\begin{equation} \label{est-perturbation}
	\begin{aligned}
		& \left| |\lambda|^2 \int_\Omega \Big( \tilde{\sigma}'(q_0) W'(\varphi_0) + \sum_{i=1,2} \xi_i'(\varphi_0) c^{(i)}(q_0) \Big) \hat{q} \overline{\hat{\varphi}} \, dx \right| \\
		&\quad \le  C |\lambda|^2 \Big[ \|\tilde{\sigma}'(q_0) W'(\varphi_0)\|_{L^\infty(\Omega)} + \sum_{i=1,2} \|\xi_i'(\varphi_0) c^{(i)}(q_0)\|_{L^\infty(\Omega)} \Big] \|\hat{q}\|_{L^2(\Omega)} \|\hat{\varphi}\|_{L^2(\Omega)} \\
		&\quad \le C |\lambda|^2 \|\hat{q}\|_{L^2(\Omega)} \|\hat{\varphi}\|_{L^2(\Omega)} \\
		&\quad \le \frac{\alpha_0}{4} |\lambda|^2 \|\hat{q}\|_{L^2(\Omega)}^2 + \frac{\sigma_0}{4} |\lambda|^2 \|\nabla\hat{\varphi}\|_{L^2(\Omega)}^2 + C_0 |\lambda|^2 \|\hat{\varphi}\|_{H_{(0)}^{-1}(\Omega)}^2,
	\end{aligned}
\end{equation}
where the constant $C_0$ depends only on $\alpha_0$ and $\sigma_0$.
Applying Hölder's and Young's inequalities to the right-hand side of \eqref{q-r-e}, we obtain
$$
	\left| \bar{\lambda} \int_\Omega \hat{g} \cdot \overline{\hat{q}} \, dx \right| \le |\lambda| \|\hat{g}\|_{L^2(\Omega)} \|\hat{q}\|_{L^2(\Omega)} \le C_0 \|\hat{g}\|_{L^2(\Omega)}^2 + \frac{\alpha_0}{2} |\lambda|^2 \|\hat{q}\|_{L^2(\Omega)}^2.
$$
To estimate the source term on the right-hand side of \eqref{phi-r-e}, we note that the inverse Neumann Laplacian acts as an isomorphism $(-\Delta_N)^{-1} : H_{(0)}^1(\Omega) \to H_N^3(\Omega)$, yielding the bound $\|(-\Delta_N)^{-1} \hat{\varphi}\|_{H_N^3(\Omega)} \le C \|\nabla \hat{\varphi}\|_{L^2(\Omega)}$. Utilizing this property alongside the natural duality pairing and Young's inequality, we deduce that
\begin{align*}
	\left| |\lambda|^2 \int_\Omega \hat{f} (-\Delta_N)^{-1} \overline{\hat{\varphi}} \, dx \right| 
	&\le |\lambda|^2 \|\hat{f}\|_{H_N^{-3}(\Omega)} \|(-\Delta_N)^{-1} \hat{\varphi}\|_{H_N^3(\Omega)} \\
	&\le \frac{\sigma_0}{2} |\lambda|^2 \|\nabla \hat{\varphi}\|_{L^2(\Omega)}^2 + C_0 |\lambda|^2 \|\hat{f}\|_{H_N^{-3}(\Omega)}^2.
\end{align*}

For the mixed-order coupling term, straightforward conjugation yields the algebraic identity
\begin{equation}\label{mix-term}
	- |\lambda|^2 \int_\Omega c^\Gamma(q_0) \hat{q} \nabla \varphi_0 \cdot \nabla \overline{\hat{\varphi}} dx = - \overline{|\lambda|^2 \int_\Omega c^\Gamma(q_0) \nabla \hat{\varphi} \cdot \nabla \varphi_0 \cdot \overline{\hat{q}} dx}.
\end{equation}
Recall $$\tilde{m}(\varphi_0, q_0, \nabla\varphi_0) = \sum_{i=1,2} M^{(i)}(c^{(i)}(q_0))\xi_i(\varphi_0) + M_\Gamma(c_\Gamma(q_0))\delta(\varphi_0, \nabla\varphi_0),$$ there exists a constant $m_0 > 0$ such that $\tilde{m}(\varphi_0, q_0, \nabla\varphi_0) \ge m_0 > 0$.
Now, we combine \eqref{phi-r-e}-\eqref{mix-term} and take the real part of both sides. By substituting the perturbation bound \eqref{est-perturbation} and absorbing the terms $\frac{\alpha_0}{4} |\lambda|^2 \|\hat{q}\|_{L^2(\Omega)}^2$ and $\frac{\sigma_0}{4} |\lambda|^2 \|\nabla \hat{\varphi}\|_{L^2(\Omega)}^2$ into the principal coercive parts, we arrive at:
\begin{align*}
	&\frac{\alpha_0}{4} |\lambda|^2 \|\hat{q}\|_{L^2(\Omega)}^2 + (\operatorname{Re} \lambda - C_0) |\lambda|^2 \|\hat{\varphi}\|_{H^{-1}_{(0)}(\Omega)}^2 + m_0\operatorname{Re} \lambda \|\hat{q}\|_{H^1(\Omega)}^2 + \frac{\sigma_0}{4} |\lambda|^2 \|\nabla \hat{\varphi}\|_{L^2(\Omega)}^2 \\
    &\quad\leq C \Big( \|\hat{g}\|_{L^2(\Omega)}^2 + |\lambda|^2 \|\hat{f}\|_{H_N^{-3}(\Omega)}^2 \Big).
\end{align*}
To ensure the uniform coercivity of the bilinear form, we restrict the resolvent parameter to a suitable right half-plane. Specifically, by choosing the real part to be sufficiently large, namely $\operatorname{Re} \lambda = \omega \ge 2C_0$, we guarantee the strict lower bound $\operatorname{Re} \lambda - C_0 \ge \frac{1}{2} \operatorname{Re} \lambda = \frac{\omega}{2} > 0$. Consequently, retaining all principal norm bounds on the left-hand side, we arrive at the fundamental a priori estimate:
\begin{align}\label{energy-inequality}
	&\frac{\alpha_0}{4} |\lambda|^2 \|\hat{q}\|_{L^2(\Omega)}^2 + \frac{1}{2} \operatorname{Re} \lambda |\lambda|^2 \|\hat{\varphi}\|_{H^{-1}_{(0)}(\Omega)}^2 + m_0\operatorname{Re} \lambda \|\hat{q}\|_{H^1(\Omega)}^2 + \frac{\sigma_0}{4} |\lambda|^2 \|\nabla \hat{\varphi}\|_{L^2(\Omega)}^2 \notag\\
	&\quad\leq C \Big( \|\hat{g}\|_{L^2(\Omega)}^2 + |\lambda|^2 \|\hat{f}\|_{H_N^{-3}(\Omega)}^2 \Big).
\end{align}

With this uniform a priori estimate established, we proceed to construct the associated resolvent operator.	
Firstly, let
\begin{equation}\label{inverst-eqas}
	(\hat{\varphi}, \hat{q})^T := R(\lambda) (\hat{f}, \hat{g})^T
\end{equation}
be the solution of \eqref{resovent equtions}.
The rigorous justification for the existence of the inverse operator $R(\lambda)$ relies on the generalized Lax-Milgram Lemma. We define the appropriate product Hilbert space $\mathbb{V} := H_{(0)}^1(\Omega) \times H^1(\Omega)$ and a bilinear form $B_\lambda(U, \Phi)$ over $\mathbb{V} \times \mathbb{V}$. Let $U = (\hat{\varphi}, \hat{q})^T$ and $\Phi = (\hat{\psi}, \hat{\eta})^T$. The bilinear form $B_\lambda(U, \Phi)$ is explicitly defined as:
\begin{align*}
B_\lambda(U, \Phi) = &\Big\langle \lambda^2 \hat{\varphi} + \lambda \Delta_N \big(\text{div}_N(\tilde{\sigma}(q_0)\nabla\hat{\varphi})\big) - \lambda \Delta_N \big(\text{div}_N(c_\Gamma(q_0)\hat{q}\nabla\varphi_0)\big) \\
&+ \lambda \Delta_N \Big( \big(\tilde{\sigma}'(q_0)W'(\varphi_0) + \sum_{i=1,2}\xi_i'(\varphi_0)c^{(i)}(q_0)\big)\hat{q} \Big), (-\Delta_N)^{-1}\overline{\hat{\psi}} \Big\rangle_{H_N^{-3}(\Omega), H_N^3(\Omega)} \\
&+ \Big( \alpha(\varphi_0, \nabla\varphi_0, q_0)\lambda\hat{q} + c_\Gamma(q_0)\nabla(\lambda\hat{\varphi}) \cdot \nabla\varphi_0, \overline{\hat{\eta}} \Big)_{L^2(\Omega)} +\Big( \tilde{m}(\varphi_0, q_0, \nabla\varphi_0)\nabla\hat{q}, \nabla\overline{\hat{\eta}} \Big)_{L^2(\Omega)}.
\end{align*}

First, we verify the continuity of the bilinear form. For any $U = (\hat{\varphi}, \hat{q})^T$ and $\Phi = (\hat{\psi}, \hat{\eta})^T$ in the product space $\mathbb{V}$, we apply standard energy estimates and Hölder's inequality to bound each term. For every given $\lambda \in \mathbb{C}$, applying the definition of the $H^{-1}_{(0)}$ inner product and performing standard integration by parts, we obtain the explicit expansion of the bilinear form:
\begin{align*}
	B_\lambda(U, \Phi) &= |\lambda|^2 ( \hat{\varphi}, \overline{\hat{\psi}} )_{H^{-1}_{(0)}(\Omega)} + \lambda\int_\Omega \tilde{\sigma}(q_0)\nabla\hat{\varphi} \cdot \nabla\overline{\hat{\psi}} \, dx \rev{-} \lambda\int_\Omega c_\Gamma(q_0)\hat{q}\nabla\varphi_0 \cdot \nabla\overline{\hat{\psi}} \, dx \\
	&\quad - \lambda\int_\Omega \Big( \tilde{\sigma}'(q_0) W'(\varphi_0) + \sum_{i=1,2}\xi_i'(\varphi_0) c^{(i)}(q_0) \Big) \hat{q} \overline{\hat{\psi}} \, dx + \lambda \int_\Omega \alpha(\varphi_0, \nabla\varphi_0, q_0) \hat{q}\overline{\hat{\eta}} \, dx \\
	&\quad + \lambda \int_\Omega c_\Gamma(q_0) (\nabla\hat{\varphi} \cdot \nabla\varphi_0) \overline{\hat{\eta}} \, dx + \int_\Omega \tilde{m}(\varphi_0, q_0, \nabla\varphi_0) \nabla\hat{q} \cdot \nabla\overline{\hat{\eta}} \, dx.
\end{align*}
Applying Cauchy-Schwarz and Hölder's inequalities, and utilizing the isometric identity $\|(-\Delta_N)^{-1}\overline{\hat{\psi}}\|_{H^{3}_{N}(\Omega)} \le C\|\hat{\psi}\|_{H^{1}_{(0)}(\Omega)}$ alongside the continuous embeddings $H^1_{(0)}(\Omega) \hookrightarrow L^2(\Omega) \hookrightarrow H^{-1}_{(0)}(\Omega) \hookrightarrow H^{-3}_N(\Omega)$, we bound each term to obtain: 
\begin{align*}
	|B_\lambda(U, \Phi)| &\le |\lambda|^2 \|\hat{\varphi}\|_{H^{1}_{(0)}(\Omega)} \|\overline{\hat{\psi}}\|_{H^{1}_{(0)}(\Omega)} + |\lambda|\|\tilde{\sigma}(q_0)\|_{L^\infty(\Omega)} \|\nabla\hat{\varphi}\|_{L^2(\Omega)} \|\nabla\overline{\hat{\psi}}\|_{L^2(\Omega)} \\
	&\quad + |\lambda|\|c_\Gamma(q_0)\nabla\varphi_0\|_{L^\infty(\Omega)} \|\hat{q}\|_{L^2(\Omega)} \|\nabla\overline{\hat{\psi}}\|_{L^2(\Omega)} \\
	&\quad + |\lambda| \Big\| \tilde{\sigma}'(q_0) W'(\varphi_0) + \sum_{i=1,2}\xi_i'(\varphi_0) c^{(i)}(q_0) \Big\|_{L^\infty(\Omega)} \|\hat{q}\|_{L^2(\Omega)} \|\overline{\hat{\psi}}\|_{L^2(\Omega)} \\
	&\quad + |\lambda| \|\alpha(\varphi_0, \nabla\varphi_0, q_0) \|_{L^\infty(\Omega)} \|\hat{q}\|_{L^2(\Omega)} \|\overline{\hat{\eta}}\|_{L^2(\Omega)} + |\lambda| \|c_\Gamma(q_0)\nabla\varphi_0\|_{L^\infty(\Omega)} \|\nabla\hat{\varphi}\|_{L^2(\Omega)} \|\overline{\hat{\eta}}\|_{L^2(\Omega)}  \\
	&\quad + \|\tilde{m}(\varphi_0, q_0, \nabla\varphi_0)\|_{L^\infty(\Omega)} \|\nabla\hat{q}\|_{L^2(\Omega)} \|\nabla\overline{\hat{\eta}}\|_{L^2(\Omega)}.
\end{align*}
Hence, there exists a constant $C > 0$ independent of $\lambda$, such that:
\begin{align*}
	|B_\lambda(U, \Phi)| &\le C (1 + |\lambda|+|\lambda|^2 ) \Big( \|\hat{\varphi}\|_{H^1_{(0)}(\Omega)} \|\overline{\hat{\psi}}\|_{H^1_{(0)}(\Omega)} + \|\hat{q}\|_{H^1(\Omega)} \|\overline{\hat{\psi}}\|_{H^1_{(0)}(\Omega)} \\&\quad\quad\quad\quad\quad\quad\quad\quad\quad+ \|\hat{q}\|_{H^1(\Omega)} \|\overline{\hat{\eta}}\|_{H^1(\Omega)} + \|\hat{\varphi}\|_{H^1_{(0)}(\Omega)} \|\overline{\hat{\eta}}\|_{H^1(\Omega)} \Big) \\
	&\le C (1 + |\lambda|+|\lambda|^2 )\big( \|\hat{\varphi}\|_{H^1_{(0)}(\Omega)} + \|\hat{q}\|_{H^1(\Omega)} \big) \big( \|\overline{\hat{\psi}}\|_{H^1_{(0)}(\Omega)} + \|\overline{\hat{\eta}}\|_{H^1(\Omega)} \big).
\end{align*}
Equipped with the product norm $\|U\|_{\mathbb{V}} = \|\hat{\varphi}\|_{H^1_{(0)}} + \|\hat{q}\|_{H^1}$, we conclude the continuity of the bilinear form:
\begin{equation}
	|B_\lambda(U, \Phi)| \le C (1 + |\lambda|+|\lambda|^2 ) \|U\|_{\mathbb{V}} \|\Phi\|_{\mathbb{V}} \quad \text{for all } U, \Phi \in \mathbb{V}.
\end{equation}

On the other hand, taking $\Phi = U$ and assuming $\operatorname{Re}\lambda = \omega > 0$ is chosen sufficiently large, we deduce from the definition of the bilinear form $B_\lambda(U,U)$ that there exist positive constants $c_1, c_2,$ and $\tilde{\alpha}$ such that the coercivity condition holds:
there are some constants $c_1, c_2$ and $\tilde{\alpha}$ such that the coercivity condition holds:
\begin{equation} \label{eq:coercivity}
	\text{Re} B_\lambda(U,U) \ge c_1 (\text{Re}\lambda) \|\hat{\varphi}\|_{H^1(\Omega)}^2 + c_2 (\text{Re}\lambda) \|\hat{q}\|_{H^1(\Omega)}^2 \ge \tilde{\alpha} \|U\|_{\mathbb{V}}^2.
\end{equation}

Thus, by the generalized Lax-Milgram Lemma, there exists a unique weak solution $U = (\hat{\varphi}, \hat{q}) \in H^1_{(0)}(\Omega) \times H^1(\Omega)$ to the resolvent equations \eqref{resovent equtions} for every $\lambda \in \mathbb{C}$ with $\operatorname{Re}\lambda = \omega$ sufficiently large. Serving as the foundation for strong solvability, this initial regularity is subsequently upgraded via an elliptic bootstrapping argument.

Given the reference states $(\varphi_0, q_0) \in H_N^3(\Omega) \times H_N^2(\Omega)$, Sobolev embeddings ensure that the coefficients $c_\Gamma(q_0)$ and $\alpha(\varphi_0, \nabla\varphi_0, q_0)$ belong to $H^2(\Omega) \hookrightarrow C^0(\overline{\Omega})$. Thus, we can rearrange the second equation in \eqref{resovent equtions} to isolate the principal diffusion term:
\begin{equation*}
	\text{div}_N \big( \tilde{m}(\varphi_0, q_0, \nabla\varphi_0) \nabla\hat{q} \big) = \lambda \alpha(\varphi_0, \nabla\varphi_0, q_0) \hat{q} + \lambda c_\Gamma(q_0) \nabla\hat{\varphi} \cdot \nabla\varphi_0 - \hat{g}.
\end{equation*}
Since the source term $\hat{g} \in L^2(\Omega)$ and the weak solution satisfies $(\hat{\varphi}, \hat{q}) \in H^1_{(0)}(\Omega) \times H^1(\Omega)$, the entire right-hand side is  in $L^2(\Omega)$. Combining this with $\tilde{m}\geq m_0>0$ and $\tilde{m}(\varphi_0, q_0, \nabla\varphi_0) \in  H^{s_0}(\Omega)$, standard elliptic regularity theory yields $\hat{q} \in H_N^2(\Omega)$.

Turning to the $\hat{\varphi}$ equation, we recall that the source term satisfies $\hat{f} \in H_{(0)}^{-1}(\Omega)$. Thanks to the reference states, the relevant composite coefficients that satisfy $$\tilde{\sigma}(q_0), \tilde{\sigma}'(q_0), W'(\varphi_0), \xi_i'(q_0), c^{(i)}(\varphi_0) \in H^2(\Omega) \hookrightarrow C^0(\overline{\Omega})$$
Coupled with the strict positivity condition that $\tilde{\sigma}(q_0)\geq \sigma_0 > 0$ in $\Omega$, we rearrange the first equation to isolate the principal spatial operator:
\begin{equation*}
	-\Delta_N \text{div}_N(\tilde{\sigma}(q_0)\nabla\hat{\varphi}) = \hat{f} - \lambda\hat{\varphi} + \Delta_N(\text{div}_N(c_\Gamma(q_0)\hat{q}\nabla\varphi_0)) - \Delta_N \Big( \big( \tilde{\sigma}'(q_0)W'(\varphi_0) + \sum_{i=1,2}\xi_i'(\varphi_0)c^{(i)}(q_0) \big) \hat{q} \Big).
\end{equation*}
Based on the obtained regularity $\hat{q} \in H_N^2(\Omega)$, the entire right-hand side is bounded in $H_{(0)}^{-1}(\Omega)$. Furthermore, by the divergence theorem and the homogeneous Neumann boundary conditions, the principal term $\text{div}_N(\tilde{\sigma}(q_0)\nabla\hat{\varphi})$ inherently possesses a zero spatial mean. Considering the isomorphism property of the restricted operator $\Delta_N: H_{(0)}^1(\Omega) \to H_{(0)}^{-1}(\Omega)$, we rigorously infer that
\begin{equation*}
	\rev{\tilde{f} := \text{div}_N(\tilde{\sigma}(q_0)\nabla\hat{\varphi}) \in H_{(0)}^1(\Omega).}
\end{equation*}
Expanding the divergence term and rearranging the equation to isolate the Laplacian:
\begin{equation*}
	\Delta \hat{\varphi} = \tilde{\sigma}(q_0)^{-1} \big( \tilde{f} - \nabla \tilde{\sigma}(q_0) \cdot \nabla \hat{\varphi} \big).
\end{equation*}
\rev{Since the coefficient satisfies $\tilde{\sigma}(q_0) \in H^2(\Omega)$, its gradient belongs to $H^1(\Omega)$. By virtue of the Sobolev embedding theorem $H^1(\Omega) \hookrightarrow L^6(\Omega)$ for dimensions $d =2, 3$ and Hölder's inequality, the product $\nabla \tilde{\sigma}(q_0) \cdot \nabla \hat{\varphi}$ is initially bounded in $L^3(\Omega)$. A standard bootstrap argument utilizing $L^p$-elliptic regularity then elevates the right-hand side to $H^1(\Omega)$. Consequently, standard elliptic regularity for the Poisson equation} immediately yields the lifted regularity $\hat{\varphi} \in H_N^3(\Omega)$.


As a conclusion, for any fixed parameter $\lambda \in \mathbb{C}$ with a sufficiently large real part $\operatorname{Re}\lambda = \omega$, the unique weak solution is elevated to $U = (\hat{\varphi}, \hat{q}) \in H_N^3(\Omega) \times H_N^2(\Omega)$. This confirms that the resolvent operator $R(\lambda): H_{(0)}^{-1}(\Omega) \times L^2(\Omega) \to H_N^3(\Omega) \times H_N^2(\Omega)$ is a rigorously well-defined and bounded linear operator. 

While the preceding Lax-Milgram framework successfully guarantees the existence and uniqueness of the solution to the resolvent equations, this well-posedness is confined to a fixed parameter $\lambda$, capturing the system at a frozen temporal state in the frequency domain. To transition from this static resolvent level back to the time-dependent dynamics of the original evolutionary system, we need the operator-valued Mikhlin Multiplier Theorem providing the exact theoretical foundation for this. To apply this theorem, we must verify the multiplier derivative conditions for the operator corresponding to the time derivative.
To establish maximal $L^p$-regularity via the Operator-Valued Mikhlin Multiplier Theorem (Lemma \ref{thm:mikhlin}), we introduce the multiplier symbol $m(\lambda)$ mapping the modified source $\tilde{F}:=(\lambda\hat{f}, \hat{g})^T$ to the scaled solution:
\begin{equation*}
	m(\lambda)\tilde{F} = (\lambda\hat{\varphi}, \lambda\hat{q})^T \in H_{(0)}^{1}(\Omega) \times L^2(\Omega).
\end{equation*}
Since the base space $H_{(0)}^1(\Omega) \times L^2(\Omega)$ is a Hilbert space, it inherently belongs to the class of $\mathcal{HT}$ spaces. In this classical setting, the $\mathcal{R}$-boundedness requirement conveniently reduces to uniform operator boundedness. Therefore, to ensure that the multiplier operator $m(D_t)$ is bounded from $L^p(\mathbb{R}; H_N^{-3}(\Omega) \times L^2(\Omega))$ to $L^p(\mathbb{R}; H_{(0)}^1(\Omega) \times L^2(\Omega))$ for any $1 < p < \infty$, it is sufficient to verify the classical uniform norm bounds:
\begin{align*}
 &\sup_{\tau \in \mathbb{R} \setminus \{0\}} \|m(\tau)\|_{\mathcal{L}(H_N^{-3}(\Omega) \times L^2(\Omega), H_{(0)}^1(\Omega) \times L^2(\Omega))} \le \kappa_0,\\  &\sup_{\tau \in \mathbb{R} \setminus \{0\}} \|\tau \partial_\tau m(\tau)\|_{\mathcal{L}(H_N^{-3}(\Omega) \times L^2(\Omega), H_{(0)}^1(\Omega) \times L^2(\Omega))} \le \kappa_1, 
\end{align*}
for some constants $\kappa_0, \kappa_1 > 0$.
Since $\lambda = \omega + i\tau$, the chain rule directly yields:
\begin{equation*}
	\frac{\partial}{\partial \tau} m(\lambda) = m'(\lambda) \frac{\partial (\omega + i\tau)}{\partial \tau} = i m'(\lambda).
\end{equation*}
Furthermore, since $|\tau| \le |\lambda| \le |\tau| + \omega$ for any $\tau \in \mathbb{R}$, establishing the classical Mikhlin conditions with respect to the real frequency $\tau$ is mathematically equivalent to verifying the bounds with respect to the complex parameter $\lambda \in \mathbb{C}$ restricted to $\operatorname{Re}\lambda = \omega$. That is, checking the uniform norm bounds reduces to proving the following spectral resolvent estimates:
\begin{equation} \label{eq:multiplier_lambda}
	\| m^{(k)}(\lambda) \|_{\mathcal{L}(H_N^{-3}(\Omega) \times L^2(\Omega), H_{(0)}^1(\Omega) \times L^2(\Omega))} \leq C |\lambda|^{-k}, 
\end{equation}
for $k = 0, 1$ and any $\lambda \in \mathbb{C}$ with $\operatorname{Re}\lambda = \omega$. Specifically, the abstract condition in \eqref{eq:multiplier_lambda} for the zeroth-order case ($k=0$) translates directly to bounding the scaled solution as follows:
\begin{equation*}
	\| (\lambda\hat{\varphi}, \lambda\hat{q})\|_{H_{(0)}^1(\Omega) \times L^2(\Omega)} \le C \|\tilde{F}\|_{H_N^{-3}(\Omega) \times L^2(\Omega)}.
\end{equation*}
Indeed, revisiting the fundamental a priori estimate established in \eqref{energy-inequality} and extracting the relevant non-negative terms on the left-hand side, we readily obtain:
\begin{equation*}
	\frac{\alpha_0}{4}|\lambda|^2 \|\hat{q}\|_{L^2(\Omega)}^2 + \frac{\sigma_0}{4}|\lambda|^2 \|\nabla \hat{\varphi}\|_{L^2(\Omega)}^2 \le C \Big( \|\hat{g}\|_{L^2(\Omega)}^2 + |\lambda|^2 \|\hat{f}\|_{H_N^{-3}(\Omega)}^2 \Big) = C \|\tilde{F}\|_{H_N^{-3}(\Omega) \times L^2(\Omega)}^2.
\end{equation*}
Equipping the zero-mean space $H_{(0)}^1(\Omega)$ with its standard gradient norm, taking the square root of both sides directly yields the desired zeroth-order resolvent bound. Consequently, the multiplier operator $m(\tau)$ is uniformly bounded by a constant $\kappa_0 > 0$ that is independent of the parameter $\lambda$.

To verify the first-order multiplier derivative bound, we must carefully account for the fact that the augmented data vector $\tilde{F} = (\lambda\hat{f}, \hat{g})^T$ depends on the spectral parameter $\lambda$. Recalling the multiplier definition $m(\lambda)(\lambda\hat{f}, \hat{g})^T = \lambda(\hat{\varphi}, \hat{q})^T = \lambda U$, we differentiate both sides with respect to $\lambda$. Applying the product rule yields:
\begin{equation} \label{eq:product_rule}
	(\partial_\lambda m(\lambda))(\lambda \hat{f}, \hat{g})^T + m(\lambda)(\hat{f}, 0)^T = (\hat{\varphi}, \hat{q})^T + \lambda(\partial_\lambda \hat{\varphi}, \partial_\lambda \hat{q})^T.
\end{equation}
Let $U_\lambda := \partial_\lambda U = (\partial_\lambda\hat{\varphi}, \partial_\lambda\hat{q})^T$. Rearranging \eqref{eq:product_rule} and multiplying by $\lambda$, we obtain the algebraic decomposition for the scaled multiplier derivative:
\begin{equation} \label{eq:operator_decomp}
	\lambda (\partial_\lambda m(\lambda)) \tilde{F} = \lambda U  - \lambda m(\lambda)(\hat{f}, 0)^T+ \lambda^2 U_\lambda.
\end{equation}

To establish $|\lambda| \| \partial_\lambda m(\lambda) \|_{\mathcal{L}(H_N^{-3}(\Omega) \times L^2(\Omega), H_{(0)}^1(\Omega) \times L^2(\Omega))} \le C$, we take the $H_{(0)}^1(\Omega) \times L^2(\Omega)$ norm of \eqref{eq:operator_decomp} and bound the three resulting terms on the right-hand side individually.
First, by the zero-order bound established previously, we have:
\begin{equation} \label{bound_1}
	\|\lambda U\|_{H_{(0)}^1(\Omega) \times L^2(\Omega)} \le C \|\tilde{F}\|_{H_N^{-3}(\Omega) \times L^2(\Omega)}.
\end{equation}
Secondly, we apply the zero-order boundedness of $m(\lambda)$ to the unscaled source component $(\hat{f}, 0)^T$:
\begin{equation} \label{bound_2}
	\begin{aligned}
	\big\| \lambda m(\lambda)(\hat{f}, 0)^T \big\|_{H_{(0)}^1(\Omega) \times L^2(\Omega)} &= |\lambda| \big\| m(\lambda)(\hat{f}, 0)^T \big\|_{H_{(0)}^1(\Omega) \times L^2(\Omega)} \le C |\lambda| \|(\hat{f}, 0)^T\|_{H_N^{-3}(\Omega) \times L^2(\Omega)} \\
	&= C \|\lambda \hat{f}\|_{H_N^{-3}(\Omega)} \le C \|\tilde{F}\|_{H_N^{-3}(\Omega) \times L^2(\Omega)}.	
\end{aligned}	
\end{equation}
Finally, it remains to bound the term $\|\lambda^2 U_\lambda\|_{H_{(0)}^1 \times L^2}$. 
Differentiating the resolvent system \eqref{resovent equtions} with respect to $\lambda$, we obtain:
\begin{equation} \label{differentiated_system_1}
	\begin{cases}
		\lambda \hat{\varphi}_\lambda + \Delta_N\big[ \Div_N(\tilde{\sigma}(q_0) \nabla \hat{\varphi}_\lambda) - \Div_N(c_\Gamma(q_0)\hat{q}_\lambda \nabla \varphi_0) \big]+ \Delta_N\Big[(\tilde{\sigma}'(q_0) W'(\varphi_0 )+\sum\limits_{i=1,2} \xi_i'(\varphi_0) c^{(i)}(q_0) )\hat{q}_\lambda\Big] = - \hat{\varphi}, \\
		\alpha \lambda \hat{q}_\lambda + c_\Gamma(q_0)\nabla(\lambda \hat{\varphi}_\lambda) \cdot \nabla \varphi_0 - \Div_N \big(\tilde{m}(\varphi_0, q_0, \nabla\varphi_0) \nabla \hat{q}_\lambda\big) = -  \alpha(\varphi_0, q_0, \nabla\varphi_0) \hat{q} - c_\Gamma(q_0)\nabla \hat{\varphi} \cdot \nabla \varphi_0.
	\end{cases}
\end{equation}
By the definition of the multiplier operator $m(\lambda)\tilde{F} = (\lambda \hat{\varphi}, \lambda \hat{q})^T$ for the scaled source $\tilde{F} := (\lambda \hat{f}, \hat{g})^T$, we can map the right-hand side of the differentiated system directly. We obtain the exact relation:
\begin{equation}\label{eq:3.21_corrected}
	\lambda U_\lambda = -m(\lambda) \begin{pmatrix} \lambda \hat{\varphi} \\ \alpha \hat{q} + c_\Gamma(q_0)\nabla\hat{\varphi} \cdot \nabla\varphi_0 \end{pmatrix}.
\end{equation}

To verify the multiplier derivative condition, we require bounds on $\lambda^2 U_\lambda$. Multiplying \eqref{eq:3.21_corrected} by $\lambda$ and applying the zero-order estimate yields:
\begin{equation}\label{eq:3.22_new}
	\|\lambda^2 U_\lambda\|_{H_{(0)}^1 \times L^2} \le C \left( \|\lambda^2 \hat{\varphi}\|_{H_N^{-3}} + \|\lambda \alpha \hat{q} + \lambda c_\Gamma(q_0)\nabla\hat{\varphi} \cdot \nabla\varphi_0\|_{L^2} \right).
\end{equation}
We now analyze the right-hand side of \eqref{eq:3.22_new}. For the second term, recalling that $\alpha$ and $c_\Gamma(q_0)\nabla\varphi_0$ are bounded in $L^\infty(\Omega)$, we have:
\begin{equation*}
	\|\lambda \alpha \hat{q} + \lambda c_\Gamma(q_0)\nabla\hat{\varphi} \cdot \nabla\varphi_0\|_{L^2} \le C \big( \|\lambda \hat{q}\|_{L^2} + \|\lambda \nabla \hat{\varphi}\|_{L^2} \big) \le C \|\lambda U\|_{H_{(0)}^1 \times L^2} \le C \|\tilde{F}\|_{H_N^{-3} \times L^2}.
\end{equation*}
For the critical term $\|\lambda^2 \hat{\varphi}\|_{H_N^{-3}}$, we extract $\lambda \hat{\varphi}$ directly from the first equation of the original resolvent system \eqref{resovent equtions}. Multiplying it by $\lambda$, we obtain:
\begin{equation*}
	\lambda^2 \hat{\varphi} = \lambda \hat{f} - \Delta_N \Big[\Div_N (\tilde{\sigma}(q_0) \lambda \nabla \hat{\varphi}) - \Div_N (c_\Gamma(q_0) \lambda \hat{q} \nabla \varphi_0\Big] - \Delta_N\Big[\big(\tilde{\sigma}'(q_0) W'(\varphi_0 )+\sum\limits_{i=1,2} \xi_i'(\varphi_0) c^{(i)}(q_0) \big)\lambda\hat{q}\Big].
\end{equation*}
Since the composite operator $\Delta_N \text{div}_N$ maps $L^2(\Omega)$ continuously into $H_N^{-3}(\Omega)$, we can take the $H_N^{-3}$ norm and utilize continuous embeddings to control this term:
\begin{equation*}
	\|\lambda^2 \hat{\varphi}\|_{H_N^{-3}} \le \|\lambda \hat{f}\|_{H_N^{-3}} + C \big( \|\lambda \nabla \hat{\varphi}\|_{L^2} + \|\lambda \hat{q}\|_{L^2} \big) \le C \|\tilde{F}\|_{H_N^{-3} \times L^2}.
\end{equation*}
Substituting these bounds back into \eqref{eq:3.22_new}, we synthesize the decisive result:
\begin{equation}\label{bound_3}
	\|\lambda^2 U_\lambda\|_{H_{(0)}^1 \times L^2} \le C \|\tilde{F}\|_{H_N^{-3} \times L^2}.
\end{equation} 
Substituting the bounds \eqref{bound_1}, \eqref{bound_2}, and \eqref{bound_3} back into the decomposition \eqref{eq:operator_decomp}, we obtain
\begin{equation}
	\big\| \lambda (\partial_\lambda m(\lambda)) \tilde{F} \big\|_{H_{(0)}^1 \times L^2} \le C \|\tilde{F}\|_{H_N^{-3} \times L^2}.
\end{equation}
This verifies the first-order multiplier derivative bound $$\sup_{\tau \in \mathbb{R} \setminus \{0\}} \|\tau \partial_\tau m(\tau)\|_{\mathcal{L}(H_N^{-3}(\Omega) \times L^2(\Omega), H_{(0)}^1(\Omega) \times L^2(\Omega))} \le \kappa_1, 
$$ 
satisfying the conditions of the operator-valued Mikhlin Multiplier Theorem. Noting that the Mikhlin Multiplier Theorem is formulated on the whole real line $\mathbb{R}$, we first introduce the zero-extension for the source data $(f, g)^T$ originally defined on $(0, \infty)$:
\begin{equation*}
	(\tilde{f}, \tilde{g})^T(t) := 
	\begin{cases} 
		(f(t), g(t))^T, & t > 0, \\ 
		(0, 0)^T, & t \le 0. 
	\end{cases}
\end{equation*}
Recall that the complex spectral parameter is given by $\lambda = \omega + i\tau$. As rigorously verified by integration by parts, the inverse Fourier transform of $\lambda(\hat{\varphi}, \hat{q})^T$ explicitly yields the shifted differential operator:
\begin{equation}\label{inverse F}
	\mathcal{F}^{-1}\big[ (\omega + i\tau) \hat{\varphi} \big] = \partial_t \varphi + \omega \varphi, \quad \text{and} \quad \mathcal{F}^{-1}\big[ (\omega + i\tau) \hat{q} \big] = \partial_t q + \omega q.
\end{equation}
We define the corresponding time-domain solution operator $\mathcal{T}_m$ by
\begin{equation} \label{eq:operator_Tm}
	\mathcal{T}_m \binom{(\partial_t + \omega)\tilde{f}}{\tilde{g} }:= \mathcal{F}^{-1} \left[ m(\omega + i\tau) \begin{pmatrix} (\omega + i\tau)\mathcal{F}(\tilde{f}) \\ \mathcal{F}(\tilde{g}) \end{pmatrix} \right].
\end{equation}
By applying the Mikhlin Multiplier Theorem on the whole real line $\mathbb{R}$ to the zero-extended source data, the associated time-domain operator $\mathcal{T}_m$ yields a global solution on $\mathbb{R}$ that formally satisfies the shifted evolution system for $t > 0$:
\begin{equation} \label{eq:shifted_system}
\begin{aligned}
	&\begin{bmatrix}
		&(\omega + \partial_t)\varphi + \Delta_N\Div_N\big(\tilde{\sigma}(q_0)\nabla\varphi-c_\Gamma(q_0)q\nabla\varphi_0\big) + \Delta_N\Big(\big(\tilde{\sigma}'(q_0)W'(\varphi_0) + \sum_{i=1,2}\xi_i'(\varphi_0)c^{(i)}(q_0)\big) q \Big) \\
		&\alpha(\varphi_0, \nabla\varphi_0, q_0)(\omega + \partial_t)q + c_\Gamma(q_0)\nabla(\omega + \partial_t)\varphi \cdot \nabla\varphi_0 -\Div_N \big(\tilde{m}(\varphi_0, q_0, \nabla\varphi_0) \nabla q\big)\\
	\end{bmatrix}\\
	&= (f,g)^T .    
\end{aligned}
\end{equation}
This holds for all source data $(f, g)^T \in C_0^\infty((0, \infty); H_N^{-3}(\Omega)) \times C_0^\infty((0, \infty); L^2(\Omega))$ provided $\omega$ is chosen sufficiently large. Given dense embeddings 
\begin{equation*}
C_0^\infty((0, \infty); L^2(\Omega)) \stackrel{d}{\hookrightarrow} L^p((0, \infty); L^2(\Omega)), \quad C_0^\infty((0, \infty); H_N^{-3}(\Omega)) \stackrel{d}{\hookrightarrow} W_{p,0}^1((0, \infty); H_N^{-3}(\Omega)
\end{equation*}
Crucially, for the zero-extended input $\big((\partial_t + \omega)\tilde{f}, \tilde{g}\big)^T$ to legitimately fall within this valid $L^p(\mathbb{R})$ domain, its distributional derivative must not contain any Dirac delta singularities at the origin $t=0$. This forces the original physical source $f$ to possess a vanishing initial trace, thereby naturally restricting $f$ to the space $W_{p,0}^1((0, \infty); H_N^{-3}(\Omega))$. Hence,we can define the solution operator $\mathcal{T}_m$ to the Banach spaces over the half-line $(0, \infty)$ as
\begin{equation} \label{eq:multiplier_operator}
	\mathcal{T}_m :W_{p,0}^1((0, \infty); H_N^{-3}(\Omega))  \times L^p((0, \infty); L^2(\Omega)) \to L^p(\R; H_{(0)}^1(\Omega)) \times L^p(\R; L^2(\Omega)).
\end{equation}

To restrict this global result back to the finite interval $(0, T)$,  we first assume $\varphi_0 = 0=q_0 $. Now, we extend the real frequency parameter $\tau$ into the complex plane by introducing the shift $z = \tau + iy$ for $y \in \mathbb{R}$. For source data $(f, g)^T \in C_0^\infty((0, \infty); H_N^{-3}(\Omega)) \times C_0^\infty((0, \infty); L^2(\Omega))$,  and any $y \le 0$, the Fourier transform of $f$ is given by
\begin{equation} \label{eq:analytic_f}
	\hat{f}(\tau + iy) = \int_0^\infty f(t) e^{-i(\tau + iy)t} dt = \int_0^\infty f(t) e^{-i\tau t} e^{yt} dt.
\end{equation}
Notice that for all $t \ge 0$ and $y \le 0$, $
	|e^{-i\tau t} e^{yt}| = e^{yt} \le 1.
$
This exponential damping ensures that the integral in \eqref{eq:analytic_f} is absolutely convergent. This $\hat{f}(z)$ admits an analytic extension to the half-plane $\{z \in \mathbb{C} : \operatorname{Im } z < 0\}$.

In addition, to establish the upper  boundedness required by the Paley-Wiener theorem, we apply Plancherel's theorem to the damped function $t \mapsto f(t)e^{yt}$. For any $y \le 0$, we obtain the uniform estimate:
\begin{equation}\label{eq:boundary_f}
\begin{aligned}
	\int_{\mathbb{R}} \|\hat{f}(\tau + iy)\|_{H_N^{-3}(\Omega)}^2 d\tau 
	&= 2\pi \int_0^\infty \|f(t)\|_{H_N^{-3}(\Omega)}^2 e^{2yt} dt \\
	&\le 2\pi \int_0^\infty \|f(t)\|_{H_N^{-3}(\Omega)}^2 dt = 2\pi \|f\|_{L^2(0, \infty; H_N^{-3}(\Omega))}^2 < \infty.
\end{aligned}
\end{equation}
An analogous uniform bound holds for $\hat{g}$ in the $L^2(\Omega)$ norm. 

Furthermore, for the boundary limit condition, we must verify that the analytic extension strongly converges to its boundary value on the real axis as $y \to 0^-$. Using Plancherel's theorem and the Lebesgue Dominated Convergence Theorem, this is straightforwardly obtained:
\begin{equation}\label{eq:boundary_limit_f}
\begin{aligned}
	\lim_{y \to 0^-} \int_{\mathbb{R}} \|\hat{f}(\tau + iy) - \hat{f}(\tau)\|_{H_N^{-3}(\Omega)}^2 d\tau 
	&= 2\pi \lim_{y \to 0^-} \int_0^\infty \|f(t)(e^{yt} - 1)\|_{H_N^{-3}(\Omega)}^2 dt \\
	&= 2\pi \int_0^\infty \lim_{y \to 0^-} \|f(t)(e^{yt} - 1)\|_{H_N^{-3}(\Omega)}^2 dt = 0.
\end{aligned}
\end{equation}
By extending the real frequency parameter $\tau$ into $z = \tau + iy$ for $y \in \mathbb{R}$, the corresponding spectral parameter thus becomes $\lambda = \omega + i(\tau + iy) = (\omega - y) + i\tau$. Since the condition $y \le 0$ strictly aligns with the fundamental requirement that $\operatorname{Re}\lambda \ge \omega$ is significantly large, which is a vital premise for the \emph{a priori} energy estimates and the uniform coercivity of the bilinear form, the resolvent operator remains well-defined and maintains the uniform bounds required for the Mikhlin Theorem. Consequently, the operator-valued multiplier $m\big((\omega - y) + i\tau\big)$ exists and is strictly analytic for all $y < 0$.

Let $\tilde{\omega}= \omega - y$, combining the analyticity and uniform boundedness of the multiplier $m(\tilde{\omega}+i\tau)$ and the source data in the lower half-plane established above, we construct their scaled product in the frequency domain:
\begin{equation} \label{eq:analytic_product}
	\hat{V}(\tau + iy) = m(\tilde{\omega}+i\tau)\hat{F}(\tau + iy),
\end{equation}
where $$\hat{V}(\tau + iy):=\begin{pmatrix} (\tilde{\omega}+i\tau) \hat{\varphi}(\tau + iy) \\ (\tilde{\omega}+i\tau)\hat{q}(\tau + iy) \end{pmatrix}, \quad \hat{F}(\tau + iy):=\begin{pmatrix} (\tilde{\omega}+i\tau) \hat{f}(\tau + iy) \\ \hat{g}(\tau + iy) \end{pmatrix}.$$  

On the one hand, $\tilde{\omega}+i\tau$ is a polynomial entire function, which preserves the analyticity of the product. Furthermore, by directly applying Plancherel and Lebesgue Dominated Convergence arguments from equations \eqref{eq:analytic_f}--\eqref{eq:boundary_f} to this derivative term, it follows that the vector-valued source $\big(\lambda \hat{f}, \hat{g}\big)^T$ belongs to the Hardy space $H^2(\mathbb{C}^-, H_N^{-3}(\Omega) \times L^2(\Omega))$. 
On the other hand, because the resolvent multiplier $m(\omega - y + i\tau )$ is analytic and uniformly bounded on $\mathbb{R}$, hence belonging to the operator-valued Hardy space $H^\infty(\mathbb{C}^-; \mathcal{L}(H_N^{-3}(\Omega) \times L^2(\Omega), H_{(0)}^1(\Omega) \times L^2(\Omega)))$. For any $y < 0$, applying the standard norm estimate yields:
\begin{equation*}
\begin{aligned}
	&\int_{\mathbb{R}} \|\hat{V}(\tau + iy)\|_{H_{(0)}^1(\Omega) \times L^2(\Omega)}^2 \, d\tau \\
	&\quad \le \left( \sup_{y < 0, \, \tau \in \mathbb{R}} \|m(\omega - y + i\tau )\|_{\mathcal{L}(H_N^{-3}(\Omega) \times L^2(\Omega), H_{(0)}^1(\Omega) \times L^2(\Omega))} \right)^2 \int_{\mathbb{R}} \|\hat{F}(\tau + iy)\|_{H_N^{-3}(\Omega) \times L^2(\Omega)}^2 \, d\tau.
\end{aligned}
\end{equation*}
Thus, $\hat{V}(\tau + iy)$ is analytic for all $y < 0$, satisfies the uniform $L^2$-bound and strongly converges to its boundary value as $y \to 0^-$ in the sense of the $L^2(\mathbb{R}; H_{(0)}^1(\Omega) \times L^2(\Omega))$ norm.
Returning to the time domain via the inverse Fourier transform as in \eqref{inverse F} evaluated on the real frequency axis:
\begin{equation} \label{eq:inverse_fourier}
\Big((\omega + \partial_t)\varphi(t),(\omega + \partial_t)q(t)\Big )^T= \mathcal{F}^{-1}\big[ \hat{V}(\tau) \big](t).
\end{equation}
By Paley-Wiener theorem, the analyticity coupled with the uniform Hardy space bounds of the frequency solution $\hat{V}(z)$ in the lower half-plane guaranties that
\begin{equation*}
\Big((\omega + \partial_t)\varphi(t),(\omega + \partial_t)q(t)\Big) =(0,0) , \quad \text{for } t < 0.
\end{equation*}
Solving this homogeneous ordinary differential equation yields 
 $$(\varphi(t), q(t))^T = e^{-\omega t}(\varphi(0^-), q(0^-))^T, \text{ for all } t<0.  $$  
 Given that $\omega > 0$, the factor $e^{-\omega t}$ grows exponentially as $t \to -\infty$. However, our global functional setting dictates that the solution must be $L^p$-integrable on the entire real line $\mathbb{R}$. This constraint forces the traces from the left to vanish, i.e., $(\varphi(0^-), q(0^-)) = (0,0)$, which matches our homogeneous initial assumption at $t=0$. 
 Hence, we obtain
 \begin{equation}
 	(\varphi(t), q(t)) \equiv (0,0), \quad \text{for } t < 0.
 \end{equation}
which implies that $\text{supp}\,\varphi \subset [0, \infty)$ and $\text{supp}\,q \subset [0, \infty)$.

For any source data $(f, g)^T$ given on $(0, T)$, we can choose an arbitrary extension to $(0, \infty)$ within the proper functional spaces to construct the global solution. If there were two distinct extensions coinciding on $(0, T)$, their difference would identically vanish on $(0, T)$. By virtue of the causality established via the Paley-Wiener theorem, the resulting difference solution must vanish identically on $(0, T)$ as well. This ensures that the restriction of the solution $(\varphi, q)^T$ onto $(0, T)$ is uniquely determined solely by the values of the source data within the same interval $(0, T)$, independent of the choice of future extension. 

In conclusion, the global solution on $\mathbb{R}$ can be safely restricted to the finite interval $(0, T)$ for any $0<T< \infty$, completing the proof of the local maximal $L^p$-regularity of system \eqref{fg-equtions} that the temporal derivatives satisfy
	\begin{equation}
		\partial_t \varphi \in L^p(0,T; H^1(\Omega)), \quad \text{and} \quad \partial_t q \in L^p(0,T; L^2(\Omega)).
	\end{equation}
Combining this with the vanishing initial traces $\varphi(0)=0$ and $q(0)=0$, this implies $\varphi \in W^1_p(0,T; H^1(\Omega))$ and $q \in W^1_p(0,T; L^2(\Omega))$. 
	
To recover spatial regularity, we return to the original linearized system \eqref{fg-equtions} and utilize elliptic bootstrapping. 
First, reformulate the second equation of \eqref{fg-equtions}:
	\begin{equation}
		-\Div_N \big( \tilde{m}(\varphi_0, q_0, \nabla \varphi_0) \nabla q \big) = g - \partial_t q - c^\Gamma(q_0) \nabla \partial_t \varphi \cdot \nabla \varphi_0.
	\end{equation}

	By assumption $g \in L^p(0,T; L^2(\Omega))$, and since $\partial_t q \in L^p(0,T; L^2(\Omega))$ and $\nabla \partial_t \varphi \in L^p(0,T; L^2(\Omega))$, the entire right-hand side belongs to $L^p(0,T; L^2(\Omega))$. Since the coefficient $\tilde{m}(\varphi_0, q_0, \nabla\varphi_0) \ge m_0 > 0$ and $\tilde{m}(\varphi_0, q_0, \nabla\varphi_0) \in H^2(\Omega)$, standard elliptic regularity theory (see, e.g., Evans \cite[Chapter 6]{Evans2010} or Gilbarg and Trudinger \cite[Chapter 9]{GilbargTrudinger2001}) yields:
	\begin{equation}
		q \in L^p(0,T; H_N^2(\Omega)).
	\end{equation}
	Similarly, we rewrite the first equation as a fourth-order elliptic problem for $\varphi$:
	\begin{equation} \label{eq:elliptic_phi}
		-\Delta_N\big(\text{div}_N(\tilde{\sigma}(q_0)\nabla\varphi)\big) = f - \partial_t \varphi + \Delta_N\big(\text{div}_N(c_\Gamma(q_0)q\nabla\varphi_0)\big) - \Delta_N\Big[(\tilde{\sigma}'(q_0) W'(\varphi_0 )+\sum\limits_{i=1,2} \xi_i'(\varphi_0) c^{(i)}(q_0) )q\Big].
	\end{equation}
	To apply elliptic regularity, we must estimate each term on the right-hand side in $L^p(0, T; H_{(0)}^{-1}(\Omega))$.
	
	Recall our data assumption that the source term $f \in L^p(0, T; H_{(0)}^{-1}(\Omega))$. For the time derivative, the maximal regularity already provides $\partial_t \varphi \in L^p(0, T; H_{(0)}^1(\Omega)) \hookrightarrow L^p(0, T; H_{(0)}^{-1}(\Omega))$.
	
	For the highest-order cross-coupling term involving $q$, noting that the third-order differential operator $\Delta_N \text{div}_N$ maps $H^2(\Omega)$ continuously into $H_{(0)}^{-1}(\Omega)$, we leverage the Banach algebra property of $H^2(\Omega)$ to obtain:
	\begin{align*}
		\big\| \Delta_N\big(\text{div}_N(c_\Gamma(q_0)q\nabla\varphi_0)\big) \big\|_{L^p(0, T; H_{(0)}^{-1}(\Omega))} 
		&\quad \le C \big\| c_\Gamma(q_0)q\nabla\varphi_0 \big\|_{L^p(0, T; H^2(\Omega))} \\
		&\quad \le C \big\| c_\Gamma(q_0)\nabla\varphi_0 \big\|_{H^2(\Omega)} \| q \|_{L^p(0, T; H^2(\Omega))}.
	\end{align*}
	Since we have established $q \in L^p(0, T; H_N^2(\Omega))$, this term is bounded in $L^p(0, T; H_{(0)}^{-1}(\Omega))$.
	
	Similarly, for the lower-order nonlinear term, the Laplacian operator $\Delta_N$ maps $H^1(\Omega)$ continuously into $H_{(0)}^{-1}(\Omega)$. Thus, we can easily bound it by:
	\begin{align*}
		&\big\| \Delta_N\Big[(\tilde{\sigma}'(q_0) W'(\varphi_0 )+\sum\limits_{i=1,2} \xi_i'(\varphi_0) c^{(i)}(q_0) )q\Big] \big\|_{L^p(0, T; H_{(0)}^{-1}(\Omega))}\\ 
		&\quad \le C \big\| (\tilde{\sigma}'(q_0) W'(\varphi_0 )+\sum\limits_{i=1,2} \xi_i'(\varphi_0) c^{(i)}(q_0) )q \big\|_{L^p(0, T; H^1(\Omega))} \\
		&\quad \le C \| q \|_{L^p(0, T; H^1(\Omega))}.
	\end{align*}
	
	Combining these estimates, we conclude that the entire right-hand side of \eqref{eq:elliptic_phi} belongs to $L^p(0, T; H_{(0)}^{-1}(\Omega))$. Therefore, standard elliptic regularity theory ensures maximal spatial regularity:
	\begin{equation}
		\varphi \in L^p(0, T; H_N^3(\Omega)).
	\end{equation}
	
Finally, to obtain the second-order temporal derivative for $\varphi$, we take the time derivative of the first equation, which gives:
\begin{equation}\label{eq:dtt_phi}
	\partial_{tt}\varphi = \partial_t f - \mathcal{A}_{11} \partial_t \varphi - \mathcal{A}_{12} \partial_t q,
\end{equation}
where the principal linear spatial differential operators $\mathcal{A}_{11}$ and $\mathcal{A}_{12}$ are defined by:
\begin{align*}
	\mathcal{A}_{11} \partial_t \varphi &= \omega \partial_t \varphi + \Delta_N \big( \text{div}_N (\tilde{\sigma}(q_0) \nabla \partial_t \varphi) \big), \\
	\mathcal{A}_{12} \partial_t q &= - \Delta_N \big( \text{div}_N (c_\Gamma(q_0) \partial_t q \nabla \varphi_0) \big) + \Delta_N \Big( \big( \tilde{\sigma}'(q_0)W'(\varphi_0) + \sum_{i=1,2}\xi_i'(\varphi_0)c^{(i)}(q_0) \big) \partial_t q \Big).
\end{align*}
 Since $\partial_t f \in L^p(0,T; H_N^{-3}(\Omega))$, $\varphi_0 \in H_N^3(\Omega)$ and $q_0 \in H_N^2(\Omega)$, Sobolev embeddings ensure that all the multiplied coefficients are bounded, i.e., $\tilde{\sigma}(q_0), \nabla\varphi_0, \tilde{\sigma}'(q_0), W'(\varphi_0),  \xi_i'(\varphi_0)c^{(i)}(q_0) \in C^0(\overline{\Omega})$. 
With the operator $\Div_N : L^2(\Omega) \to H_{(0)}^{-1}(\Omega)$ and $\Delta_N : H_{(0)}^{-1}(\Omega) \to H_N^{-3}(\Omega)$, $\mathcal{A}_{11}$ maps $L^p(0,T; H^1(\Omega))$ into $L^p(0,T; H_N^{-3}(\Omega))$. And similarly, $\mathcal{A}_{12}$ maps $L^p(0,T; L^2(\Omega))$ into $L^p(0,T; H_N^{-2}(\Omega)) \hookrightarrow L^p(0,T; H_N^{-3}(\Omega))$. Hence, since all terms on the right-hand side of \eqref{eq:dtt_phi} belong to $L^p(0,T; H_N^{-3}(\Omega))$, we immediately deduce that $\partial_{tt} \varphi \in L^p(0,T; H_N^{-3}(\Omega))$, which asserts:
\begin{equation}
	\varphi \in W_p^2(0,T; H_N^{-3}(\Omega)).
\end{equation}
Gathering all the regularity results above, we complete the proof of the linear maximal $L^p$-regularity.	
		
\subsubsection{Extension to Non-zero Initial Data via Lifting}\label{subsec:extension_lifting}

Having established the maximal $L^p$-regularity for the linearized system with zero initial conditions, we now extend the well-posedness result to incorporate arbitrary compatible initial data. Recall the definition of the fundamental Banach space $\tilde{Z}_T = Z_T^2 \times Z_T^3$ (without initial value restrictions), where
\begin{align*}
	Z_T^2 &:= W_p^1(0, T; H^1(\Omega)) \cap L^p(0, T; H_N^3(\Omega)) \cap W_p^2(0, T; H_N^{-3}(\Omega)), \\
	Z_T^3 &:= L^p(0, T; H_N^2(\Omega)) \cap W_p^1(0, T; L^2(\Omega)).
\end{align*}      
By the trace theory of real interpolation spaces, the trace operator $\gamma_0:\tilde{Z}_T \to \mathbb{Z}_0$, defined by $\gamma_0(\varphi, q) = (\varphi_0, q_0)$ onto the space
\begin{equation}
	\mathbb{Z}_0 := (H_N^3(\Omega), H^1(\Omega))_{1-\frac{1}{p}, p} \times (H_N^2(\Omega), L^2(\Omega))_{1-\frac{1}{p}, p},
\end{equation} 
is a bounded surjection. Therefore, by the Open Mapping Theorem, there exists a bounded right-inverse operator $\gamma_r^{-1} : \mathbb{Z}_0 \to \tilde{Z}_T$. For any given initial data $(\varphi_0, q_0) \in \mathbb{Z}_0$, this allows us to define a reference function $U^* := (\varphi^*, q^*) \in Z_T$ as 
\begin{equation}
	U^* = \gamma_r^{-1}(\varphi_0, q_0).
\end{equation}
This construction guarantees that
\begin{equation}
	U^*(0) = \gamma_0 \gamma_r^{-1}(\varphi_0, q_0) = (\varphi_0, q_0).
\end{equation}
Crucially, because $U^*$ satisfies the required non-zero initial conditions, it follows that $U^* \in \tilde{X}_T:=X_T^2 \times X_T^3$, where $\tilde{X}_T \subset \tilde{Z}_T$ is the affine linear subspace incorporating the initial data. We also obtain the norm estimate $\|U^*\|_{\tilde{Z}_T} \le C \|(\varphi_0, q_0)\|_{\mathbb{Z}_0}$.

By the definition of the linearized operator $\tilde{\mathcal{L}}_T: Z_T^2 \times Z_T^3\to Y_T^2 \times Y_T^3$, for any $(\varphi, q) \in \tilde{Z}_T$, we have the target space:
\begin{equation}
	Y_T^2 \times Y_T^3 := \big(W_p^1(0,T; H^{-3}_N(\Omega)) \cap L^p(0,T; H_{(0)}^{-1}(\Omega))\big) \times L^p(0,T; L^2(\Omega)).
\end{equation}
Recalling the explicit structure of the linearized operator $\tilde{\mathcal{L}}_T$ defined by the system \eqref{fg-equtions}, and applying the triangle inequality alongside the standard Sobolev embedding properties for the variable coefficients, we obtain the following bound for the components in $Y_T^2 \times Y_T^3$:
\begin{align*}
	&\| \mathscr{L}_T (\varphi, q)^T \|_{Y_T^2 \times Y_T^3}\\ 
	&\quad \le C \Big[ \|\varphi\|_{W_p^2(0,T; H_N^{-3}(\Omega)) \cap W_p^1(0,T; H_{(0)}^{-1}(\Omega))} + \|\tilde{\sigma}(q_0)\Delta \varphi\|_{W_p^1(0,T; L^2(\Omega)) \cap L^p(0,T; H^2(\Omega))} \\
	&\quad\quad + \| c_\Gamma(q_0) q \nabla \varphi_0 \|_{W_p^1(0,T; L^2(\Omega)) \cap L^p(0,T; H^2(\Omega))} \Big] \\
	&\qquad+\Big\| \Big(\tilde{\sigma}'(q_0) W'(\varphi_0 )+\sum\limits_{i=1,2} \xi_i'(\varphi_0) c^{(i)}(q_0) \Big)q \Big\|_{W_p^1(0,T; H_{(0)}^{-1}(\Omega)) \cap L^p(0,T; H^1(\Omega))}\\
	&\qquad+ C \Big[ \|\alpha(q_0, \varphi_0, \nabla \varphi_0) \partial_t q\|_{L^p(0,T; L^2(\Omega))} + \| c_\Gamma(q_0) \nabla \varphi_0 \cdot \nabla \partial_t \varphi \|_{L^p(0,T; L^2(\Omega))} \\
	&\qquad\quad\quad + \| \tilde{m}(\varphi_0, q_0, \nabla \varphi_0) \nabla q \|_{L^p(0,T; H^1(\Omega))} \Big] \\
	&\quad \le C \Big( \|\varphi\|_{W_p^2(0,T; H_N^{-3}(\Omega)) \cap W_p^1(0,T; H^1(\Omega)) \cap L^p(0,T; H_N^3(\Omega))} + \|q\|_{W_p^1(0,T; L^2(\Omega)) \cap L^p(0,T; H_N^2(\Omega))} \Big) \\
	&\quad \le C \| (\varphi, q) \|_{\tilde{Z}_T}.
\end{align*}
This establishes that $\tilde{\mathcal{L}}_T : Z_T^2 \times Z_T^3\to Y_T^2 \times Y_T^3$ is a bounded linear operator. Since $U^* \in \tilde{X}_T \subset \tilde{Z}_T$, we immediately have:
\begin{equation}
	\|\tilde{\mathcal{L}}_T U^*\|_{Y_T^2 \times Y_T^3} \le C \|U^*\|_{\tilde{Z}_T} \le C' \|(\varphi_0, q_0)\|_{\mathbb{Z}_0}.
\end{equation}

Now, let $V := U - U^*$. Then $V$ solves the shifted system
\begin{equation}
	\begin{cases}
	\tilde{\mathcal{L}}_T V = \tilde{\mathcal{F}} - \tilde{\mathcal{L}}_T U^*, \\
		V|_{t=0} = 0.
	\end{cases}
\end{equation}
Since $V \in \tilde{X}_T$, and the effective source satisfies $\tilde{\mathcal{F}} - \tilde{\mathcal{L}}_T U^* \in  Y_T^2 \times Y_T^3$, applying the conclusion of the zero initial data problem established above, and recovering the original variable $U = V + U^*$, we get
\begin{equation}
	U \in \tilde{X}_T, \quad \text{for every }\tilde{\mathcal{F}} \in  Y_T^2 \times Y_T^3 \text{ and initial data } (\varphi_0, q_0) \in \mathbb{Z}_0.
\end{equation} 
This completes the proof of the maximal $L^p$-regularity for the strongly coupled system with arbitrary compatible initial values.

\section{The local well-posedness of Model D}\label{sec:model_D}

In this section, we establish the \rev{existence and uniqueness} of strong solutions for \eqref{mod1}. Similarly to our overarching strategy for Model C, to apply the Banach fixed-point theorem, we will first verify the local Lipschitz continuity of the nonlinear part, and subsequently establish the existence and continuity of the inverse linear operator. However, it is crucial to emphasize that the mathematical methodology required for analyzing the linear operator in \eqref{mod1} differs significantly from that utilized for \eqref{main system}. We will provide a detailed and rigorous analysis of this distinct linear framework within this section. 

To initiate this procedure, we first reformulate system \eqref{mod1} into its linear and nonlinear \rev{parts}. Operating within the fundamental Banach space $Z_T = Z_T^1 \times Z_T^2 \times Z_T^3$ (where initial values are unrestricted), the nonlinear system \eqref{mod1} is equivalently rewritten in the abstract form:
\begin{equation}
	\mathcalL (\ve, \varphi,q) = \rev{\tilde{\mathcal{F}}(\ve, \varphi,q), \quad \text{together with } (\ve, \varphi,q)|_{t=0} = (\ve_0, \varphi_0,q_0)},
\end{equation}
where $\mathcalL$ captures the principal linear terms along with the initial traces, mapping into the product space $Y_T \times \gamma_0(Z_T)$, while $\tilde{\mathcal{F}}$ contains the corresponding nonlinearities.

To this end, we first rewrite the momentum equation into an equivalent form:
\begin{equation*}
	\dt(\rho \ve) + \sum_{j=1}^{d}\dxj \left[ \rev{\ve} \left( \rho \ve_j + \bJ_j \right) \right] + \grad p - \Div(2 \eta(\varphi) D\ve) = -\sum_{j=1}^{d}\dxj \left( \varepsilon \rev{\nabla} \varphi \dxj \varphi \right).
\end{equation*}
Using the constitutive relation for $\dt \rho$ alongside the product rule for derivatives, we obtain the equivalent \emph{reduced equation}:
\begin{equation*} 
	\rho \dt \ve + (\rho\ve + \bJ) \cdot \grad \ve + \grad \Pi - \Div(2 \eta(\varphi) D\ve) = -\varepsilon \Lap \varphi \, \grad \varphi, 
\end{equation*}
where the modified pressure is defined as $\Pi = p + \frac{1}{2} \varepsilon |\grad \varphi|^2$. For clarity, by explicitly inserting the constitutive relation for the diffuse flux $\bJ = -\frac{\tilde{\rho}_1 - \tilde{\rho}_2}{2} m(\varphi) \grad \mu$, the momentum equation expands to:
\begin{equation}
	\rho \dt \ve + \left[ \rho \ve - \frac{\tilde{\rho}_1 - \tilde{\rho}_2}{2} m(\varphi) \grad \left(-\varepsilon \Lap \varphi + \frac{1}{\varepsilon} h(q) W'(\varphi)\right) \right] \cdot \grad \ve + \grad \Pi - \Div(2 \eta(\varphi) D\ve) = -\varepsilon \Lap \varphi \, \grad \varphi. 
\end{equation}

Analogous to the procedure for system \eqref{main system}, linearizing around the initial data yields the following formulation. For any $(\ve, \varphi, q) \in Z_T$, the linear operator $\mathcalL$ and the nonlinear source mapping $\mathcal{F}$ are explicitly given by:
\rev{\begin{equation}\label{linear matrix}
	\begin{split}
		\mathcalL(\ve, \varphi, q) &= \begin{bmatrix}
			P_\sigma \dt \ve - P_\sigma \Div\left(2\eta(\varphi_0) D\ve\right)  \\
			\dt \varphi + \varepsilon  m(\varphi_0) \Lap^2 \varphi \\
			\left( \frac{1}{\varepsilon} f'(q_0) W(\varphi_0) + g'(q_0) \right) \dt q + \frac{1}{\varepsilon} f(\varphi_0) W'(\varphi_0) \dt \varphi - \tilde{m}(\varphi_0, q_0) \Lap q 
		\end{bmatrix},\\
		\mathcal F(\ve, \varphi, q) &= \begin{bmatrix}
			P_\sigma  \mathcal F_1(\ve, \varphi, q),
			\mathcal F_2(\ve, \varphi, q),
			\mathcal F_3(\ve, \varphi, q)
		\end{bmatrix}^T,
	\end{split}
\end{equation}}
where the nonlinear coupling terms are collected in:
\begin{equation*}
	\begin{split}
		\mathcal F_1(\ve, \varphi, q) =& (\rho(\varphi_0) - \rho(\varphi)) \dt \ve - \Div\left(2\eta(\varphi_0) D\ve\right) + \Div\left(2\eta(\varphi) D\ve\right) - \varepsilon \Lap \varphi \grad \varphi \\
		& - \left[ \rho \ve - \frac{\tilde{\rho}_1 - \tilde{\rho}_2}{2} m(\varphi) \grad \left(-\varepsilon \Lap \varphi + \frac{1}{\varepsilon} h(q) W'(\varphi)\right) \right] \cdot \grad \ve,\\
		\mathcal F_2(\ve, \varphi, q) =& -\ve \cdot \grad \varphi + \frac{1}{\varepsilon}\Div\left(  m(\varphi)\grad( h(q) W'(\varphi)) \right) +(\varepsilon m(\varphi_0)-\varepsilon m(\varphi)) \Lap^2 \varphi - \varepsilon m'(\varphi)\grad\varphi\cdot \grad\Lap \varphi, \\
		\mathcal F_3(\ve, \varphi, q) =& \left( \frac{1}{\varepsilon} f'(q_0) W(\varphi_0) + g'(q_0) - \frac{1}{\varepsilon} f(q) W(\varphi) - g'(q) \right) \dt q- \ve \cdot \grad \left( \frac{1}{\varepsilon} f(q) W(\varphi) \right) \\
		&- \left( \frac{1}{\varepsilon} f(q) W'(\varphi) - \frac{1}{\varepsilon} f(q_0) W'(\varphi_0) \right) \dt\varphi + g(q)  - \tilde{m}(\varphi_0, q_0) \Lap q + \Div\left(\tilde{m}(\varphi, q) \grad q\right).
	\end{split}
\end{equation*}

To establish the local well-posedness for \eqref{mod1}, we adopt a functional framework analogous to that of \eqref{main system}. We first define the fundamental Banach spaces $Z_T = Z_T^1 \times Z_T^2 \times Z_T^3$ and the target bulk spaces $Y_T = Y_T^1 \times Y_T^2 \times Y_T^3$ as:
\begin{equation*}
\begin{aligned}
	Z_T^1 &:= W_2^1(0,T; L_\sigma^2(\Omega)) \cap L^2(0,T; H^2(\Omega)^d \cap H_0^1(\Omega)^d), \\
	Z_T^2 &:= W_p^1(0,T; L^p(\Omega)) \cap L^p(0,T; W_{p,N}^4(\Omega)), \\
	Z_T^3 &:= W_p^1(0,T; L^p(\Omega)) \cap L^p(0,T; W_{p,N}^2(\Omega)), \\
\end{aligned}
\end{equation*}
and
\begin{equation*}
 Y_T = Y_T^1 \times Y_T^2 \times Y_T^3 := L^2(0,T; L^2_\sigma(\Omega))\times L^p(0,T;L^p(\Omega))\times L^p(0,T; L^p(\Omega)). 
\end{equation*}
By virtue of the maximal $L^p$-regularity framework and the trace theory of real interpolation spaces, the initial values must belong to the respective trace spaces of $Z_T$ at $t = 0$. Specifically, we require:
\begin{align*}
	\mathbf{v}_0 \in \gamma_0(Z_T^1) &= (L^2_\sigma(\Omega), H^2(\Omega)^d \cap H_0^1(\Omega)^d)_{1-\frac{1}{2}, 2} = H_0^1(\Omega)^d \cap L^2_\sigma(\Omega), \\
	\varphi_0 \in \gamma_0(Z_T^2) &= (L^p(\Omega), W_{p,N}^4(\Omega))_{1-\frac{1}{p}, p}, \\
	q_0 \in \gamma_0(Z_T^3) &= (L^p(\Omega), W_{p,N}^2(\Omega))_{1-\frac{1}{p}, p}.
\end{align*}

To ensure that the embedding constants remain independent of $T$, we equip $Z_T^1$, $Z_T^2$, and $Z_T^3$ with the equivalent norms $\|\cdot\|'_{Z_T^1}$, $\|\cdot\|'_{Z_T^2}$, and $\|\cdot\|'_{Z_T^3}$ defined by adding the initial trace norms:
\begin{align*}
	\|\mathbf{v}\|'_{Z_T^1} &:= \|\mathbf{v}\|_{Z_T^1} + \|\mathbf{v}(0)\|_{\gamma_0(Z_T^1)}, \\
	\|\varphi\|'_{Z_T^2} &:= \|\varphi\|_{Z_T^2} + \|\varphi(0)\|_{\gamma_0(Z_T^2)}, \\
	\|q\|'_{Z_T^3} &:= \|q\|_{Z_T^3} + \|q(0)\|_{\gamma_0(Z_T^3)}.
\end{align*}

Finally, for the application of the Banach contraction mapping principle to the nonlinear system, we define the affine linear subspace $X_T := X_T^1 \times X_T^2 \times X_T^3$ by incorporating the prescribed initial conditions:
\begin{align*}
	X_T^1 &:= \big\{ \mathbf{v} \in Z_T^1 \;\big|\; \mathbf{v}|_{t=0} = \mathbf{v}_0 \big\}, \\
	X_T^2 &:= \big\{ \varphi \in Z_T^2 \;\big|\; \varphi|_{t=0} = \varphi_0 \big\}, \\
	X_T^3 &:= \big\{ q \in Z_T^3 \;\big|\; q|_{t=0} = q_0 \big\}.
\end{align*}
We emphasize that, $X_T$ is not a vector space but an affine subspace of $Z_T$.

With the rigorous functional framework established, we are now positioned to present the two fundamental ingredients required for the Banach fixed-point argument. We first state the local Lipschitz continuity for the nonlinear source mapping $\mathcal{F}$.
\begin{proposition}\label{proposition}
	Let Assumptions \ref{asumm_2} hold, and let $\ve_0$, $\varphi_0$ and $q_0$ be given as in Theorem \ref{th1} and let $R>0$. Then, there exists a constant $C(T, R) > 0$ such that
	\begin{equation*}
		\|\mathcal{F}(\ve_1, \varphi_1, q_1) - \mathcal{F}(\ve_2, \varphi_2, q_2)\|_{Y_T} \leq C(T, R) \|(\ve_1 - \ve_2, \varphi_1 - \varphi_2,  q_1 - q_2)\|_{X_T},
	\end{equation*}
	for all $(\ve_i, \varphi_i, q_i) \in X_T$ with $\|(\ve_i, \varphi_i, q_i)\|_{X_T} \leq R$ for $i = 1, 2$. Moreover, $C(T, R) \to 0$ as $T \to 0$.
\end{proposition}

Complementing the nonlinear estimates, the second essential ingredient is the continuous invertibility of the linearized operator $\mathcal{L}$. 
\begin{proposition}\label{inv-con}
	Let $T > 0$ and $\mathcal{L}, X_T, Y_T$ be defined as before. Then, the \rev{operator} $\mathcal{L}: X_T \to Y_T$ is invertible. Moreover, for every $T_0 > 0$, there exists a constant $C(T_0) > 0$ such that
	\rev{
	\begin{equation*}
		\|\mathcal{L}^{-1}(\mathcal{F}_1) - \mathcal{L}^{-1}(\mathcal{F}_2)\|_{X_T} \leq C(T_0) \|\mathcal{F}_1 - \mathcal{F}_2\|_{Y_T} \quad \forall \mathcal{F}_1, \mathcal{F}_2 \in Y_T,
	\end{equation*}
	}
	for all $T \in (0, T_0]$.
\end{proposition}

Equipped with Proposition \ref{proposition} and Proposition \ref{inv-con}, the full nonlinear problem can be rigorously reduced to a \rev{fixed point equation}. We now \rev{use} these results to establish the existence and uniqueness of the local strong solution.
\begin{proof}[Proof of Theorem \ref{th1}]
	 With the help of Proposition \ref{proposition} and Proposition \ref{inv-con}, one can easily 
	 complete the proof of Theorem \ref{th1} with the same argument as in the poof of Theorem \ref{thm:local_well_posedness}.
\end{proof}

\subsection{Lipschitz continuity of \texorpdfstring{$\mathcal{F}$}{F}}

To rigorously establish the local Lipschitz continuity of the nonlinear mapping $\mathcal{F}$, the integrability exponent for the phase-field space $X_T^2$ must strictly satisfy $p > 4$. This specific constraint is fundamentally dictated by the necessity to control the highly singular nonlinear cross-term $\grad\Lap\varphi \cdot \grad\ve$. 

Furthermore, the rigorous analysis of the surfactant concentration space $X_T^3$ imposes an even stricter integrability constraint, requiring $5 < p \le 6$. This secures continuous embeddings into spaces with sufficient classical differentiability to properly bound higher-order derivative couplings, such as the terms introducing $\grad^2\varphi \cdot \grad\ve$ or $m'(\varphi)|\grad\varphi|^2$. While this restricted range is not necessary for every individual lower-order term, we uniformly adopt $5 < p < 6$ to guarantee consistency in all nonlinear estimates.

Exploiting these precisely chosen integrability exponents, we deduce a series of fundamental embedding properties essential for closing the nonlinear estimates. Drawing upon the maximal regularity framework in \cite{AbelsWeber2021}, the fluid velocity $\ve \in X_T^1$ satisfies:
\begin{equation}
	\ve \in X_T^1 \hookrightarrow \mathrm{BUC}([0,T]; B_{2,2}^1(\Omega)) = \mathrm{BUC}([0,T]; H^1(\Omega)).
\end{equation}
In particular, via standard interpolation, the spatial gradient $\grad\ve$ inherits the following spatio-temporal integrability:
\rev{\begin{equation}
	\begin{aligned}
		\grad \ve &\in L^\infty(0,T; L^2(\Omega)) \cap L^2(0,T; L^6(\Omega)) \hookrightarrow L^{8/3}(0,T; L^4(\Omega))\cap L^4(0,T; L^3(\Omega)). \\
	\end{aligned}
\end{equation}}

Simultaneously, the high-order spatial and temporal integrability of the phase-field variable $\varphi \in X_T^2$ yields the trace embedding:
\begin{equation}
	\varphi \in L^p(0,T; W_{p,N}^4(\Omega)) \cap W_p^1(0,T; L^p(\Omega)) \hookrightarrow \mathrm{BUC}\left([0,T]; W_p^{4-4/p}(\Omega)\right).
\end{equation}
Furthermore, applying mixed-derivative interpolation bounds, we obtain:
\begin{equation}
	\varphi \in C^{0,1-1/p}([0,T]; L^p(\Omega)) \cap C\left([0,T]; W_p^{4-4/p}(\Omega)\right) \hookrightarrow C^{0,(1-1/p)\theta_1}([0,T]; W_p^3(\Omega)), 
\end{equation}
for some appropriately chosen $\theta_1 \in (0,1)$. Crucially, for spatial dimensions $d \in \{2, 3\}$ and our restricted exponent $5 < p < 6$, the Sobolev embedding $W_p^3(\Omega) \hookrightarrow C^2(\overline{\Omega})$ holds. This guarantees the essential uniform bound:
\begin{equation}\label{phi-boun}
	\varphi \in \mathrm{BUC}([0,T]; C^2(\overline{\Omega})).
\end{equation}

On the other hand, let $q \in X_T^3$ be given. We first recall the definition of the underlying Sobolev space $Z_T^3$:
\begin{equation*}
	Z_T^3 := W_p^1(0,T; L^p(\Omega)) \cap L^p(0,T; W_{p,N}^2(\Omega)).
\end{equation*}
From \eqref{con-em-th}, we obtain 
\begin{equation}\label{q-boun}
    q \in W_p^1(0,T; L^p(\Omega)) \cap L^p(0,T; W_{p,N}^2(\Omega)) \hookrightarrow \BUC([0,T]; W_p^{2-2/p}(\Omega)).
\end{equation}
For $p > d+2$, the embedding $W_p^{2-2/p}(\Omega) \hookrightarrow C^1(\overline{\Omega})$ directly yields $q \in C([0,T]; C^1(\overline{\Omega}))$ and $$\nabla q \in \BUC([0,T]; W_p^{1-2/p}(\Omega)).$$

Furthermore, the continuous embedding $W_p^1(0,T; L^p(\Omega)) \hookrightarrow C^{0,1-1/p}([0,T]; L^p(\Omega))$ allows us to interpolate with \eqref{q-boun}. By choosing $\theta_2 = \frac{p-2}{2p-2}$, we have
\begin{equation}
    \left(B_{p,p}^{2-\frac{2}{p}}(\Omega), L^p(\Omega)\right)_{\theta_2,2} = B_{p,2}^1(\Omega) \hookrightarrow W_p^1(\Omega),
\end{equation}
which yields the pointwise estimate
\begin{equation}
    \norm{q(t)}_{W_p^1(\Omega)} \leq C \norm{q(t)}_{W_p^{2-\frac{2}{p}}(\Omega)}^{1- \theta_2} \norm{q(t)}_{L^p(\Omega)}^{\theta_2}, \quad \forall t \in [0,T].
\end{equation}
Consequently, by Lemma \ref{XYX}, we deduce the Hölder continuity in time:
\begin{equation}
    q \in C^{0,1-1/p}([0,T]; L^p(\Omega)) \cap C([0,T]; W_p^{2-2/p}(\Omega)) \hookrightarrow \Cspace{(1-1/p)\theta_2}{W_p^1(\Omega)},
\end{equation}
where the prerequisite $1 > \frac{d}{p}$ is satisfied since $5 < p < 6$ and $d =2,3$.

Before proceeding further, we analyze the Banach spaces bounding auxiliary quantities like $\ve$,  $\varphi$, $q$, $\grad q$, $\grad\varphi$, $m(\varphi)$, $\tilde{m}(\varphi, q)$, and so on are bounded. Let $\theta=\min\{\theta_1, \theta_2\}$, and $r_\varphi \in \{\eta(\varphi), m(\varphi)\}$. We have:
\begin{equation}\label{r1-boun}
	r_\varphi \text{ is bounded in } \Cspace{(1-1/p)\theta}{W_p^3(\Omega)},
\end{equation}
which can be readily established by the same arguments as in Section \ref{sec:model_c}.

Furthermore, we need to prove that for any $q \in X_T^3$ with $\|q\|_{X_T^3} \leq R$, there exists a constant $C(R) > 0$ such that 
\begin{equation*}
	\|r_q(q)\|_{Z_T^3} \leq C(R),
\end{equation*} 
where $r_q \in \{f, g, h\}$. To this end, given that $q$ is bounded in $X_T^3$, the continuous embedding $$q \in \BUC([0,T]; W_p^{2-2/p}(\Omega)) \hookrightarrow \BUC([0,T]; C^1(\overline{\Omega}))$$ strictly ensures that $\|q(t)\|_{C^1(\overline{\Omega})} \leq C(R)$ for all $t \in [0,T]$. Consequently, we initially possess:
\begin{equation*}
	r_q(q) \in L^\infty(0,T; W_p^1(\Omega)).
\end{equation*}
Since $W_p^1(\Omega)$ is a Banach algebra, which holds because $1 > \frac{d}{p}$ for $d \in \{2,3\}$ and $p > 5$, we obtain the pointwise estimate:
\begin{equation*}
	\|\grad r_q(q(t))\|_{W_p^1(\Omega)} \leq C \|r_q'(q(t))\|_{W_p^1(\Omega)} \|\grad q(t)\|_{W_p^1(\Omega)} \quad \text{a.e. } t \in [0,T].
\end{equation*}
In addition, since $r_q'(q)$ is bounded in $L^\infty(0,T; W_p^1(\Omega))$ and $q \in L^p(0,T; W_p^2(\Omega))$, integrating over time yields the boundedness for $\grad r_q(q)$. Thus, there exists a constant $C(R)$ such that:
\begin{equation*}
	\|\grad r_q(q)\|_{L^p(0,T; W_p^1(\Omega))} \leq C(R) \quad \forall q \in X_T^3 \quad \text{with} \quad \|q\|_{X_T^3} \leq R,
\end{equation*}
which implies that $r_q(q)$ is bounded in $L^p(0,T; W_p^2(\Omega))$. 

Additionally, because $q \in W_p^1(0,T; L^p(\Omega))$, we can derive that $r_q(q)$ is bounded in $W_p^1(0,T; L^p(\Omega))$ via the chain rule
\begin{equation*}
	\dt r_q(q(t)) = r_q'(q(t)) \dt q(t),
\end{equation*}
where $r_q'(q)$ is bounded in $C^0(\overline{Q}_T)$. This establishes the continuous embedding:
\begin{equation*}
	r_q(q) \in W_p^1(0,T; L^p(\Omega)) \hookrightarrow C^{0,1-1/p}([0,T]; L^p(\Omega)).
\end{equation*}
Coupling this with Lemma \ref{con-em-th} yields:
\begin{equation}
	r_q(q) \in C^{0,1-1/p}([0,T]; L^p(\Omega)) \cap C([0,T]; W_p^{2-2/p}(\Omega)) \hookrightarrow C^{0,(1-1/p)\theta}([0,T]; W_p^1(\Omega)).
\end{equation}
In particular, since the convective term $\ve \cdot \nabla\big(f(q)W(\varphi)\big) = W(\varphi) \ve \cdot \nabla f(q) + f(q) \ve \cdot \nabla W(\varphi)$ appears in $\mathcal{F}_3$, it is necessary to highlight the trace embedding:
\[
\nabla f(q) \in L^\infty\left(0,T; W_p^{1-\frac{2}{p}}(\Omega)\right) \hookrightarrow L^\infty\big([0,T]; C^0(\overline{\Omega})\big), \quad \text{valid for any } p > 5.
\]
\rev{Next, let $z \in \{\tilde{m}, a_i\}$ satisfying $z(\varphi, q) \in C^2(\mathbb{R}^2)$. Then, for every $R > 0$, there exists a constant $L = L(R) > 0$ such that, for every $t \in [0, T]$ and $\varphi_i \in X_T^2, q_i \in X_T^3$ such that $\|\varphi_i\|_{X_T^2}, \|q_i\|_{X_T^3} \le R$ ($i = 1, 2$), it follows from Theorem \ref{lem:Runst_Sickel} that}
\begin{equation}
	\|z(\varphi_1(t), q_1(t)) - z(\varphi_2(t), q_2(t))\|_{W_p^1(\Omega)} \le L \Big( \|\varphi_1(t) - \varphi_2(t)\|_{W_p^1(\Omega)} + \|q_1(t) - q_2(t)\|_{W_p^1(\Omega)} \Big).
\end{equation}
\begin{equation}\label{zma}
	\|z(\varphi_1(t), q_1(t)) - z(\varphi_2(t), q_2(t))\|_{W_p^1(\Omega)} \leq L \left( \|\varphi_1(t) - \varphi_2(t)\|_{W_p^1(\Omega)} + \|q_1(t) - q_2(t)\|_{W_p^1(\Omega)} \right).
\end{equation}
In particular, this simplifies for single-variable functions $r_q \in \{f, g, h\}$ to:
\begin{equation*}
	\|r_q(q_1(t)) - r_q(q_2(t))\|_{W_p^1(\Omega)} \leq L \|q_1(t) - q_2(t)\|_{W_p^1(\Omega)}.
\end{equation*}
Analogously, we possess the uniform bound $\|z(\varphi, q)\|_{W_p^1(\Omega)} \leq C(R),$
and
\begin{equation}\label{z-bound}
	z(\varphi, q) \text{ is bounded in } C^{0,(1-1/p)\alpha}([0,T]; W_p^1(\Omega)).
\end{equation}

\begin{proof}[Proof of Proposition \ref*{proposition}]
Let $(\ve_i, \varphi_i, q_i) \in X_T$ for $i=1,2$ \rev{with $\|(\ve_i, \varphi_i, q_i)\|_{X_T} \le R$, we} estimate the full nonlinear \rev{terms}:
\begin{equation*}
	\begin{aligned}
		&\|\mathcal F(\ve_1, \varphi_1, q_1) - \mathcal F(\ve_2, \varphi_2, q_2)\|_{Y_T} \\
		&\quad= \|P_{\sigma}\left(\mathcal F_1(\ve_1, \varphi_1, q_1) - \mathcal F_1(\ve_2, \varphi_2, q_2)\right)\|_{L^2(0,T; L^2(\Omega))} + \|\mathcal F_2(\ve_1, \varphi_1, q_1) - \mathcal F_2(\ve_2, \varphi_2, q_2)\|_{L^p(0,T; L^p(\Omega))} \\
		&\qquad+ \|\mathcal F_3(\ve_1, \varphi_1, q_1) - \mathcal F_3(\ve_2, \varphi_2, q_2)\|_{L^p(0,T; L^p(\Omega))}.
	\end{aligned}
\end{equation*}
    Based on the uniform boundedness of the variables $\ve$, $\varphi$, $q$, $\nabla q$, $\nabla\varphi$, $m(\varphi)$, and $\tilde{m}(\varphi, q)$ derived above within their respective topological spaces, the complete local Lipschitz continuity of $\mathcal{F}$ can be systematically established by synthesizing two rigorous frameworks. Specifically, the estimates for the nonlinearities in the Navier-Stokes and phase-field equations follow the exact procedures detailed in of \cite[Section 4]{AbelsWeber2021}, whereas all nonlinear terms within the surfactant equation are bounded analogously to the methods developed in Section \ref{sec:model_c}.
    
	Recall the momentum nonlinearity:
	\begin{equation*}
		\begin{aligned}
			\mathcal F_1(\ve, \varphi, q) =& (\rho_0 - \rho) \dt \ve - \Div(\eta(\varphi_0) D\ve) + \Div(\eta(\varphi) D\ve) - \Lap \varphi\grad\varphi \\
			&- \left[ \rho \ve + \frac{\tilde{\rho}_1 - \tilde{\rho}_2}{2} m(\varphi) \grad \left(-\varepsilon \Lap \varphi + \frac{1}{\varepsilon}h(q)W'(\varphi)\right) \right] \cdot \grad \ve.
		\end{aligned}
	\end{equation*}
Due to the distinct function space required for the velocity $\ve$ compared to Theorem \ref{thm:local_well_posedness}, we highlight the estimate for the highly coupled convective term in $\mathcal F_1$. Let the associated vector field be defined as $a_1(\varphi, q) := m(\varphi)\grad\left(h(q)W'(\varphi)\right)$. Exploiting the embeddings established above, we bound this term as follows:
\begin{align*}
	&\| a_1(\varphi_1, q_1) \cdot \nabla \ve_1 - a_1(\varphi_2, q_2) \cdot \nabla \ve_2 \|_{Y_T^1} \\
	&\quad \leq C \| a_1(\varphi_1, q_1) \cdot (\nabla \ve_1 - \nabla \ve_2) \|_{Y_T^1} + C \| (a_1(\varphi_1, q_1) - a_1(\varphi_2, q_2)) \cdot \nabla \ve_2 \|_{Y_T^1} \\
	&\quad \leq C \sup_{t \in [0,T]} T^{1/8} \| a_1(\varphi_1, q_1) \|_{W_p^1(\Omega)} \left( \int_0^T \| \nabla \ve_1 - \nabla \ve_2 \|_{L^4(\Omega)}^{8/3} \, dt \right)^{3/8} \\
	&\qquad + C T^{1/8} \| \nabla \ve_2 \|_{L^{8/3}(0,T; L^4(\Omega))} \| a_1(\varphi_1, q_1) - a_1(\varphi_2, q_2) \|_{C([0,T]; C^0(\overline{\Omega}))} \\
	&\quad \leq C(R) T^{1/8} \| \ve_1 - \ve_2 \|_{X_T^1} + C(R) T^{1/8} \left( \| \varphi_1 - \varphi_2 \|_{X_T^2} + \| q_1 - q_2 \|_{X_T^3} \right) \\
	&\quad \leq C(R, T) \| (\ve_1 - \ve_2, \varphi_1 - \varphi_2, q_1 - q_2) \|_{X_T}.
\end{align*} 
Combining these with the preceding estimates yields
\begin{equation}
	\|\mathcal{F}_1(\ve_1, \varphi_1, q_1) - \mathcal{F}_1(\ve_2, \varphi_2, q_2)\|_{Y_T} \leq C(R,T) \|(\ve_1 - \ve_2, \varphi_1 - \varphi_2, q_1 - q_2)\|_{X_T},
\end{equation}
where $
C(R,T) \to 0 \text{ as } T \to 0$.

For $\mathcal F_2$, recall that 
\begin{equation*}
\mathcal F_2(\ve, \varphi, q) = -\ve \cdot \grad \varphi + \frac{1}{\varepsilon}\Div\left(  m(\varphi)\grad( h(q) W'(\varphi)) \right) +(\varepsilon m(\varphi_0)-\varepsilon m(\varphi)) \Lap^2 \varphi - \varepsilon m(\varphi)\grad\varphi\cdot \grad\Lap \varphi,
\end{equation*}
Let $a_2(\varphi,q)= \frac{1}{\varepsilon} \tilde{m}(\varphi) h(q) W'(\varphi)$. Then we get from \eqref{zma} that
\begin{equation*}
	\begin{aligned}
	&\norm{\Div\left(a_2(\varphi_1,q_1)- a_2(\varphi_2,q_2) \right)}_{Y_T^2}\\ 
    &\quad \leq\norm{a_2(\varphi_1,q_1)- a_2(\varphi_2,q_2)}_{L^p(0,T;W^1_p(\Omega ))} \\ 
	&\quad \leq C\sup_{t \in [0, t]} \|(\varphi_1,q_1)-(\varphi_{1}(0),q_{1}(0))-\left((\varphi_2,q_2)-(\varphi_{2}(0),q_{2}(0))\right)\|_{W^1_p(\Omega)}T^{1/p} \\
		&\quad \leq C(R) T^{1/p+(1-1/p)\alpha} \|(\varphi_1,q_1) -(\varphi_2, q_2)\|_{C^{0,(1-1/p)\alpha}([0,T]; W_p^1(\Omega))}\\
		&\quad \leq  C(R,T)  \|(\ve_1 -\ve_2,\varphi_1 -\varphi_2, q_1 -q_2)\|_{X_T}.
	\end{aligned}
\end{equation*}
Combining these with the preceding estimates as in the \cite[Section 4]{AbelsWeber2021}, yields
\begin{equation}
	\|\mathcal{F}_2(\ve_1, \varphi_1, q_1) - \mathcal{F}_2(\ve_2, \varphi_2, q_2)\|_{Y_T} \leq C(R,T) \|(\ve_1 - \ve_2, \varphi_1 - \varphi_2, q_1 - q_2)\|_{X_T},
\end{equation}
where $
C(R,T) \to 0 \text{ as } T \to 0. $ 

Recall
\begin{equation*}
		\begin{aligned}
	\mathcal{F}_3(v, \varphi, q) =& \left( \frac{1}{\varepsilon} f'(q_0) W(\varphi_0) + g'(q_0) - \frac{1}{\varepsilon} f'(q) W(\varphi) - g'(q) \right) \dt q - \left( \frac{1}{\varepsilon} f(q) W'(\varphi) - \frac{1}{\varepsilon} f(q_0) W'(\varphi_0) \right) \dt \varphi \\&- \ve \cdot \grad \left( \frac{1}{\varepsilon} f(q) W(\varphi) + g(q) \right) - \tilde{m}(\varphi_0, q_0) \Lap q + \Div \left( \tilde{m}(\varphi, q) \grad q \right).
		\end{aligned}
\end{equation*}

Let $a_3(\varphi,q)= \frac{1}{\varepsilon} f'(q) W(\varphi)$. Since $ f'(q), W(\varphi) \in C^2(\mathbb{R})$, it holds $a_3(\varphi,q)\in C^2(\mathbb{R}^2)$. Then we obtain the estimate through \eqref{z-bound} for the first part:
\begin{equation*}
	\begin{aligned}
		&\norm{(a_3(\varphi_1,q_1)-a_3(\varphi_0,q_0))\dt q_1 - (a_3(\varphi_2,q_2)-a_3(\varphi_0,q_0))\dt q_2}_{Y_T^3} \\
		&\quad \leq \norm{(a_3(\varphi_1,q_1)-a_3(\varphi_0,q_0))(\dt q_1 - \dt q_2)}_{Y_T^3} + \norm{(a_3(\varphi_1,q_1) - a_3(\varphi_2,q_2))\dt q_2}_{Y_T^3} \\
		&\quad = H_{11} + H_{12}.
	\end{aligned}
\end{equation*}
It get from \eqref{zma} that
\begin{equation*}
	\begin{aligned}
	H_{11} \leq& C \sup_{t \in [0,T]} \norm{a_3(\varphi_1,q_1) - a_3(\varphi_1{(0)},q_1{(0)})}_{L^\infty(\Omega)} \norm{\dt q_1 - \dt q_2}_{L^p(0,T;L^p(\Omega))}\\
	\leq & C \sup_{t \in [0,T]} \norm{(\varphi_1,q_1) - (\varphi_1{(0)},q_1{(0)})}_{W^{1,p}(\Omega)} \norm{q_1 - q_2}_{W^{1,p}(0,T;L^p(\Omega))}\\
	\leq & C(R) T^{(1-1/p)\alpha} \norm{q_1 - q_2}_{X_T^3} \\
\end{aligned}
\end{equation*}
and
\begin{equation*}
\begin{aligned}
H_{12} \leq &C \sup_{t \in [0,T]} \norm{a_3(\varphi_1,q_1) - a_3(\varphi_2,q_2)}_{L^\infty(\Omega)} \norm{q_2}_{L^p(0,T;W^{1,p}(\Omega))}\\
\leq& C \sup_{t \in [0,T]} \norm{(\varphi_1,q_1) - (\varphi_2,q_2)}_{W^{1,p}(\Omega)} \norm{q_2}_{X_T^3}\\
\leq& C(R) T^{(1-1/p)\alpha}\norm{(v_1 - v_2, \varphi_1 - \varphi_2, q_1 - q_2)}_{X_T}
\end{aligned}
\end{equation*}	
for $5 < p < 6$.
For the second part of the first term, we obtain	
\begin{equation}
	\begin{aligned}
	&\| (g'(q_1) - g'(q_0) ) \dt q_1 - (g'(q_2) - g'(q_0) ) \dt q_2 \|_{Y_T^3} \\
	&\quad \leq C \sup_{t \in [0,T]} \| g'(q_1(t)) - g'(q_2(t)) \|_{W^{1,p}(\Omega)} \| \dt q_1 \|_{L^p(0,T; L^p(\Omega))} \\
	&\qquad + C \sup_{t \in [0,T]} \| g'(q_2(t)) - g'(q_2{(0)}) \|_{W^{1,p}(\Omega)} \| \dt q_1 - \dt q_2 \|_{L^p(0,T; L^p(\Omega))} \\
	&\quad \leq  C(R) T^{(1-1/p)\alpha} \| q_1 - q_2 \|_{X_T^3}.
	\end{aligned}
\end{equation} 
Combining the estimates of $H_{11}$ and $H_{12}$, the estimate for the first term of $\mathcal{F}_{3}$ can be written as
\begin{equation*}
	\norm{(a_3(\varphi_1,q_1)-a_3(\varphi_0,q_0))(\dt q_1 - \dt q_2)}_{Y_T^3}  \leq C(R) T^{(1-1/p)\alpha}\norm{(\ve_1 - \ve_2, \varphi_1 - \varphi_2, q_1 - q_2)}_{X_T}.
\end{equation*} 
Applying analogous arguments to the remaining terms in $\mathcal{F}_3$, we obtain the comprehensive bound for the third component:
\begin{equation}\label{est-F3}
	\norm{\mathcal{F}_3(\ve_1, \varphi_1, q_1) - \mathcal{F}_3(\ve_2, \varphi_2, q_2)}_{Y_T} \leq C(R, T)\norm{(\ve_1 - \ve_2, \varphi_1 - \varphi_2, q_1 - q_2)}_{X_T},
\end{equation}
where $C(R,T) \to 0$ as $T \to 0$.

Finally, collecting the individual estimates for $\mathcal{F}_1$, $\mathcal{F}_2$, and $\mathcal{F}_3$ above, we conclude that the full nonlinear source mapping $\mathcal{F}: X_T \rightarrow Y_T$ satisfies:
\begin{equation*}
	\left\| \mathcal{F}(\ve_1,\varphi_1,q_1) - \mathcal{F}(\ve_2,\varphi_2,q_2) \right\|_{Y_T} \leq C(T,R) \left\| (\ve_1 - \ve_2, \varphi_1 - \varphi_2, q_1 - q_2) \right\|_{X_T}.
\end{equation*}
\end{proof}

\subsection{Existence and continuity of \texorpdfstring{$\mathcal{L}^{-1}$}{L}}
To prove Theorem \ref{inv-con}, we must demonstrate that the  operator $\mathcal{L}: Z_T \to \mathbb{Y}_T$ is invertible and that its inverse possesses a uniform upper bound for $T\in (0,T_0]$, where the target space is structured block-wise as $\mathbb{Y}_T := (Y_{T,1} \times \gamma_0(Z_{T,1})) \times (Y_{T,2} \times \gamma_0(Z_{T,2})) \times (Y_{T,3} \times \gamma_0(Z_{T,3}))$. From \eqref{linear matrix}, we observe that the first block depends solely on the velocity variable $\ve$, the second block depends exclusively on the phase-field variable $\varphi$, whereas the third block couples the variables $\varphi$ and $q$. 

Restating the equations in a decoupled block format, we have
\begin{equation*}
	\mathcalL(\ve, \varphi, q) = \begin{bmatrix}
		\mathcalL_1 (\ve) \\
		\mathcalL_2 (\varphi) \\
		\mathcalL_3 (\varphi, q)
	\end{bmatrix}
	= \begin{bmatrix}
		\left( P_\sigma \dt \ve - P_\sigma \Div\left(2\eta(\varphi_0) D\ve\right), \, \ve|_{t=0} \right) \\
		\left( \dt \varphi + \varepsilon m(\varphi_0) \Lap^2 \varphi, \, \varphi|_{t=0} \right) \\
		\left( \left( \frac{1}{\varepsilon} f'(q_0) W(\varphi_0) + g'(q_0) \right) \dt q + \frac{1}{\varepsilon} f(\varphi_0) W'(\varphi_0) \dt \varphi - \tilde{m}(\varphi_0, q_0) \Lap q, \, q|_{t=0} \right)
	\end{bmatrix}.
\end{equation*}
Therefore, we may first solve the top two equations independently before addressing the coupled third equation.

\begin{lemma}\label{inv-l2}
	Let Assumption \ref{asumm_2} hold. Let $\varphi_0 \in \left(L^{p}(\Omega), W^4_{p,N}(\Omega) \right)_{1-\frac{1}{p},p}$, $f \in L^{p}(0, \rev{T_0}; L^{p}(\Omega))$, and $0 < T_0 < \infty\rev{, 1 < p < \infty}$ be given. Then, \rev{there} exists a unique solution
	$$
	\varphi \in L^{p}(0, T; W^4_{p,N}(\Omega)) \cap \left\{ u \in W^{1}_p(0, T; L^{p}(\Omega)) : u|_{t=0} = \varphi_0 \right\}
	$$
	\rev{to}
	\begin{equation}\label{phi-system}
		\begin{aligned}
			\dt\varphi + \varepsilon m(\varphi_0) \Delta^2 \varphi &= f \quad \text{in } Q_T, \\
			\dn\varphi|_{\partial\Omega} &= 0 \quad \text{on } (0, T) \times \partial\Omega, \\
			\dn{\Delta \varphi}|_{\partial\Omega} &= 0 \quad \text{on } (0, T) \times \partial\Omega, \\
			\varphi(0) &= \varphi_0 \quad \text{in } \Omega.
		\end{aligned}
	\end{equation}
	Moreover, we have the uniform bound:
	\begin{equation}\label{einv-l2}
		\norm{\mathcal{L}_{2,T}^{-1}}_{\mathcal{L}(Y_T^2\times \mathbb{Z}_{0}^2, Z_T^2)} \leq \norm{\mathcal{L}_{2,T_0}^{-1}}_{\mathcal{L}(Y_{T_0}^2\times \mathbb{Z}_{0}^2, Z_{T_0}^2)} < \infty \quad \text{for all } 0 < T < T_0.
	\end{equation}
\end{lemma}

\begin{remark}
	The \rev{proof of} Lemmas \ref{inv-l2} is established in \cite{AbelsWeber2021}. Consequently, we obtain the continuity and bounded invertibility of $\mathcal{L}_{1,T}$ and $\mathcal{L}_{2,T}$. Noting that the last equation of system \eqref{mod1} is a coupled equation for the variables $\varphi$ and $q$, we will subsequently apply the following general result for linear parabolic equations to establish the invertibility of this coupled operator and obtain a uniform upper bound for its inverse $\mathcal{L}_{3,T}^{-1}$.
\end{remark}

\begin{lemma}\label{exi-q}
	Let $q_0 \in W^{2-\frac{2}{p}}_{p, N}(\Omega)$ with $1 < p < \infty$. For any $l\in L^p(0,T; L^p(\Omega))$ and $m\in C^0(\overline{\Omega})$ with $m(x)>0$ for all $x\in \overline{\Omega}$, there exists a unique solution 
	$$
	q \in L^p(0,T; W^2_{p,N}(\Omega) ) \cap \left\{ v \in W^{1,p}(0,T; L^p(\Omega)) \mid v|_{t=0}=q_0 \right\}
	$$
	to the following parabolic system:
	\begin{equation}\label{qqq}
		\begin{cases}
			\dt q - \tilde{m}(x) \Delta q &= l  \quad \text{in } (0,T) \times \Omega, \\
			\partial_n q &= 0  \quad\text{on } (0,T) \times \partial\Omega, \\
			q|_{t=0} &= q_0  \quad\text{in } \Omega.
		\end{cases}	
	\end{equation}
\end{lemma}

\begin{remark}\label{rem:1.5}
	By the continuity of $\beta:=\frac{1}{\varepsilon}f'(q_0)W(\varphi_0)+g'(q_0)$, we deduce thet there exists a constant $m_0>0$ such that $\beta \geq m_0 > 0$, which implies $\beta^{-1}$ is well-defined and uniformly bounded in $\Omega$. Let $\tilde{m}=\beta^{-1}\tilde{m}(\varphi_0,q_0)$. Since $\tilde{m}(\varphi_0,q_0)\in C^0(\overline{\Omega})$, it directly follows that $\tilde{m}=\beta^{-1}\tilde{m}(\varphi_0,q_0)\in C^0(\overline{\Omega})$.
\end{remark}	

\begin{proof}[Proof of Lemma \ref{exi-q}]
	We apply the maximal regularity framework for parabolic boundary value problems as in \cite[Theorem 8.2]{denk2003r}. Here, the spatial domain $\Omega \subset \mathbb{R}^d$ is a bounded domain with a $C^2$-boundary $\partial\Omega$.
	
\noindent\textbf{(i) Parameter-Ellipticity of the Operator.} First, we verify that the principal symbol of the operator $-\tilde{m}(x)\Delta$, given by
	\begin{equation}\label{principal_symbol}
		\mathcal{A}_\#(x,\xi) = \tilde{m}(x)|\xi|^2,
	\end{equation}
	is parameter-elliptic. Note that the symbol depends only on $x \in \Omega$ and not on $t$. Since $\tilde{m}(x) > 0$ on $\overline{\Omega}$, the spectrum trivially satisfies $\sigma(\mathcal{A}_\#(x,\xi)) \subset \mathbb{R}_+$ for all $x \in \overline{\Omega}$ and $\xi \in \mathbb{R}^d$ with $|\xi|=1$. Thus, the angle of ellipticity is exactly $\phi_{\mathcal{A}} = 0$, meaning the operator is parameter-elliptic for any angle $\phi > 0$.
	
\noindent\textbf{(ii) Boundary Condition Handling.} Since $\Omega$ has a $C^2$-boundary, for any fixed $x_0 \in \partial\Omega$, there exists a local $C^2$-diffeomorphism straightening $\partial\Omega$ near $x_0$ into the half-space
	\[
		\mathbb{R}^d_+ = \{y = (x', y_d) \mid x' \in \mathbb{R}^{d-1}, \, y_d > 0\},
	\]
	where $x'$ denotes the tangential coordinates and $y_d$ is the normal coordinate. Let $u(y)$ denote the localized function in these new coordinates. Then the homogeneous Neumann boundary condition becomes $\partial_{y_d} u(x', 0) = 0$.
	
	Next, we formulate the corresponding ODE problem on the half-line. Performing the Fourier transform with respect to the tangential coordinates $x'$ (with dual variable $\xi'$), the principal symbol becomes $\mathcal{A}_\#(x_0, \xi', \partial_{y_d}) = \tilde{m}(x_0)|\xi'|^2 - \tilde{m}(x_0)\partial_{y_d}^2$. This yields the following ODE for $y_d > 0$:
	\[
		\big(\lambda + \tilde{m}(x_0)|\xi'|^2 - \tilde{m}(x_0)\partial_{y_d}^2\big) \hat{u}(\xi', y_d) = 0.
	\]
	Solving the characteristic equation for this ODE, we obtain a decaying stable solution $\hat{u}(y_d) = A e^{-r y_d}$, with the root
	\[
		r = \left(\frac{\lambda}{\tilde{m}(x_0)} + |\xi'|^2\right)^{1/2}.
	\]
	Applying the Neumann boundary condition at $y_d = 0$, we have
	\[
		\partial_{y_d} \hat{u}(0) = -rA = 0.
	\]
	From part (i), since $(\lambda, \xi') \neq (0,0)$ and the parameter $\lambda \in \overline{\Sigma_\phi}$, it  follows that $r \neq 0$. This forces the constant $A = 0$ and completes the verification of the Lopatinskii-Shapiro condition. Since our initial assumptions strictly guarantee the requisite coefficient regularity $\tilde{m} \in C^0(\overline{\Omega})$ and data regularity $l \in L^p(0,T; L^p(\Omega))$, all prerequisites of the maximal $L^p$-regularity result \cite[Theorem 8.2]{denk2003r} are fulfilled, thereby establishing the existence of a unique strong solution $q$. 
\end{proof}

The following steps establish the continuity of the linear operators in the coupled parabolic system:
\begin{equation}\label{couple}
	\begin{cases} 
		\dt \varphi +\varepsilon m(\varphi_0) \Lap^2 \varphi = f_1 & \text{in } (0,T) \times \Omega, \\
		\beta \dt q - \tilde{m}(\varphi_0, q_0) \Lap q + \gamma \dt \varphi = f_2 & \text{in } (0,T) \times \Omega, \\
		\dn \varphi|_{\partial\Omega} = 0, \quad \dn q|_{\partial\Omega} = 0 & \text{on } (0,T) \times \partial\Omega, \\
		\dn \Lap \varphi|_{\partial\Omega} = 0 & \text{on } (0,T) \times \partial\Omega, \\
		\varphi|_{t=0} = \varphi_0, \quad  q|_{t=0} = q_0 & \text{in } \Omega,
	\end{cases}
\end{equation} 
where $\beta, \gamma \in C^0(\overline{\Omega})$ with $\beta(x) > 0$.

First, we rewrite system \eqref{couple} in an exact operator form. Let the foundational Banach spaces be $Z_T = Z_T^2 \times Z_T^3$ and $\mathbb{Y}_T = (Y_T^2 \times \mathbb{Z}_0^2) \times (Y_T^3 \times \mathbb{Z}_0^3)$. We define the coupled linear operator $\mathcal{L}_{*,T}: Z_T \to \mathbb{Y}_T$ as:
\begin{equation}\label{matrix-couple}
	\mathcal{L}_{*,T}(\varphi, q) = 
	\begin{pmatrix} 
		\mathcal{L}_{2,T}(\varphi) \\ 
		\mathcal{L}_{q,T}(q) + (\gamma \dt \varphi, 0) 
	\end{pmatrix}
	= \begin{pmatrix} 
		(f_1, \varphi_0)\\ 
		(f_2, q_0) 
	\end{pmatrix},
\end{equation}
where $\mathcal{L}_{q,T}(q) := \left( \beta \dt q - m(\varphi_0, q_0) \Lap q, \, q|_{t=0} \right)$.

\begin{lemma}\label{inv-l3}
	Let the assumptions of Lemmas \ref{inv-l2} and \ref{exi-q} hold. Given $0 < T_0 < \infty$, the inverse operator $\mathcal{L}_{*,T}^{-1}$ exists and is uniformly bounded for all $0 < T \le T_0$. Specifically, there exists a constant $C(T_0) > 0$ independent of $T$ such that:
	\begin{equation}\label{einv-l3}	\|\mathcal{L}_{*,T}^{-1}\|_{\mathcal{L}(\mathbb{Y}_T, Z_T)} \leq C(T_0) < \infty \quad \text{for all } 0 < T \le T_0.
	\end{equation}
\end{lemma}

\begin{proof}
	By Lemmas \ref{inv-l2} and \ref{exi-q}, the block-diagonal counterparts $\mathcal{L}_{2,T}$ and $\mathcal{L}_{q,T}$ are both continuously invertible. Consequently, system \eqref{matrix-couple} admits a unique strong solution $(\varphi, q) \in Z_T^2 \times Z_T^3$ for any given right-hand side data in $\mathbb{Y}_T$.
	
	To explicitly establish the uniform bound with respect to $T$, we exploit the lower-triangular block structure of \eqref{matrix-couple}, which sequentially decouples the system. From the first row, we immediately obtain $\varphi = \mathcal{L}_{2,T}^{-1}(f_1, \varphi_0)$. Substituting this phase-field solution into the second row yields:
	\begin{equation*}
		\mathcal{L}_{q,T}(q) = (f_2 - \gamma \dt \varphi, q_0).
	\end{equation*}
	Since the uniform maximal regularity for $\mathcal{L}_{2,T}$ is already established in Lemma \ref{inv-l2}, we inherently possess the uniform estimate $\|\varphi\|_{Z_T^2} \leq C_{T_0} \|(f_1, \varphi_0)\|_{Y_T^2 \times \mathbb{Z}_0^2}$. Thus, we only need to rigorously derive the uniform boundedness of the coupled component $q$. 
	
	Let us define source term on the fixed time interval $(0, T_0)$:
	\begin{equation*}
		\widetilde{f}_0(t) = 
		\begin{cases} 
			f_2(t) - \gamma \dt \varphi(t) & \text{if } t \in (0, T], \\
			0 & \text{if } t \in (T, T_0).
		\end{cases}
	\end{equation*}
	Because $f_2, \dt \varphi \in L^p(0,T; L^p(\Omega))$, the trivial extension ensures $\widetilde{f}_0 \in L^p(0,T_0; L^p(\Omega))$, and their norms satisfy $\|\widetilde{f}_0\|_{Y_{T_0}^3} = \|f_2 - \gamma \dt \varphi\|_{Y_T^3}$.
	
	According to Lemma \ref{exi-q}, there exists a unique solution $\tilde{q} \in Z_{T_0}^3$ to the auxiliary parabolic problem $\mathcal{L}_{q,T_0}(\tilde{q}) = (\widetilde{f}_0, q_0)$. Due to the causality and uniqueness of the initial-boundary value problem, the restriction of $\tilde{q}$ to $(0,T)$ exactly coincides with $q$, i.e., $\tilde{q}|_{(0, T)} = q$.
	
	Employing the uniform bound of $\mathcal{L}_{q,T_0}^{-1}$ from Lemma \ref{exi-q}, and standard triangle inequalities, we deduce:
	\begin{equation*}
		\begin{aligned}
			\| q \|_{Z_T^3} &\leq \| \tilde{q} \|_{Z_{T_0}^3} = \left\| \mathcal{L}_{q,T_0}^{-1} (\widetilde{f}_0 ,q_0)\right\|_{Z_{T_0}^3} \\
			&\leq \left\| \mathcal{L}_{q,T_0}^{-1} \right\|_{\mathcal{L}(Y_{T_0}^3\times \mathbb{Z}_{0}^3, Z_{T_0}^3)} \left( \| \widetilde{f}_0 \|_{Y_{T_0}^3} + \| q_0 \|_{\mathbb{Z}_0^3} \right) \\
			&\leq C(T_0) \left( \| f_2 - \gamma \dt \varphi \|_{Y_T^3} + \| q_0 \|_{\mathbb{Z}_0^3} \right) \\
			&\leq C(T_0) \left( \| f_2 \|_{Y_T^3} + \| \gamma \|_{L^\infty(\Omega)} \| \dt \varphi \|_{Y_T^2} + \| q_0 \|_{\mathbb{Z}_0^3} \right).
		\end{aligned}
	\end{equation*}
	Crucially, since $\| \dt \varphi \|_{Y_T^2} \leq \| \varphi \|_{Z_T^2}$, we can feed the estimate of $\varphi$ into the inequality to obtain:
	\begin{equation*}
		\| q \|_{Z_T^3} \leq \widetilde{C}(T_0) \left( \| f_1 \|_{Y_T^2} + \| \varphi_0 \|_{\mathbb{Z}_0^2} + \| f_2 \|_{Y_T^3} + \| q_0 \|_{\mathbb{Z}_0^3} \right).
	\end{equation*}
	This uniform estimate bounds both $\varphi$ and $q$ by the initial configurations and the sources on the right-hand side. It shows that the fully coupled inverse operator $\mathcal{L}_{*,T}^{-1}$ is uniformly bounded by a constant dependent only on $T_0$, thus completing the proof.
\end{proof}

\section*{Acknowledgements}	
	 SZ. Chen acknowledges support by the China Scholarship Council (CSC) through grant 202508340071

         \bibliographystyle{abbrv}
\bibliography{references}
	
\end{document}